\documentclass[a4paper,11pt]{amsart}
\usepackage{amssymb,amsmath,amscd,tikz,url}
\usepackage[all]{xy}
\usepackage{hyperref}

\newtheorem{thm}{Theorem}[section]
\newtheorem{conjecture}[thm]{Conjecture}
\newtheorem{cor}[thm]{Corollary}
\newtheorem{lem}[thm]{Lemma}
\newtheorem{prop}[thm]{Proposition}

\newtheorem{rem}[thm]{Remark}
\newtheorem{notation}[thm]{Notation}
\newtheorem{ex}[thm]{Example}

\newtheorem{con}[thm]{Condition}

\newcommand{\Zset}{\mathbb{Z}}
\newcommand{\Rset}{\mathbb{R}}

\newcommand{\Cset}{\mathbb{C}}

\def\red{{\rm red}}
\def\ss{{\rm ss}}
\def\top{{\rm top}}

\newcommand{\bc}{\bf c}

\newcommand{\ab}{{\rm ab}}
\def\en{{\rm en}}

\newcommand{\fss}{{\rm fss}}

\def\fo{{\mathfrak o}}

\def\PPhi{{\Phi}}

\def\rG{{{\rm G}}}

\def\ch{{\rm ch}}
\def\cB{{\mathcal B}}

\def\cP{{\mathcal P}}

\def\Cent{{\rm Z}}
\def\Z{{\rm Z}}
\def\Nor{{\rm N}}

\def\Rep{{\rm Rep}}
\def\ie{{\it i.e.,\,}}

\def\cJ{{\mathcal J}}

\def\im{{\rm im\,}}

\def\H{{\rm H}}

\def\KLR{{\rm KLR}}
\def\SL{{\rm SL}}

\def\SO{{\rm SO}}

\def\GL{{\rm GL}}

\def\PGL{{\rm PGL}}

\def\der{{\rm der}}
\def\Aut{{\rm Aut}}
\def\Hom{{\rm Hom}}
\def\End{{\rm End}}

\def\top{{\rm top}}
\def\triv{{\rm triv}}
\def\unr{{\rm unr}}

\def\bI{{\mathbf I}}
\def\bW{{\mathbf W}}

\def\Lie{{\rm Lie}}

\def\der{{\rm der}}

\def\temp{{\rm temp}}

\def\Irr{{\mathbf {Irr}}}
\def\Mod{{\rm Mod}}
\def\Lie{{\rm Lie}}

\def\cG{{\mathcal G}}

\def\cH{{\mathcal H}}

\def\cM{{\mathcal M}}
\def\cO{{\mathcal O}}

\def\cT{{\mathcal T}}
\def\cU{{\mathcal U}}

\def\cW{{\mathcal W}}

\def\fs{{\mathfrak s}}

\def\fB{{\mathfrak B}}

\def\fB{{\mathfrak B}}

\def\Ad{{\rm Ad}}
\def\triv{{\rm triv}}

\def\q{{/\!/}}

\def\ab{{\rm ab}}
\def\der{{\rm der}}

\def\PPhi_F{{\rm Frob}}
\def\Frob{{\rm Frob}}
\def\Flag{{\rm Flag}}

\def\cpt{{\rm cpt}}
\def\square{{\natural}}

\newcommand{\Q}{\mathbb Q}
\newcommand{\R}{\mathbb R}
\newcommand{\C}{\mathbb C}
\newcommand{\matje}[4]{\left(\begin{smallmatrix} #1 & #2 \\ 
#3 & #4 \end{smallmatrix}\right)}
\def\lexp#1#2{{\kern\scriptspace\vphantom{#2}^{#1}\kern-\scriptspace#2}}

\begin{document}
\title{The principal series of $p$-adic groups with disconnected centre}   


\author[A.-M. Aubert]{Anne-Marie Aubert}
\address{IMJ-PRG, U.M.R. 7586 du C.N.R.S., U.P.M.C., Paris, France}
\email{anne-marie.aubert@imj-prg.fr}
\author[P. Baum]{Paul Baum}
\address{Mathematics Department, Pennsylvania State University,  University Park, PA 16802, USA}
\email{baum@math.psu.edu}
\thanks{The second author was partially supported by NSF grant DMS-0701184}
\author[R. Plymen]{Roger Plymen}
\address{School of Mathematics, Southampton University, Southampton SO17 1BJ,  England \emph{and} 
School of Mathematics, Manchester University, Manchester M13 9PL, England}
\email{r.j.plymen@soton.ac.uk \quad plymen@manchester.ac.uk}
\author[M. Solleveld]{Maarten Solleveld}
\address{Radboud Universiteit Nijmegen, Heyendaalseweg 135, 6525AJ Nijmegen, the Netherlands}
\email{m.solleveld@science.ru.nl}

\date{\today}
\maketitle

\begin{abstract}
Let $\cG$ be a split connected reductive group over a local non-archimedean field.
We classify all irreducible complex $\cG$-representations in the principal series, 
irrespective of the (dis)connectedness of the centre of $\cG$. This leads to a local 
Langlands correspondence for principal series representations, which satisfies all 
expected properties. We also prove that the ABPS conjecture about the geometric 
structure of Bernstein components is valid throughout the principal series of $\cG$.
\end{abstract}

\tableofcontents

\section{Introduction}
\label{sec:intro}

In this paper, we construct (granted a mild restriction on the residual characteristic) 
a local Langlands correspondence throughout the principal series of any connected split 
reductive $p$-adic group $\cG$.   In addition, we prove that the ABPS geometric structure 
conjecture is valid throughout the principal series of $\cG$.

We do not assume that the centre of $\cG$ is connected.   Previous results on this subject 
needed the condition that the centre of $\cG$ is connected, see Kazhdan--Lusztig \cite{KL}, 
Reeder \cite{R}, Aubert--Baum--Plymen--Solleveld \cite{ABPS2}.   

Let $F$ be a local non-Archimedean field and let $\cG$ be the group of the $F$-rational points of 
an $F$-split connected reductive algebraic group, and
let $\mathcal T$ be a maximal torus in $\cG$. The principal series consists 
of all $\mathcal G$-representations that are constituents of parabolically induced representations 
from irreducible characters of $\mathcal T$.   Let $\fB(\cG)$ denote the Bernstein spectrum of $\cG$, 
and let $\fB(\cG, \cT)$ be the subset of $\fB(\cG)$ given by all cuspidal pairs $(\cT, \chi)$, 
where $\chi$ is a character of $\cT$.     

For each $\fs\in\mathfrak{B}(\cG,\cT)$ we construct a commutative triangle of bijections 
\begin{equation}\label{eq:introTriangle}
\xymatrix{   
& (T^\fs/\!/W^\fs)_2 \ar[dr]\ar[dl] & \\ 
\Irr (\cG)^\fs \ar[rr] & & \Psi(G)_{\en}^{\fs} 
}
\end{equation}
Here $T^\fs$ and $W^\fs$ are Bernstein's torus and finite group for $\fs$, and 
$(T^\fs \q W^\fs)_2$ is the extended quotient of the second kind resulting from the action of
$W^\fs$ on $T^\fs$.   Equivalently, $(T^\fs \q W^\fs)_2$ is the set of equivalence classes 
of irreducible representations of the crossed product algebra
$\cO(T^\fs) \rtimes W^\fs$:
\[
(T^\fs \q W^\fs)_2 \simeq \Irr(\cO(T^\fs) \rtimes W^\fs).
\]
Here $\Irr(\cG)^\fs$ is the Bernstein component of $\Irr(\cG)$ attached to $\fs \in \fB(\cG,\cT)$, 
$\Psi(G)^{\fs}_{\en}$ is the set of enhanced 
Langlands parameters associated to $\fs$, and  $G$ is the Langlands dual of $\cG$.   

In examples, $(T^\fs /\!/ W^\fs)_2$ is much simpler to directly calculate than either 
$\Irr (G)^\fs$ or $\Psi(G)_{\en}^{\fs}$.

The point $\fs \in \fB(\cG, \cT)$ determines a certain complex reductive group $H^\fs$ 
in the Langlands dual group $G$.   If $\cG$ has connected centre, then:
\begin{itemize}
\item $H^\fs$ is connected
\item Bernstein's finite group $W^\fs$ is the Weyl group of $H^\fs$
\item Bernstein's torus $T^\fs$ is the maximal torus of $H^\fs$
\item the action of $W^\fs$ on $T^\fs$ is the standard action of the Weyl group of 
$H^\fs$ on the maximal torus of $H^\fs$.
\end{itemize}   

If $\cG$ does not have connected centre, then: 
\begin{itemize}
\item $H^\fs$ can be non-connected
\item $W^\fs$ is the semidirect product
$W^\fs = \cW^{H_0^\fs} \rtimes \pi_0(H^\fs)$
where $H_0^\fs$ is the identity component of $H^\fs$, and $\cW^{H_0^\fs}$ 
is the Weyl group of $H_0^\fs$
\item $T^\fs$ is the maximal torus $T$ of $H_0^\fs$ 
\item $W^\fs = \Nor_{H^\fs}(T)/T$.   The action on $T$ is the evident conjugation action, 
and $\Nor_{H^\fs}(T)$ is the normalizer in $H^\fs$ of $T$.
\end{itemize}  
See Lemma \ref{lem:centrals} and Eqn. \eqref{eq:WsH}.    

Semidirect products by $\pi_0(H^\fs)$ occur frequently in this paper, e.g.
\[
\cH^\fs = \cH(H_0^\fs) \rtimes \pi_0(H^\fs)
\]
Here $\cH^\fs$ is a finite type algebra attached by Bernstein to $\fs$ and $\cH(H_0^\fs)$ 
is the affine Hecke algebra 
of $H_0^\fs$, with parameter $q$   equal to the cardinality of the residue field.   
Thus, $\cH^\fs$ is an extended affine Hecke algebra.

Similarly, $\pi_0(H^\fs)$ acts on Lusztig's asymptotic algebra $\mathcal{J}(H_0^\fs)$.   
The crossed product algebra 
\[
\mathcal{J}(H_0^\fs) \rtimes \pi_0(H^\fs)
\]
features crucially in Section \ref{sec:mainH}.

In the above commutative triangle, the right slanted arrow is constructed and proved 
to be a natural bijection by suitably generalising the Springer correspondence for 
finite and affine Weyl groups (Sections \ref{sec:celldec} and \ref{sec:affSpringer}),
and by comparing the involved parameters (Sections \ref{sec:Borel} and \ref{sec:comppar}).

The left slanted arrow in \eqref{eq:introTriangle} is defined and proved to be 
a bijection by applying the representation theory
of affine Hecke algebras and, in particular, Lusztig's asymptotic algebra.   
However, in order to apply this theory, it is necessary 
to prove the equality of certain $2$-cocycles, see \S \ref{sec:mainH}.  
The technical issues that are confronted in this paper arise from Clifford theory 
and are very closely connected to the analysis of these $2$-cocycles.

Similar 2-cocycles for connected non-split groups can be non-trivial.   
Hence, for connected non-split groups, a twisted extended quotient must be used in the 
statement of the ABPS geometric structure conjecture.  The ABPS conjecture for connected 
non-split reductive $p$-adic groups is developed in \cite{ABPS7}.

The horizontal arrow in our main result (see the above commutative triangle and 
Theorem \ref{thm:main} and Proposition \ref{prop:LLC}) generalises the Kazhdan--Lusztig 
parametrization of the irreducible representations of affine Hecke algebras with equal
parameters (\S \ref{sec:repAHA}), and also generalises the Reeder--Roche 
parametrization of the irreducible $\cG$-representations in the principal series for 
groups $\cG$ with connected centre (cf. \S \ref{sec:Roche}). We note that most of the
representations considered by Roche--Reeder have positive depth.

We use the new
input from $( T^\fs /\!/ W^\fs )_2$ to prove that, although the horizontal arrow in
\eqref{eq:introTriangle} is in general not canonical, every element of 
$\Irr (\cG)^\fs$ does canonically determine a Langlands parameter for $\cG$.
To establish the horizontal arrow as a local Langlands correspondence for these
representations, we also show that it satisfies all the desiderata of Borel, 
see Sections \ref{sec:LLC} and \ref{sec:func}. Thus we prove the local Langlands 
conjectures for a class of representations which contains elements of 
arbitrarily high depth.

The union over all the $\fs\in\fB (\cG, \cT)$ of
the extended quotients of the second kind $( T^\fs /\!/ W^\fs )_2$ 
is the extended quotient of the second kind
$(\Irr (\mathcal T) /\!/ \mathcal W^\cG )_2$, with $\mathcal W^\cG=\Nor_{\cG}(\cT)/\cT$, 
and the triangles \eqref{eq:introTriangle} for different $\fs$ combine to a 
bijective commutative diagram
\[ 
\xymatrix{   & (\Irr (\mathcal T) /\!/ \mathcal W^\cG )_2 \ar[dr]\ar[dl] & \\  
\Irr(\cG ,\cT)    \ar[rr] & & 
\Psi (G)_\en^{\text{prin}} 
}
\]
where $\Psi (G)_\en^{\text{prin}}$ denotes the collection of enhanced L-parameters
for the principal series of $\cG$, and $\Irr(\cG ,\cT)$ denotes the collection of 
irreducible principal series representations of $\cG$. All this holds under the 
restrictions on the residual characteristic stated in Condition \ref{con:char}.

The ABPS geometric structure conjecture includes the assertion  that the local Langlands 
correspondence factors through the appropriate extended quotient.   
This extended quotient is much easier to directly calculate than either the source or 
the target of the Langlands correspondence.   

In fact, the ABPS conjecture has more precision, as explained
in \S \ref{sec:unip} of this paper.    With this level of precision, we provide a complete proof, 
in \S \ref{sec:unip} and \S \ref{sec:cochar},  of the conjecture for the principal series  
of all connected split  reductive $p$-adic groups (always
with the restriction on the residue characteristic).   This includes, in \S \ref{sec:unip},  
a labelling by unipotent classes in $H^\fs$ and, in \S \ref{sec:cochar}, a complete account 
of the relation between correcting cocharacters and $L$-packets.

\section{Extended quotients}  
\label{sec:extquot}

Let $\Gamma$ be a finite group acting on a topological space $X$ 
\[ 
\Gamma \times X \to X.
\]
The quotient space $X/\Gamma$ is obtained by collapsing each orbit to a point. 
For $x\in X $, $\Gamma_x$ denotes the stabilizer group of $x$:
$$ 
\Gamma_x = \{\gamma\in \Gamma : \gamma x = x\}.
$$
$c(\Gamma_x)$ denotes the set of conjugacy classes of  $\Gamma_x$. 
The \emph{extended quotient of the first kind}
is obtained by replacing the orbit of $x$ by $c(\Gamma_x)$. This is done as follows:\\

\noindent Set $\widetilde{X} = \{(\gamma, x) \in \Gamma \times X : \gamma x = x\}$, a 
subspace of $\Gamma \times X$. The group $\Gamma$ acts on $\widetilde{X}$:
\begin{align*}
\Gamma &\times \widetilde{X} \to \widetilde{X}\\
\alpha(\gamma, x) = & (\alpha\gamma \alpha^{-1}, \alpha x), \quad\quad \alpha \in \Gamma,
\quad (\gamma, x) \in \widetilde{X}.
\end{align*}

\noindent The extended quotient, denoted $ X/\!/\Gamma $,  is $\widetilde{X}/\Gamma$. 
Thus the extended quotient  $ X/\!/\Gamma $ is the usual quotient for the action of 
$\Gamma$ on $\widetilde{X}$. 
The projection $\widetilde{X} \to X$,  $(\gamma, x) \mapsto x$ is $\Gamma$-equivariant 
and so passes to quotient spaces 
\[
\rho\colon  X/\!/\Gamma \to X/\Gamma.
\]   
This map will be referred to as the projection of the extended quotient onto the ordinary quotient.\\

\noindent The inclusion 
\begin{align*}
X&\hookrightarrow \widetilde{X}\\
x&\mapsto (e,x)\qquad e=\text{identity element of }\Gamma
\end{align*}
is $\Gamma$-equivariant and so passes to quotient spaces to give an inclusion 
$X/\Gamma\hookrightarrow X/\!/\Gamma$. This will be referred to as the inclusion 
of the ordinary quotient in the extended quotient.

With $\Gamma ,\: X ,\: \Gamma_x$ as above, let $\Irr (\Gamma_x)$ 
be the set of (equivalence classes of) irreducible representations of $\Gamma_x$. 
The \emph{extended quotient of the second kind}, 
denoted $(X/\!/\Gamma)_2$, is constructed by replacing the orbit of x (for the given action of 
$\Gamma$ on $X$) by $\Irr(\Gamma_x)$. This is done as follows : \\

\noindent Set $\widetilde{X}_2 = \{(x, \tau)\thinspace \big|\thinspace x \in X$ and
$\tau \in \Irr(\Gamma_x)\}$. Then $\Gamma$ acts on $\widetilde{X}_2.$
\begin{align*}
& \Gamma \times \widetilde{X}_2 \to \widetilde{X}_2 , \\
& \gamma(x, \tau) = (\gamma x, \gamma_*\tau) ,
\end{align*}
where $\gamma_*\colon \Irr(\Gamma_x)\rightarrow \Irr(\Gamma_{\gamma x})$.
Now we define
\[
(X/\!/\Gamma)_2  := \:\widetilde{X}_2/\Gamma , 
\]
i.e. $ (X/\!/\Gamma)_2 $ is the usual quotient for the action of $\Gamma$ on $\widetilde{X}_2$.
The projection $\widetilde{X}_2\rightarrow  X$\quad $(x, \tau) \mapsto x$\quad is 
$\Gamma$-equivariant  and so passes to quotient spaces to give the projection of 
$(X/\!/\Gamma)_2$ onto $X/\Gamma$.
\[
\rho_2 \colon (X/\!/\Gamma)_2 \longrightarrow X/\Gamma  
\]
Denote by $\mathrm{triv}_x$ the trivial one-dimensional representation of $\Gamma_{x}$.
\noindent The inclusion 
\begin{align*}
X&\hookrightarrow \widetilde{X}_2\\
x&\mapsto (x, \mathrm{triv}_x)
\end{align*}
is $\Gamma$-equivariant and so passes to quotient spaces to give an inclusion
\[
X/\Gamma\hookrightarrow (X/\!/\Gamma)_2
\] 
This will be referred to as the inclusion of the ordinary quotient in the extended quotient
of the second kind.  

Notice that the fibers $\rho^{-1}(\Gamma x)$ and $\rho_2^{-1}(\Gamma x)$ always have the
same number of elements. Hence there exist non-canonical bijections 
$\epsilon \colon X/\!/\Gamma\rightarrow (X/\!/\Gamma)_2$ with commutativity in the diagrams
\begin{equation}\label{eq:diagramsExtquots} 
\xymatrix{
X/\!/\Gamma \ar[dr]_{\rho} \ar[rr]^{\epsilon}  &  &  (X/\!/\Gamma)_2 \ar[dl]^{\rho_2}  &
X/\!/\Gamma  \ar[rr]^{\epsilon}  &  &  (X/\!/\Gamma)_2  \\
& X/\Gamma & & & X/\Gamma \ar[ul] \ar[ur] & }
\end{equation}
To construct a bijection $\epsilon$, some choices must be made. 
We will make use of a family $\psi$ of bijections
\[
\psi_x : c (\Gamma_x) \to \Irr (\Gamma_x)
\]
such that for all $x \in X$:
\begin{itemize}
\item $\psi_x ([1]) = \mathrm{triv}_x$;
\item $\psi_{\gamma x} ([\gamma g \gamma^{-1}]) = \phi_x ([g]) \circ \mathrm{Ad}_\gamma^{-1}$
for all $g \in \Gamma_x, \gamma \in \Gamma$.
\end{itemize}
We shall refer to such a family of bijections as a $c$-$\Irr$ system. Clearly $\psi$
induces a map $\widetilde X \to {\widetilde X}_2$ which preserves the $X$-coordinates.
By the second property this map is $\Gamma$-equivariant, so it descends to a map
\[
\epsilon = \epsilon_\psi : X /\!/ \Gamma \to (X /\!/ \Gamma )_2 .
\]
Observe that $\epsilon_\psi$ makes the diagrams from \eqref{eq:diagramsExtquots} commute,
the first by construction and the second by the first property of $\psi_x$.
The restriction of $\epsilon_\psi$ to the fiber over $\Gamma x \in X / \Gamma$
is $\psi_x$, and in particular is bijective. Therefore $\epsilon_\psi$ is bijective.

Next we will define a twisted version of an extended quotient.
Let $\natural$ be a given function which assigns to each 
$x \in X$ a 2-cocycle $\natural(x) : \Gamma_x \times \Gamma_x \to \C^\times$ 
where $\Gamma_x = \{\gamma \in \Gamma : \gamma x = x\}$.   
It is assumed that $\natural(\gamma x)$ and $\gamma_*\natural(x)$ define the same
class in $H^2 (\Gamma_x , \C^\times)$,
where $\gamma_* : \Gamma_x \to \Gamma_{\gamma x}, \alpha \mapsto \gamma \alpha \gamma^{-1}$.    
Define 
\[
\widetilde X_2^\natural : = \{(x,\rho) : x \in X, \rho \in \Irr \,\C[\Gamma_x, \natural(x)] \}.
\]
We  require, for every $(\gamma,x) \in \Gamma \times X$, a definite  algebra isomorphism
\[
\phi_{\gamma,x} : \C[\Gamma_x,\natural(x)]  \to \C[\Gamma_{\gamma x},\natural (\gamma x)]
\]
such that:
\begin{itemize}
\item $\phi_{\gamma,x}$ is inner if $\gamma x = x$;
\item $\phi_{\gamma',\gamma x} \circ \phi_{\gamma,x} = 
\phi_{\gamma' \gamma,x}$ for all $\gamma',\gamma \in \Gamma, x \in X$.
\end{itemize}
We call these maps connecting homomorphisms, because they are reminiscent of a connection
on a vector bundle.
Then we can define $\Gamma$-action on $\widetilde X_2^\natural$ by
\[
\gamma \cdot (x,\rho) = (\gamma x, \rho \circ \phi_{\gamma,x}^{-1}).
\]
We form the \emph{twisted extended quotient}
\[
(X\q \Gamma)_2^\natural : = \widetilde{X}_2^\natural/\Gamma.
\]
Notice that this reduces to the extended quotient of the second kind
if $\natural (x)$ is trivial for all $x \in X$.
We will apply this construction in the following two special cases.  

\smallskip

{\bf 1.} Given two finite groups $\Gamma_1$, $\Gamma$ and a group homomorphism
$\Gamma \to \Aut(\Gamma_1)$, we can form the semidirect product  $\Gamma_1 \rtimes\Gamma$.  
Let $X = \Irr \, \Gamma_1$.   Now $\Gamma$ acts on $\Irr\, \Gamma_1$ and we get $\square$ 
as follows. Given $x \in \Irr\, \Gamma_1$ choose an irreducible representation  
$\pi_x: \Gamma_1 \to \GL(V)$ whose isomorphism class is $x$. 
For each $\gamma \in \Gamma$ consider $\pi_x$ twisted by $\gamma$ \ie consider 
$\gamma \cdot \pi_x : \gamma_1 \mapsto \pi_x (\gamma^{-1} \gamma_1 \gamma)$.
Since $\gamma \cdot \pi_x$ is equivalent to $\pi_{\gamma x}$, there exists 
a nonzero intertwining operator 
\[
T_{\gamma,x} \in \Hom_{\Gamma_x} (\gamma \cdot \pi_x , \pi_{\gamma x}) .
\]
By Schur's lemma it is unique up to scalars, but in general there is no preferred choice.
For $\gamma, \gamma' \in \Gamma_x$ there exists a unique $c \in \C^\times$ such that
\[
T_{\gamma,x} \circ T_{\gamma',x} = c T_{\gamma \gamma',x}.
\]
We define the 2-cocycle by $\square (x) (\gamma,\gamma') = c$. Let $N_{\gamma,x}$
with $\gamma \in \Gamma_x$ be the standard basis of $\C [\Gamma_x,\natural(x)]$.
The algebra homomorphism $\phi_{g,x}$ is essentially conjugation by $T_{g,x}$,
but the precise definition is 
\begin{equation}\label{eq:twisting}
\phi_{g,x} (N_{\gamma,x}) = \lambda N_{g \gamma g^{-1},g x} \quad \text{if} \quad
T_{g,x} T_{\gamma,x} T_{g,x}^{-1} = \lambda T_{g \gamma g^{-1}, g x}, \lambda \in \C^\times .
\end{equation}
Notice that \eqref{eq:twisting} does not depend on the choice of $T_{g,x}$.
This leads to a new formulation  of a classical theorem of Clifford.

\begin{lem} 
\label{lem:Clifford} 
There is a bijection
\[
\Irr (\Gamma_1 \rtimes \Gamma) \longleftrightarrow (\Irr \, \Gamma_1 \q \Gamma)_2^{\square}.
\]
\end{lem}
\begin{proof}  
The proof proceeds by comparing our construction with the classical theory of Clifford; 
for an exposition of Clifford theory, see \cite{RamRam}. 
\end{proof}

The above bijection is in general not canonical, it depends on the choice of the
intertwining operators $T_{\gamma,x}$. 

\begin{lem} \label{lem:Clifford_abelian}
If $\Gamma_1$ is abelian, then we have a natural bijection
\[
\Irr (\Gamma_1 \rtimes \Gamma) \longleftrightarrow (\Irr \, \Gamma_1 \q \Gamma)_2.
\]
\end{lem}
\begin{proof} 
The irreducible representations of $\Gamma_1$ are $1$-dimensional, 
and we have $\gamma \cdot \pi_x = \pi_x$ for $\gamma \in \Gamma_x$.   
In that case we take each $T_{\gamma,x}$ to be the identity, so that 
$\natural(x)$ is trivial. Then the projective representations of $\Gamma_x$ which 
occur in the construction are all true representations and \eqref{eq:twisting} 
simplifies to $\phi_{g,x}(T_{\gamma,x}) = T_{g \gamma g^{-1},g x}$. Thus we
recover the extended quotient of the second kind in Lemma \ref{lem:Clifford}.
\end{proof}

{\bf 2.} Given a $\Cset$-algebra $R$, a finite group $\Gamma$ and a group
homomorphism
$\Gamma \to \Aut(R)$, we can form the crossed product algebra 
\[R\rtimes\Gamma:=\{\sum_{\gamma\in\Gamma}r_\gamma\gamma\,:\,r_\gamma\in
R\},\]
with multiplication given by the distributive law and the relation
\[\gamma r=\gamma(r)\gamma,\quad\text{for $\gamma\in \Gamma$ and $r\in R$.}\]  
Now $\Gamma$ acts on $X := \Irr\, R$. Assuming that all simple $R$-modules have
countable dimension, so that Schur's lemma is valid, we construct $\square (V)$ 
and $\phi_{\gamma,V}$ as above for group algebras.
Here we have
\[
\widetilde X_2^\square=\{(V,\tau)\,:\,V\in\Irr\,R,\;\tau\in
\Irr \, \C [\Gamma_V, \square(V)] \}.
\]

\begin{lem} \label{lem:Clifford_algebras} 
There is a bijection
\[
\Irr(R \rtimes \Gamma) \longleftrightarrow (\Irr \, R \q \Gamma)_2^{\square}.
\]
If all simple $R$-modules are one-dimensional, then it becomes a natural bijection
\[
\Irr(R \rtimes \Gamma) \longleftrightarrow (\Irr \, R \q \Gamma)_2.
\]
\end{lem}
\begin{proof}  The proof proceeds by comparing our construction with the 
theory of Clifford as stated in \cite[Theorem~A.6]{RamRam}.
The naturality part can be shown in the same way as Lemma \ref{lem:Clifford_abelian}. 
\end{proof}

\begin{notation} \label{not:rtimes} 
{\rm For $(V,\tau)$ as above, $V \otimes V_\tau$ is a simple $R \rtimes \Gamma_V$-module,
in a way which depends on the choice of intertwining operators $T_{\gamma,V}$. The
simple $R \rtimes \Gamma$-module associated to $(V,\tau)$ by the bijection of Lemma
\ref{lem:Clifford_algebras} is
\begin{equation}\label{eq:VrtimesTau}
V \rtimes \tau := \mathrm{Ind}_{R \rtimes \Gamma_V}^{R \rtimes V} (V \otimes V^*_\tau) .
\end{equation}
Similarly, we shall denote by $\tau_1\rtimes\tau$ the element of 
$\Irr (\Gamma_1 \rtimes \Gamma)$ 
which corresponds to $(\tau_1,\tau)$ by the bijection of Lemma~\ref{lem:Clifford}.}
\end{notation}

\section{Weyl groups of disconnected groups}
\label{sec:Wdc}

Let $M$ be a reductive complex algebraic group. Then $M$ may have a finite number 
of connected components,  $M^0$ is the identity component of $M$, and $\cW^{M^0}$ 
is the Weyl group of $M^0$:
\[
\cW^{M^0}: = \Nor_{M^0}(T)/T
\]
where $T$ is a maximal torus of $M^0$.   We will need the analogue of the Weyl 
group for the possibly disconnected group $M$. 

\begin{lem} \label{lem:disconnected}
Let $M,\, M^0,\, T$ be as defined above. Then we have 
\[
\Nor_M (T)/T \cong \cW^{M^0} \rtimes \pi_0(M).
\]
\end{lem}
\begin{proof}
The group $\cW^{M^0}$ is a normal subgroup of $ \Nor_M(T)/T$. Indeed,
let $n\in\Nor_{M^0}(T)$ and let $n'\in\Nor_M(T)$, then $n'nn^{\prime-1}$
belongs to $M^0$ (since the latter is normal in $M$) and normalizes $T$, that is,
$n'nn^{\prime-1}\in\Nor_{M^0}(T)$. On the other hand,
$n'(nT)n^{\prime-1}=n'nn^{\prime-1}(n'Tn^{\prime-1})=n'nn^{\prime-1}T$.

Let $B$ be a Borel subgroup of $M^0$ containing $T$.   
Let $w\in \Nor_M(T)/T$. Then $wBw^{-1}$ is a Borel subgroup of $M^0$
(since, by definition, the Borel subgroups of an algebraic group are 
the maximal closed connected solvable subgroups). Moreover, $wBw^{-1}$  
contains $T$. 
In a connected reductive algebraic group, the intersection of two Borel 
subgroups always contains a maximal torus and the two Borel subgroups are 
conjugate by a element of the normalizer of that torus. Hence $B$ and
$wBw^{-1}$ are conjugate by an element $w_1$ of $\cW^{M^0}$.
It follows that $w_1^{-1}w$ normalises $B$. Hence
\[w_1^{-1}w\in \Nor_M(T)/T \cap \Nor_{M}(B)=\Nor_{M}(T,B)/T,\] 
that is, \[
\Nor_M(T)/T = \cW^{M^0}\cdot(\Nor_M(T,B)/T).\] 
Finally, we have
\[\cW^{M^0}\cap(\Nor_M(T,B)/T)=\Nor_{M^0}(T,B)/T=\{1\},\] 
since $\Nor_{M^0}(B)=B$ and $B\cap \Nor_{M^0}(T)=T$. This proves that
\[
\Nor_M (T) \cong \Nor_{M^\circ}(T) \rtimes \Nor_M (B,T) . 
\]
Now consider the following map:
\begin{align}\label{MM}
\Nor_{M}(T,B)/T\to M/M^0\quad\quad mT\mapsto mM^0.
\end{align}
It is injective. Indeed, let $m,m'\in\Nor_{M}(T,B)$ such that
$mM^0=m'M^0$. Then $m^{-1}m'\in M^0\cap\Nor_{M}(T,B)=\Nor_{M^0}(T,B)=T$
(as we have seen above). Hence $mT=m'T$.

On the other hand, let $m$ be an element in $M$. Then $m^{-1}Bm$ is a
Borel subgroup of $M^0$, hence there exists $m_1\in M^0$ such that
$m^{-1}Bm=m_1^{-1}Bm_1$. It follows that $m_1m^{-1}\in\Nor_M(B)$. Also
$m_1m^{-1}Tmm_1^{-1}$ is a torus of $M^0$ which is contained in 
$m_1m^{-1}Bmm_1^{-1}=B$. Hence $T$ and $m_1m^{-1}Tmm_1^{-1}$ are conjugate
in $B$: there is $b\in B$ such that $m_1m^{-1}Tmm_1^{-1}=b^{-1}Tb$. Then 
$n:=bm_1m^{-1}\in\Nor_M(T,B)$. It gives $m=n^{-1}bm_1$. Since $bm_1\in
M^0$, we obtain $mM^0=n^{-1}M^0$. Hence the map  (\ref{MM}) is surjective.
\end{proof}

Let $G$ be a connected complex reductive group and let $T$ be a maximal torus in $G$.
The Weyl group of $G$ is denoted $\cW^G$. 

\begin{lem} \label{lem:centrals} 
Let $A$ be a subgroup of $T$ and write $M = \Z_G (A)$. Then the isotropy subgroup of $A$ in 
$\cW^G$ is
\[
\cW^G_A = \Nor_M (T) / T \cong \cW^{M^0} \rtimes \pi_0(M) .
\]
In case that the group $M$ is connected, $\cW^G_A$ is the Weyl group of $M$.
\end{lem}
\begin{proof} 
Let $R(G,T)$ denote the root system of $G$. According to \cite[\S~4.1]{SpringerSteinberg},
the group $M = \Cent_G(A)$ is the reductive subgroup of $G$ generated 
by $T$ and those root groups $U_\alpha$ for which $\alpha \in R(G,T)$ has 
trivial restriction to $A$ together with those Weyl group representatives 
$n_w \in \Nor_G(T) \; (w \in \cW^G)$ for which $w(t) = t$ for all $t \in A$.
This shows that $\cW^G_A = \Nor_M (T) / T$, which by Lemma \ref{lem:disconnected}
is isomorphic to $\cW^{M^0} \rtimes \pi_0(M)$.

Also by \cite[\S~4.1]{SpringerSteinberg}, the identity component of $M$ is 
generated by $T$ and those root groups $U_\alpha$ for which $\alpha$ has trivial 
restriction to $A$. Hence the Weyl group $\mathcal W^{M^\circ}$ is the normal 
subgroup of $\cW^G_A$ generated by those reflections $s_\alpha$ and 
\[
\cW^G_A / \cW^{M^\circ} \cong M / M^\circ .
\]
In particular, if $M$ is connected then $\cW^G_A$ is the Weyl group of $M$.
\end{proof}

Consequently, for $t \in T$ such that $M = \Z_G (t)$ we have 
\begin{align}
(T\q \cW^G)_2  & =  \{(t,\sigma) : t \in T, \sigma \in
\Irr(\cW_t^G)\}/\cW^G \label{EXT1} , \\
\Irr\, \cW^G_t  & =  (\Irr \, \cW^{M^0}\q
\pi_0(M))_2^{\square} \label{EXT2} .
\end{align}

We fix a Borel subgroup $B_0$ of $M^\circ$ containing $T$ and let $\Delta (B_0,T)$ be the
set of roots of $(M^\circ,T)$ that are simple with respect to $B_0$. We may and will assume
that this agrees with the previously chosen simple reflections in $\mathcal W^{M^\circ}$.
In every root subgroup $U_\alpha$ with $\alpha \in \Delta (B_0,T)$ we pick a nontrivial 
element $u_\alpha$. The data $(M^\circ,T,(u_\alpha)_{\alpha \in \Delta (B_0,T)})$ are
called a \emph{pinning} of $M^\circ$. This notion is useful in the following well-known
result:

\begin{lem}\label{lem:pinning}
The short exact sequence
\[
1 \to M^\circ / \Z (M^\circ) \to M / \Z (M^\circ) \to \pi_0 (M) \to 1 
\]
is split. A splitting can be obtained by sending $C \in \pi_0 (M)$ to the unique
element of $C / \Z (M^\circ) \subset M / \Z (M^\circ)$ that preserves the chosen pinning.
\end{lem}
\begin{proof}
The connected reductive group $M^\circ$ acts transitively on the set of pairs $(B',T')$
with $B'$ a Borel subgroup containing a maximal torus $T'$. Since the different
simple roots are independent functions on $T$, $M^\circ$ also acts transitively on
the set of pinnings. The stabilizer of a given pinning is $\Z (M^\circ)$, so 
$M^\circ / \Z (M^\circ)$ acts simply transitively on the set of pinnings for $M^\circ$.
This shows that the given recipe is valid and produces a splitting.
\end{proof}

\section{An extended Springer correspondence} 
\label{sec:celldec}
Let $M^\circ$ be a connected reductive complex group.
We take $x \in M^\circ$ unipotent and we abbreviate 
\begin{align}
A_x: = \pi_0 (\Cent_{M^0}(x)).
\end{align}
Let $x \in M^\circ$ be unipotent, $\mathcal B^x = \mathcal B^x_{M^\circ}$ the variety of 
Borel subgroups of $M^\circ$ containing $x$. All the irreducible components of 
$\mathcal B^x$ have the same dimension $d(x)$ over $\Rset$, see \cite[Corollary 3.3.24]{CG}.
Let $H_{d(x)} (\mathcal B^x,\C)$ be its top homology, let $\rho$ be an irreducible 
representation of $A_x$ and write
\begin{equation} \label{eq:tauxrho}
\tau (x,\rho) = \mathrm{Hom}_{A_x} \big( \rho, H_{d(x)}(\mathcal B^x ,\C) \big) .
\end{equation}
We call $\rho \in \Irr (A_x)$ geometric if $\tau (x,\rho) \neq 0$.
The Springer correspondence yields a bijection
\begin{equation} \label{eqn:Springercor}
(x,\rho) \mapsto \tau(x,\rho)
\end{equation}
between the set of $M^0$-conjugacy classes of pairs $(x,\rho)$ formed by a
unipotent element $x \in M^0$ and an irreducible geometric representation 
$\rho$ of $A_x$, and the equivalence classes of irreducible representations 
of the Weyl group $\cW^{M^0}$.

\begin{rem} \label{rem:Springer}
The Springer correspondence which we employ here
sends the trivial unipotent class to the trivial $\cW^{M^\circ}$-representation
and the regular unipotent class to the sign representation.
It coincides with the correspondence constructed by Lusztig by means of intersection
cohomology. The difference with Springer's construction via a reductive group
over a field of positive characteristic consists of tensoring with
the sign representation of $\cW^{M^0}$, see \cite{Hot}.
\end{rem}

Choose a set of simple reflections for $\mathcal W^{M^\circ}$ and let $\Gamma$ be a group 
of automorphisms of the Coxeter diagram of $W$. Then $\Gamma$ acts on $\mathcal W^{M^\circ}$
by group automorphisms, so we can form the semidirect product $\mathcal W^{M^\circ} \rtimes \Gamma$. 
Furthermore $\Gamma$ acts on $\Irr (\mathcal W^{M^\circ})$, by $\gamma \cdot \tau = 
\tau \circ \gamma^{-1}$. The stabilizer of $\tau \in \Irr (\mathcal W^{M^\circ})$ is denoted 
$\Gamma_\tau$. As described in Section \ref{sec:extquot}, Clifford theory for 
$\mathcal W^{M^\circ} \rtimes \Gamma$ produces a 2-cocycle 
$\natural(\tau) : \Gamma_\tau \times \Gamma_\tau \to \C^\times$.

By Lemma \ref{lem:pinning} The action of $\gamma \in \Gamma$ on the Coxeter diagram of 
$\mathcal W^{M^\circ}$ lifts uniquely to an action of $\gamma$ on $M^\circ$ which preserves 
the pinning chosen in Section \ref{sec:Wdc}. In this way we construct the semidirect product 
$M := M^\circ \rtimes \Gamma$.
By Lemma \ref{lem:centrals} we may identify $\cW^M$ with $\cW^{M^\circ} \rtimes \Gamma$.
We want to generalize the Springer correspondence to this kind of group. First we need 
to prove a technical lemma, which in a sense extension of Lemma \ref{lem:pinning}.

\begin{lem}\label{lem:S.3}
Let $\rho \in \Irr (\pi_0 (\Z_{M^\circ}(x)))$ and write 
\[
\Z_M (x,\rho) =
\{ m \in \Z_M (x) | \\ \rho \circ \mathrm{Ad}_m^{-1} \cong \rho \}.
\]
The following short exact sequence splits:
\[
1 \to \pi_0 \big( \Z_{M^\circ}(x,\rho) / \Z(M^\circ) \big) \to 
\pi_0 \big( \Z_M (x,\rho) / \Z(M^\circ) \big) \to \Gamma_{[x,\rho]_{M^\circ}} \to 1 .
\] 
\end{lem}
\begin{proof}
First we ignore $\rho$. According to the classification of unipotent orbits in complex
reductive groups \cite[Theorem 5.9.6]{Carter} we may assume that $x$ is distinguished unipotent
in a Levi subgroup $L \subset M^\circ$ that contains $T$. Notice that the derived subgroup 
$\mathcal D(L)$ contains only the part of $T$ generated by the coroots of $(L,T)$. The roots of
\[
L' := \Z_{M^\circ}(\mathcal D (L)) (T \cap \mathcal D (L)) = \Z_{M^\circ}(\mathcal D (L)) T .
\]
are precisely those that are orthogonal to the coroots of $(L,T)$.
We choose Borel subgroups $B_L \subset L$ and $B'_L \subset L'$ such that $x \in B_L$ and
$T \subset B_L \cap B'_L$.

Let $[x]_{M^\circ}$ be the $M^\circ$-conjugacy class of $x$ and 
$\Gamma_{[x]_{M^\circ}}$ its stabilizer in $\Gamma$. Any $\gamma \in 
\Gamma_{[x]_{M^\circ}}$ must also stabilize the $M^\circ$-conjugacy class of $L$,
and $T = \gamma (T) \subset \gamma (L)$, so there exists a $w_1 \in \mathcal W^{M^\circ}$ with 
$w_1 \gamma (L) = L$. Adjusting $w_1$ by an element of $W(L,T) \subset \mathcal W^{M^\circ}$, 
we can achieve that moreover 
$w_1 \gamma (B_L) = B_L$. Then $w_1 \gamma (L') = L'$, so we can find a unique 
$w_2 \in W (L',T) \subset \mathcal W^{M^\circ}$ with $w_2 w_1 \gamma (B'_L) = B'_L$. 
Notice that the centralizer of $\Phi (B_L,T) \cup \Phi (B'_L,T)$ in $\mathcal W^{M^\circ}$ 
is trivial, because it is generated by
reflections and no root in $\Phi (M^\circ,T)$ is orthogonal to this set of roots.
Therefore the above conditions completely determine $w_2 w_1 \in \mathcal W^{M^\circ}$.

The element $w_1 \gamma \in \mathcal W^{M^\circ} \rtimes \Gamma$ acts on $\Delta (B_L,T)$ 
by a diagram automorphism. So upon choosing $u_\alpha \in U_\alpha \setminus \{1\}$ for 
$\alpha \in \Delta (B_L,T)$, Lemma \ref{lem:pinning} shows that $w_1 \gamma$ can be 
represented by a unique element 
\[
\overline{w_1 \gamma} \in \mathrm{Aut} \big( \mathcal D (L),T,
(u_\alpha )_{\alpha \in \Delta (B_L,T)} \big) .
\]
The distinguished unipotent class of $x \in L$ is determined by its Bala--Carter diagram.
The classification of such diagrams \cite[\S 5.9]{Carter} shows that there exists an element
$\bar x$ in the same class as $x$, such that $\mathrm{Ad}_{\overline{w_1 \gamma}}(\bar x) = 
\bar x$. We may just as well assume that we had $\bar x$ instead of $x$ from the start,
and that $\overline{w_1 \gamma} \in \Z_M (x)$. Clearly we can find a representative 
$\overline{w_2}$ for $w_2$ in $\Z_M (x)$, so we obtain
\[
\overline{w_2} \, \overline{w_1 \gamma} \in \Z_M (x) \cap \Nor_M (T) \quad \text{and} \quad
w_2 w_1 \gamma \in \displaystyle{\frac{\Z_M (x) \cap \Nor_M (T)}{\Z(M^\circ) \, T}} . 
\]
Since $w_2 w_1 \in \mathcal W^{M^\circ}$ is unique,
\begin{equation}\label{eq:splittingX}
s : \Gamma_{[x]_{M^\circ}} \to \displaystyle{\frac{\Z_M (x) \cap 
\Nor_M (T)}{\Z(M^\circ) \, T}} ,\; \gamma \mapsto w_2 w_1 \gamma 
\end{equation}
is a group homomorphism. 

We still have to analyse the effect of $\Gamma_{[x]_{M^\circ}}$ on $\rho \in \Irr (A_x)$. 
Obviously composing with $\mathrm{Ad}_m$ for $m \in \Z_{M^\circ}(x)$ does not change 
the equivalence class of any representation of $A_x = \pi_0 (\Z_{M^\circ}(x))$. 
Hence $\gamma \in \Gamma_{[x]_{M^\circ}}$ stabilizes $\rho$ if and only if any lift 
of $\gamma$ in $\Z_M (x)$ does. This applies in particular to 
$\overline{w_2} \, \overline{w_1 \gamma}$, and therefore
\[
s ( \Gamma_{[x,\rho]_{M^\circ}}) \subset 
\big(\Z_M (x,\rho) \cap \Nor_M (T)\big) \big/ \big( \Z(M^\circ) \, T \big) .
\]
Since the torus $T$ is connected, $s$ determines a group homomorphism from
$\Gamma_{[x,\rho]_{M^\circ}}$ to $\pi_0 \big( \Z_M (x,\rho) / \Z(M^\circ) \big)$, 
which is the required splitting.
\end{proof}

A further step towards a Springer correspondence for $\cW^M$ is:

\begin{prop}\label{prop:S.2}
The class of $\natural(\tau)$ in $H^2 (\Gamma_\tau , \C^\times)$ is trivial for all 
$\tau \in \Irr (\mathcal W^{M^\circ})$. There is a bijection between
\[
\big( \Irr (\mathcal W^{M^\circ}) /\!/ \Gamma \big)_2 \qquad  \text{and} \qquad
\Irr (\mathcal W^{M^\circ} \rtimes \Gamma) = \Irr (\mathcal W^M ) .
\]
\end{prop}
\begin{proof}
There are various ways to construct the Springer correspondence for 
$\mathcal W^{M^\circ}$, for the current proof we use the method with Borel--Moore 
homology. Let $\mathcal \Z_{M^\circ}$ be the Steinberg variety of $M^\circ$ and 
$H_{\top} (\mathcal \Z_{M^\circ})$ its homology in the top degree 
\[
2 \dim_\C \mathcal \Z_{M^\circ} = 4 \dim_\C \mathcal B_{M^\circ} = 
4 (\dim_\C M^\circ - \dim_\C B_0) ,
\]
with rational coefficients. We define a natural algebra isomorphism 
\begin{equation}\label{eq:S.5}
\Q [\mathcal W^{M^\circ}] \to H_{\top}(\mathcal \Z_{M^\circ})  
\end{equation}
as the composition of \cite[Theorem 3.4.1]{CG} and a twist by the sign representation 
of $\Q [\mathcal W^{M^\circ}]$. By \cite[Section 3.5]{CG} the action of $\mathcal W^{M^\circ}$ 
on $H_* (\mathcal B^x ,\C)$ (as defined by Lusztig) corresponds to the convolution 
product in Borel--Moore homology. 

Since $M^\circ$ is normal in $M$, the groups $\Gamma ,M$ and $M / \Z(M)$ act on the 
Steinberg variety $\mathcal \Z_{M^\circ}$ via conjugation. The induced action of the 
connected group $M^\circ$ on $H_{\top}(\mathcal \Z_{M^\circ})$ is trivial, and it easily 
seen from \cite[Section 3.4]{CG} that the action of $\Gamma$ on $H(\mathcal \Z_{M^\circ})$ 
makes \eqref{eq:S.5} $\Gamma$-equivariant.

The groups $\Gamma ,M$ and $M / \Z(M)$ also act on the pairs $(x,\rho)$ 
and on the varieties of Borel subgroups, by
\begin{align*}
& \mathrm{Ad}_m (x,\rho) = (m x m^{-1}, \rho \circ \mathrm{Ad}_m^{-1}) , \\
& \mathrm{Ad}_m : \mathcal B^x \to \mathcal B^{m x m^{-1}} ,\; B \mapsto m B m^{-1} .
\end{align*}
Given $m \in M$, this provides a linear bijection $H_* (\mathrm{Ad}_m)$ : 
\[
\mathrm{Hom}_{A_x}(\rho, H_* (\mathcal B^x ,\C)) \to
\mathrm{Hom}_{A_{mxm^{-1}}}(\rho \circ \mathrm{Ad}_m^{-1}, H_* (\mathcal B^{mxm^{-1}} ,\C)) .
\]
The convolution product in Borel--Moore homology is compatible with these 
$M$-actions so, as in \cite[Lemma 3.5.2]{CG}, the following diagram commutes 
for all $h \in H_{\top}(\mathcal \Z_{M^\circ})$:
\begin{equation}\label{eq:S.6}
\begin{array}{ccc}
H_* (\mathcal B^x ,\C) & \xrightarrow{\; h \;} & H_* (\mathcal B^x ,\C) \\
\downarrow \scriptstyle{H_* (\mathrm{Ad}_m)} & & 
\downarrow \scriptstyle{H_* (\mathrm{Ad}_m)} \\
H_* (\mathcal B^{mxm^{-1}} ,\C) & \xrightarrow{m \cdot h} & 
H_* (\mathcal B^{mxm^{-1}} ,\C) .
\end{array}
\end{equation}
In case $m \in M^\circ \gamma$ and $m \cdot h$ corresponds to $w \in \mathcal W^{M^\circ}$, 
the element $h \in H(\mathcal \Z_{M^\circ})$ corresponds to $\gamma^{-1}(w)$, 
so \eqref{eq:S.6} becomes
\begin{equation}\label{eq:S.7}
H_* (\mathrm{Ad}_m) \circ \tau (x,\rho) (\gamma^{-1}(w)) = 
\tau (mxm^{-1}, \rho \circ \mathrm{Ad}_m^{-1})(w) \circ H_* (\mathrm{Ad}_m) .
\end{equation}
Denoting the $M^\circ$-conjugacy class of $(x,\rho)$ by $[x,\rho]_{M^\circ}$, we can write
\begin{align}\label{eq:S.9}
\Gamma_{\tau (x,\rho)} & = \{ \gamma \in \Gamma \mid \tau (x,\rho) 
\circ \gamma^{-1} \cong \tau (x,\rho) \} \\
\nonumber & = \{ \gamma \in \Gamma \mid [\mathrm{Ad}_\gamma 
(x,\rho)]_{M^\circ} = [x,\rho]_{M^\circ} \} =: \Gamma_{[x,\rho]_{M^\circ}} .
\end{align}
This group fits in an exact sequence
\begin{equation}\label{eq:S.4}
1 \to \pi_0 \big( \Z_{M^\circ} (x,\rho) / \Z(M^\circ) \big) \to \pi_0 \big( \Z_M (x,\rho) / 
\Z(M^\circ) \big) \to \Gamma_{[x,\rho]_{M^\circ}} \to 1 ,
\end{equation}
which by Lemma \ref{lem:S.3} admits a splitting 
\[
s : \Gamma_{[x,\rho]_{M^\circ}} \to \pi_0 \big( \Z_M (x,\rho) / \Z(M^\circ) \big) . 
\]
By homotopy invariance in Borel--Moore homology $H_* (\mathrm{Ad}_z) = 
\mathrm{id}_{H_* (\mathcal B^x ,\C)}$ for any $z \in \Z_{M^\circ}(x,\rho)^\circ 
\Z(M^\circ)$, so $H_* (\mathrm{Ad}_m)$ is well-defined for \\
$m \in \pi_0 \big( \Z_M (x,\rho) / \Z(M^\circ) \big)$. In particular we obtain for every 
$\gamma \in \Gamma_{\tau (x,\rho)} = \Gamma_{[x,\rho]_{M^\circ}}$ a linear bijection
\[
H_* (\mathrm{Ad}_{s(\gamma )}) : \mathrm{Hom}_{A_x}(\rho, H_{d(x)}
(\mathcal B_x ,\C)) \to \mathrm{Hom}_{A_x}(\rho, H_{d(x)} (\mathcal B_x ,\C)) ,
\]
which by \eqref{eq:S.7} intertwines the $\mathcal W^{M^\circ}$-representations 
$\tau (x,\rho)$ and $\tau (x,\rho) \circ \gamma^{-1}$. By construction
\begin{equation}\label{eq:sGammaMult}
H_* (\mathrm{Ad}_{s(\gamma )}) \circ H_* (\mathrm{Ad}_{s(\gamma' )}) =
H_* (\mathrm{Ad}_{s(\gamma \gamma')}) .
\end{equation}
This establishes the triviality of the 2-cocycle $\natural(\tau) = \natural (\tau (x,\rho))$.

Consider any $g \in \Gamma \setminus \Gamma_x$. Then $g \tau$ corresponds to
\[
\mathrm{Ad}_g (x, \rho) = (g x g^{-1}, \rho \circ \mathrm{Ad}_g^{-1} ) .
\]
For $\gamma \in \Gamma_x$ we define an intertwining operator in
\[
\mathrm{End}_{\cW^{M^\circ}} \big( \mathrm{Hom}_{A_{g x g^{-1}}}
(\rho \circ \mathrm{Ad}_g^{-1}, H_{d(x)} (\mathcal B_{g x g^{-1}},\C) ) \big)
\]
associated to $g \gamma g^{-1} \in \Gamma_{g x g^{-1}}$ as
\begin{equation}\label{eq:connectingHom}
H_{d(x)}(\mathrm{Ad}_{g s(\gamma) g^{-1}}) =
H_{d(x)}(\mathrm{Ad}_g) H_{d(x)}(\mathrm{Ad}_{s(\gamma)}) H_{d(x)}(g^{-1}) . 
\end{equation}
We do the same for any other point in the $\Gamma$-orbit of $(x,\rho)$. 
Then \eqref{eq:sGammaMult} shows that the resulting intertwining operators do not depend 
on the choices of the elements $g$.
 
We follow the same recipe for any other $\Gamma$-orbit of Springer parameters $(x',\rho')$.
As connecting homomorphism $\phi_{g,(x',\rho')}$ we take conjugation 
by $H_{d(x')}(\mathrm{Ad}_g )$. From this construction and Lemma
\ref{lem:Clifford_algebras} we obtain a bijection between 
$\Irr (\cW^{M^\circ} \rtimes \pi_0 (M))$ and the extended quotient of the second kind 
$\big( \Irr (\mathcal W^{M^\circ}) /\!/ \Gamma \big)_2$.
\end{proof}

We note that the bijection from Proposition \ref{prop:S.2} is in general not canonical,
because the splitting from Lemma \ref{lem:S.3} is not. But with some additional effort 
we can extract a natural description of $\Irr (\cW^M)$ from Proposition \ref{prop:S.2}. 

We say that an irreducible representation $\rho_1$ of 
$\Cent_M (x)$ is geometric if every irreducible $\Cent_{M^\circ}(x)$-subrepresentation of
$\rho_1$ is geometric in the previously defined sense. Notice that this condition forces 
$\rho_1$ to factor through the component group $\pi_0 (\Cent_M (x))$. 

We note that $\pi_0 (\Cent_M (x))$ acts naturally on $H_{d(x)}(\cB^x)$ and on
$\C [\Gamma]$, via the isomorphism
\begin{equation}\label{eq:isoCentGamma}
\Cent_{M}(x) / \Cent_{M^\circ}(x) \cong \Gamma_{[x]_{M^\circ}} .
\end{equation}

\begin{thm}\label{thm:SpringerExtended}
There is a natural bijection from
\[
\big\{ (x,\rho_1) \mid x \in M^\circ \text{ unipotent} , \rho_1 \in 
\Irr \big( \pi_0 (\Cent_M (x)) \big) \text{ geometric} \big\} / M 
\]
to $\Irr (\mathcal W^M )$, which sends $(x,\rho_1)$ to
\[
\Hom_{\pi_0 (\Cent_M (x))} \big( \rho_1, H_{d(x)}(\cB^x) \otimes \C [\Gamma] \big) .
\]
\end{thm}
\begin{proof}
Let us take another look at the geometric representations of 
$A_x =\Cent_{M^\circ}(x)$.
By construction they factor through $\pi_0 (\Cent_{M^\circ}(x) / \Cent (M^\circ))$.
From \eqref{eq:splittingX} we get a group isomorphism
\begin{equation}\label{eq:quotientFromSplitting}
\pi_0 (\Cent_{M}(x) / \Cent (M^\circ)) \cong \pi_0 (\Cent_{M^\circ}(x) / 
\Cent (M^\circ)) \rtimes s \big( \Gamma_{[x]_{M^\circ}} \big) .
\end{equation}
Suppose that $\rho \in \Irr (A_x)$ is geometric. Then the operators
$H_{d(x)}(\mathrm{Ad}_{s(\gamma)})$ intertwine $\rho$ with the 
$\pi_0 (\Cent_{M^\circ}(x) / \Cent (M^\circ))$-representation $s(\gamma) \cdot \rho$
and they satisfy the multiplicativity relation \eqref{eq:sGammaMult}.
Now it follows from Lemma \ref{lem:Clifford} that every irreducible geometric 
representation of $\pi_0 (\Cent_M (x))$ can be written in a unique way as 
$\rho \rtimes \sigma$, with $\rho \in \Irr (A_x)$ geometric and 
\[
\sigma \in \Irr s(\Gamma_{[x,\rho]_{M^\circ}}) = \Irr (\Gamma_{[x,\rho]_{M^\circ}}).
\] 
This enables us to rewrite $\widetilde{\Irr (\cW^{M^\circ})}$ as a union of pairs
$(x,\rho_1 = \rho \rtimes \sigma)$, with $x$ in a finite union of chosen 
$\Gamma$-orbits of unipotent elements. Clearly $M$ acts on the larger space
\[
\big\{ (x,\rho_1) \mid x \in M^\circ \text{ unipotent} , \rho_1 \in 
\Irr \big( \pi_0 (\Cent_M (x)) \big) \text{ geometric} \big\} 
\]
by conjugation of the $x$-parameter and the action induced by $H_* (\mathrm{Ad}_m)$ on
the $\rho_1$-parameter. By \eqref{eq:connectingHom} and the construction of $s(\gamma)$ 
in Lemma \ref{lem:S.3}, this extends the action of $\Gamma$ on 
$\widetilde{\Irr (\cW^{M^\circ})}$.
That provides the bijection from $\big( \Irr (\mathcal W^{M^\circ}) /\!/ \Gamma \big)_2$ 
to set of the $M$-association classes of pairs $(x,\rho_1)$. Combining this with 
Proposition \ref{prop:S.2}, we obtain a bijection between $\Irr (\cW^M)$ and the latter 
set. If we work out the definitions and use \eqref{eq:VrtimesTau}, we see that it sends 
$(x,\rho_1 = \rho \rtimes \sigma)$ to
\[ 
\tau (x,\rho) \rtimes \sigma = 
\mathrm{Ind}_{\cW^{M^\circ} \rtimes \Gamma_{[x,\rho]_{M^\circ}}}^{\cW^{M^\circ} \rtimes \Gamma}
\big( \tau (x,\rho) \otimes \sigma ) .
\]
Since every irreducible complex representation of a finite group is isomorphic to its
contragredient, we can rewrite this as
\begin{align*}
& \mathrm{Ind}_{\cW^{M^\circ} \rtimes \Gamma_{[x,\rho]_{M^\circ}}}^{\cW^{M^\circ} \rtimes \Gamma} 
\big( \Hom_{A_x} (\rho, H_{d(x)}(\cB^x)) \otimes \sigma^* \big) \cong \\
& \mathrm{Ind}_{\cW^{M^\circ} \rtimes \Gamma_{[x,\rho]_{M^\circ}}}^{\cW^{M^\circ} \rtimes \Gamma} 
\big( \Hom_{\Gamma_{[x,\rho]_{M^\circ}}} \big( \sigma, \Hom_{A_x} (\rho, H_{d(x)}(\cB^x)) 
\otimes \C [\Gamma_{[x,\rho]_{M^\circ}}] \big) \big).
\end{align*}
In view of Lemma \ref{lem:S.3}, the previous line is isomorphic to
\begin{align*}
& \mathrm{Ind}_{\cW^{M^\circ} \rtimes \Gamma_{[x,\rho]_{M^\circ}}}^{\cW^{M^\circ} \rtimes \Gamma} 
\big( \Hom_{\Cent_M (x,\rho)} \big( \rho \otimes \sigma, H_{d(x)}(\cB^x) \otimes 
\C [\Gamma_{[x,\rho]_{M^\circ}}] \big) \big) \cong \\
& \mathrm{Ind}_{\cW^{M^\circ} \rtimes \Gamma_{[x]_{M^\circ}}}^{\cW^{M^\circ} \rtimes \Gamma} 
\big( \Hom_{\Cent_M (x,\rho)} \big( \rho \otimes \sigma, 
H_{d(x)}(\cB^x) \otimes \C [\Gamma_{[x]_{M^\circ}}] \big) \big).
\end{align*}
With Frobenius reciprocity and  \eqref{eq:isoCentGamma} we simplify the above expression to
\begin{align*}
& \mathrm{Ind}_{\cW^{M^\circ} \rtimes \Gamma_{[x]_{M^\circ}}}^{\cW^{M^\circ} \rtimes \Gamma}
\big( \Hom_{\Cent_M (x)} \big( \rho \rtimes \sigma, 
H_{d(x)}(\cB^x) \otimes \C [\Gamma_{[x]_{M^\circ}}] \big) \big) \cong \\
& \Hom_{\pi_0 (\Cent_M (x))} \big( \rho \rtimes \sigma, 
H_{d(x)}(\cB^x) \otimes \C [\Gamma] \big) .
\end{align*}
The last line is natural in $(x,\rho_1 = \rho \rtimes \sigma)$ because the 
$Z_M (x)$-representation $H_{d(x)}(\mathcal B^x)$ depends in a natural way on $x$, as we 
observed at the start of the proof of Proposition \ref{prop:S.2}.
\end{proof}

There is natural partial order on the unipotent classes in $M$:
\[
\cO < \cO' \quad \text{when} \quad \overline{\cO} \subsetneq \overline{\cO'} .
\]
Let $\cO_x \subset M$ be the class containing $x$.
We transfer this to partial order on our extended Springer data by defining
\begin{equation}\label{eq:S.32}
(x,\rho_1) < (x',\rho'_1) \quad \text{when} \quad 
\overline{\cO_x} \subsetneq \overline{\cO_{x'}} .
\end{equation}
We will use it to formulate a property of the composition series of some 
$\cW^M$-representations that will appear later on.

\begin{lem}\label{lem:S.6}
Let $x \in M$ be unipotent and let $\rho \rtimes \sigma$ 
be a geometric irreducible representation of $\pi_0 (\Z_M (x))$.
There exist multiplicities \\ $m_{x,\rho \rtimes \sigma ,x',\rho' \rtimes \sigma'} 
\in {\mathbb Z}_{\geq 0}$ such that
\begin{multline*}
\mathrm{Ind}_{W \rtimes \Gamma_{[x,\rho]_{M^\circ}}}^{W \rtimes \Gamma} \big(
\Hom_{A_x} \big( \rho,H_* (\mathcal B^x,\C) \big) \otimes \sigma \big) \cong \\
\tau (x,\rho) \rtimes \sigma \oplus 
\bigoplus_{(x',\rho' \rtimes \sigma') > (x,\rho \rtimes \sigma)} 
m_{x,\rho \rtimes \sigma ,x',\rho' \rtimes \sigma'}\, \tau (x',\rho') \rtimes \sigma' .
\end{multline*}
\end{lem}
\begin{proof}
Consider the vector space $\Hom_{A_x} \big( \rho,H_* (\mathcal B^x,\C) \big)$
with the $\cW^{M^\circ}$-action coming from \eqref{eq:S.5}. 
The proof of Proposition \ref{prop:S.2} remains valid for these representations. 
By \cite[Theorem 4.4]{BM} (attributed to Borho and MacPherson) there exist 
multiplicities $m_{x,\rho,x',\rho'} \in {\mathbb Z}_{\geq 0}$ such that
\begin{equation}
\Hom_{A_x} \big( \rho,H_* (\mathcal B^x,\C) \big) \cong
\tau (x,\rho) \oplus \bigoplus_{(x',\rho') > (x,\rho)} m_{x,\rho ,x',\rho'}\, \tau (x',\rho') .
\end{equation}
By \eqref{eq:S.9} and \eqref{eq:S.7} $\Gamma_{[x,\rho]_{M^\circ}}$ also
stabilizes the $\tau (x',\rho')$ with $m_{x,\rho,x',\rho'} > 0$, and by 
Proposition \ref{prop:S.2} the associated 2-cocycles are trivial. It follows that 
\begin{multline}\label{eq:S.31}
\mathrm{Ind}_{W \rtimes \Gamma_{[x,\rho]_{M^\circ}}}^{W \rtimes \Gamma} \big(
\Hom_{A_x} \big( \rho,H_* (\mathcal B^x,\C) \big) \otimes \sigma \big) \cong \\
\tau (x,\rho) \rtimes \sigma \oplus 
\bigoplus_{(x',\rho') > (x,\rho)} m_{x,\rho ,x',\rho'} 
\mathrm{Ind}_{W \rtimes \Gamma_{[x,\rho]_{M^\circ}}}^{W \rtimes \Gamma} 
\big( \tau (x',\rho') \otimes \sigma \big) .
\end{multline}
Decomposing the right hand side into irreducible representations then gives the
statement of the lemma.
\end{proof}

\section{Langlands parameters for the principal series} 
\label{sec:Lp}
Let $\mathbf{W}_F$ denote the Weil group of $F$, let $\mathbf{I}_F$ be the inertia
subgroup of $\mathbf{W}_F$. 
Let $\mathbf{W}_F^{\der}$ denote the closure of the commutator subgroup of $\mathbf{W}_F$, 
and write $\mathbf{W}_F^{\ab} = \mathbf{W}_F/\mathbf{W}^{\der}_F$. 
The group of units in $\mathfrak{o}_F$ will be denoted $\fo_F^\times$.

We recall the Artin reciprocity map $\mathbf{a}_F : \mathbf{W}_F \to F^{\times}$ 
which has the following properties (local class field theory):
\begin{enumerate}
\item The map $\mathbf{a}_F$ induces a topological isomorphism 
$\mathbf{W}^{\ab}_F \simeq F^{\times}$.
\item An element $x \in \mathbf{W}_F$ is a geometric Frobenius if and only if 
$\mathbf{a}_F(x)$ is a prime element $\varpi_F$ of $F$.
\item We have $\mathbf{a}_F(\mathbf{I}_F) = \fo_F^\times$.
\end{enumerate}
We now consider the  principal series of $\cG$. We recall that  
$\mathcal{G}$ denotes a connected reductive split $p$-adic group with 
maximal split torus $\mathcal{T}$, and that
$G ,\;T$ denote the Langlands dual groups of $\mathcal{G} ,\; \mathcal{T}$. 
Next, we consider conjugacy classes in $G$ of continuous morphisms
\[
\Phi\colon \mathbf{W}_F\times \SL_2 (\Cset) \to G
\] 
which are rational on $\SL_2 (\Cset)$ and such that $\Phi(\mathbf{W}_F)$ 
consists of semisimple elements in $G$. 

The (conjectural) local Langlands correspondence is supposed to
be compatible with respect to inclusions of Levi subgroups. Therefore every
Langlands parameter $\Phi$ for a principal series representation should have
$\Phi (\mathbf{W}_F)$ contained in a maximal torus of $G$. As $\Phi$ is only
determined up to $G$-conjugacy, it should suffice to consider Langlands
parameters with $\Phi (\mathbf{W}_F) \subset T$. 

In particular, for such parameters
$\Phi \big|_{\mathbf{W}_F}$ factors through $\mathbf{W}_F^{ab} \cong F^\times$.
We view the domain of $\Phi$ to be $F^{\times} \times \SL_2 (\Cset)$: 
\[
\Phi\colon F^{\times} \times \SL_2 (\Cset) \to G.
\]
In this section we will build such a continuous morphism $\Phi$ from $\fs$ and data 
coming from the extended quotient of second kind. In Section \ref{sec:Borel} we show
how such a Langlands parameter $\Phi$ can be enhanced with a parameter $\rho$.

Throughout this article, a Frobenius element $\Frob_F$ has been chosen and fixed.  
This determines a uniformizer $\varpi_F$ via the equation $\mathbf{a}_F(\Frob_F) = \varpi_F$.  
That in turn gives rise to a group isomorphism $\fo_F^\times \times \mathbb Z \to F^\times$,
which sends $1 \in \mathbb Z$ to $\varpi_F$.
Let $\cT_0$ denote the maximal compact subgroup of $\cT$. As the latter is $F$-split,
\begin{equation}\label{eq:cT0}
\cT \cong F^\times \otimes_{\mathbb Z} X_* (\cT) \cong (\fo_F^\times \times \mathbb Z)
\otimes_{\mathbb Z} X_* (\cT) = \cT_0 \times X_* (\cT) .
\end{equation}
Because $\cW$ does not act on $F^\times$, these isomorphisms are $\cW$-equivariant if
we endow the right hand side with the diagonal $\cW$-action.
Thus \eqref{eq:cT0} determines a $\cW$-equivariant isomorphism of character groups 
\begin{equation}\label{split}
\Irr (\cT) \cong \Irr (\cT_0) \times \Irr (X_* (\cT)) = \Irr (\cT_0) \times X_{\unr}(\cT) .
\end{equation}
The way $\Irr (\cT_0)$ is embedded depends on the choice of $\varpi_F$. However, the
isomorphisms 
\begin{align}
& \Irr (\cT_0) \cong \Hom (\fo_F^\times ,T) , \label{eq:IrrT0} \\
& X_{\unr}(\cT) \cong \Hom (\Zset ,T) = T . \label{eq:XunrT}
\end{align}
are canonical.

\begin{lem}\label{lem:cBernstein} 
Let $\chi$ be a character of $\cT$, and let 
$[\cT,\chi]_{\cG}$ be the inertial class of the pair $(\cT,\chi)$. Let 
\begin{align}\label{artin}
\fs = [\cT,\chi]_{\cG}.
\end{align}
Then $\fs$ determines, and is determined by, the $\cW$-orbit of a smooth morphism
\[
c^\fs \colon \fo_F^\times \to T.
\]
\end{lem}
\begin{proof}
There is a natural isomorphism
\begin{multline*}
\Irr (\cT) = \Hom (F^\times \otimes_{\Zset} X_* (\cT),\C^\times) \\
\cong \Hom (F^\times ,\C^\times \otimes_\Zset X^* (\cT)) = \Hom (F^\times ,T) . 
\end{multline*}
Let $\hat \chi \in \Hom (F^\times ,T)$ be the image of $\chi$ under these isomorphisms. 
By \eqref{eq:IrrT0} the restriction of $\hat \chi$ to $\fo_F^\times$ is not disturbed by 
unramified twists, so we take that as $c^\fs$. Conversely, by \eqref{split} $c^\fs$ 
determines $\chi$ up to unramified twists. Two elements of $\Irr (\cT)$ are 
$\cG$-conjugate if and only if they
are $\cW$-conjugate so, in view of \eqref{split}, the $\cW$-orbit
of the $c^\fs$ contains the same amount of information as $\fs$.
\end{proof}

Let $H = \Cent_G (\text{im} \, c^\fs)$ and let $M = \Z_H (t)$ for some $t \in T$.
Recall that a unipotent element $x\in M^0$ is said to be
\emph{distinguished} if the connected center $\Z_{M^0}^0$ of $M^0$ is 
a maximal torus of $\Cent_{M^0}(x)$. Let $x\in M^0$ unipotent. 
If $x$ is not distinguished, then there is a Levi subgroup $L$ of $M^0$ 
containing $x$ and such that $x\in L$ is distinguished. 

Let $X\in \text{Lie }M^0$ such that $\exp(X) = x$.
A cocharacter $h \colon \Cset^\times\to M^0$ is said to be \emph{associated to} $x$ if 
\[
\Ad(h(t))X=t^2 X \quad\text{for each $t\in\Cset^\times$},
\]
and if the image of $h$ lies in the derived group of some Levi subgroup
$L$ for which $x\in L$ is distinguished (see \cite[Rem.~5.5]{J}
or \cite[Rem.2.12]{FR}). 

A cocharacter associated to a unipotent element $x\in M^0$ is not unique.
However, any two cocharacters associated to a given $x\in M^0$ are
conjugate under elements of $\Cent_{M^0}(x)^0$ (see for instance 
\cite[Lem.~5.3]{J}).

\smallskip

We work with the Jacobson--Morozov theorem \cite[p. 183]{CG}.  Let
$\matje{1}{1}{0}{1}$ be the standard unipotent matrix in $\SL_2(\Cset)$ and let 
$x$ be a unipotent element in $M^0$. There exist rational homomorphisms 
\begin{equation} \label{eqn:gamt}
\gamma \colon \SL_2 (\Cset) \to M^0 \quad \text{with} \quad \gamma \matje{1}{1}{0}{1} = x ,
\end{equation}
see \cite[\S 3.7.4]{CG}. Any two such homomorphisms  $\gamma$ are conjugate by
elements of $\Z_{M^\circ}(x)$. 

For $\alpha \in \Cset^{\times}$ we define the following matrix in $\SL_2 (\Cset)$:
\[
Y_{\alpha} = \matje{\alpha}{0}{0}{\alpha^{-1}} .
\]
Then each $\gamma$ as above determines a cocharacter 
$h\colon \Cset^\times\to M^0$ by setting 
\begin{equation}\label{hh}
h(\alpha):=\gamma(Y_{\alpha})\quad\text{for } \alpha\in\Cset^\times .
\end{equation}
Each cocharacter $h$ obtained in this way is associated to $x$, see \cite[Rem.~5.5]{J}
or \cite[Rem.2.12]{FR}. 
Hence each two such cocharacters are conjugate under $\Cent_{M^0}(x)^0$. 

We set $\Phi (\varpi_F) = t \in T$.
Define the Langlands parameter $\Phi$ as follows:
\begin{equation}\label{eqn:Phi}
\Phi \colon F^{\times} \times \SL_2 (\Cset) \to G, \qquad  
(u\varpi_F^n,Y) \mapsto c^\fs (u) \cdot t^n\cdot \gamma(Y) 
\end{equation}
for all  $u \in \fo_F^\times, \; n \in \Zset,\; Y \in \SL_2 (\Cset)$.

Note that the definition of $\Phi$ uses the appropriate data: 
the semisimple element $t \in T$, the map $c^\fs$, and the 
homomorphism $\gamma$ (which depends on the Springer parameter $x$).  

Since $x$ determines $\gamma$ up to $M^\circ$-conjugation, $c^\fs,x$ and $t$ 
determine $\Phi$ up to conjugation by their common centralizer in $G$. 
Notice also that one can recover $c^\fs, x$ and $t$ from $\Phi$ and that
\begin{equation}\label{eq:hPhi}
h (\alpha) = \Phi (1, Y_\alpha) . 
\end{equation}

\section{Varieties of Borel subgroups}
\label{sec:Borel}

We clarify some issues with different varieties of Borel subgroups and different
kinds of parameters arising from them. 
Let $G$ be a connected reductive complex group and let 
\[
\Phi \colon \mathbf W_F \times \SL_2 (\C) \to G 
\]
be as in \eqref{eqn:Phi}. We write 
\begin{align*}
& H = \Z_G (\Phi (\mathbf I_F)) = \Z_G (\im c^\fs), \\
& M = \Z_G (\Phi (\mathbf W_F)) = \Z_H (t). 
\end{align*}
Although both $H$ and $M$ are in general disconnected, $\Phi (\mathbf W_F)$ is always 
contained in $H^\circ$ because it lies in the maximal torus $T$ of $G$ and 
$H^\circ$. Hence $\Phi (\mathbf I_F) \subset \Z(H^\circ)$.

By construction $t$ commutes with $\Phi (\SL_2 (\C)) \subset M$. 
For any $q^{1/2} \in \C^\times$ the element 
\begin{equation}\label{eq:S.12}
t_q := t \Phi \big( Y_{q^{1/2}} \big)
\end{equation} 
satisfies the familiar relation $t_q x t_q^{-1} = x^q$. Indeed
\begin{equation}\label{eq:tqx}
\begin{split}
t_q x t_q^{-1} & = t \Phi (Y_{q^{1/2}}) \Phi \matje{1}{1}{0}{1} \Phi (Y_{q^{1/2}}^{-1}) t^{-1} \\
& = t \Phi \big( Y_{q^{1/2}} \matje{1}{1}{0}{1} Y_{q^{1/2}}^{-1} \big) t^{-1} \\
& = t \Phi \matje{1}{q}{0}{1} t^{-1} = x^q .
\end{split}
\end{equation}
Recall that $B_2$ denotes the upper triangular Borel subgroup of $\SL_2 (\C)$. In the 
flag variety of $M^\circ$ we have the subvarieties $\mathcal B^x_{M^\circ}$ and 
$\mathcal B^{\Phi (B_2)}_{M^\circ}$ of Borel subgroups containing $x$ and $\Phi (B_2)$, 
respectively. Similarly the flag variety of $H^\circ$ has subvarieties 
$\mathcal B^{t,x}_{H^\circ} ,\; \mathcal B^{t_q,x}_{H^\circ}$ and 
\[
\mathcal B^{t, \Phi (B_2)}_{H^\circ} = \mathcal B^{t_q, \Phi (B_2)}_{H^\circ} .
\]
Notice that $\Phi (\mathbf I_F)$ lies in every Borel subgroup of $H^\circ$, because it
is contained in $\Z(H^\circ)$. We abbreviate $\Z_H (\Phi) = 
\Z_H (\Phi (\mathbf W_F \times \SL_2 (\C)))$ and similarly for other groups.

\begin{prop}\label{prop:S.1}
\begin{enumerate}
\item The inclusion maps
\[
\begin{array}{ccccccc}
 & & \Z_{M^\circ}(\Phi) & \to & \Z_{M^\circ}(\Phi (B_2)) & 
\to & \Z_{M^\circ}(x) , \\
\Z_H (t_q,x) & \leftarrow & \Z_H (\Phi) & \to & \Z_H (t, \Phi (B_2)) & 
\to & \Z_H (t,x) , 
\end{array}
\]
are homotopy equivalences. In particular they induce isomorphisms between the 
respective component groups.
\item The inclusions $\mathcal B^{\Phi (B_2)}_{M^\circ} \to \mathcal B^x_{M^\circ}$ and 
$\mathcal B^{t_q,x}_{H^\circ} \leftarrow \mathcal B^{t,\Phi (B_2)}_{H^\circ} 
\to \mathcal B^{t,x}_{H^\circ}$ are homotopy equivalences.
\end{enumerate}
\end{prop}
\begin{proof} 
It suffices to consider the statements for $H$ and $t_q$, 
since the others can be proven in the same way.\\
(1) Our proof uses some elementary observations from \cite[\S 4.3]{R}. 
There is a Levi decomposition
\[
\Z_{H^\circ} (x) = \Z_{H^\circ} (\Phi (\SL_2 (\C))) U_x 
\]
with $\Z_{H^\circ} (\Phi (\SL_2 (\C))) = \Z_{H^\circ} (\Phi (B_2))$ reductive and $U_x$ unipotent. 
Since $t_q \in \Nor_{H^\circ} (\Phi (\SL_2 (\C)))$ and $\Z_H (x^q) = \Z_H (x)$, conjugation by $t_q$ 
preserves this decomposition. Therefore
\begin{equation}\label{eq:S.1}
\Z_{H^\circ} (t_q,x) = \Z_{H^\circ} (\Phi) \Z_{U_x}(t_q) = 
\Z_{H^\circ} (t_q, \Phi (B_2)) \Z_{U_x}(t_q). 
\end{equation}
We note that 
\[
\Z_{U_x}(t_q) \cap \Z_{H^\circ} (t_q, \Phi (B_2)) \subset U_x \cap \Z_{H^\circ} (\Phi (B_2)) = 1
\]
and that $\Z_{U_x}(t_q) \subset U_x$ is contractible, because it is a unipotent complex group. 
It follows that
\begin{equation}\label{eq:S.10}
\Z_{H^\circ} (\Phi) = \Z_{H^\circ} (t_q, \Phi (B_2)) \to \Z_{H^\circ} (t_q,x) 
\end{equation}
is a homotopy equivalence. If we want to replace $H^\circ$ by $H$, we find
\[
\Z_H (\Phi) / \Z_{H^\circ}(\Phi) = 
\{ h H^\circ \in \pi_0 (H) \mid h \Phi h^{-1} \in \mathrm{Ad}(H^\circ) \Phi \} ,
\]
and similarly with $(t_q, \Phi (B_2))$ or $(t_q,x)$ instead of $\Phi$. 

Let us have a closer look
at the $H^\circ$-conjugacy classes of these objects. Given any $\Phi$, we obviously know what 
$t_q$ and $x$ are. Conversely, suppose that $t_q$ and $x$ are given. We apply a refinement of
the Jacobson--Morozov theorem due to Kazhdan and Lusztig. According to \cite[\S 2.3]{KL} there
exist homomorphisms $\Phi : \mathbf W_F \times \SL_2 (\C) \to G$ as above, which return $t_q$ 
and $x$ in the prescribed way. Moreover all such homomorphisms are conjugate under 
$\Z_{H^\circ} (t_q,x)$, see \cite[\S 2.3.h]{KL} or Section 19. So from $(t_q,x)$ we can 
reconstruct the Ad$(H^\circ)$-orbit of $\Phi$, and this 
gives bijections between $H^\circ$-conjugacy classes of $\Phi ,\; (t_q, \Phi (B_2))$ and $(t_q,x)$. 
Since these bijections clearly are $\pi_0 (H)$-equivariant, we deduce 
\begin{equation}\label{eq:S.11}
\Z_H (\Phi) / \Z_{H^\circ}(\Phi) = \Z_H (t_q, \Phi (B_2)) / \Z_{H^\circ}(t_q, \Phi (B_2)) = 
\Z_H (t_q,x) / \Z_{H^\circ}(t_q,x) .
\end{equation}
Equations \eqref{eq:S.10} and \eqref{eq:S.11} imply that
\[
\Z_H (\Phi) = \Z_H (t_q, \Phi (B_2)) \to \Z_H (t_q,x) 
\]
is also a homotopy equivalence. \\
(2) By the aforementioned result \cite[\S 2.3.h]{KL} 
\begin{equation}\label{eq:S.2}
\Z_{H^\circ} (t_q,x) \cdot \mathcal B^{t_q, \Phi (B_2)}_{H^\circ} = 
\mathcal B^{t_q,x}_{H^\circ} .
\end{equation}
On the other hand, by \eqref{eq:S.1}
\begin{equation}\label{eq:S.3}
\Z_{H^\circ} (t_q,x) \cdot \mathcal B^{t_q, \Phi (B_2)}_{H^\circ} = 
\Z_{U_x}(t_q) \Z_H (t_q, \Phi (B_2)) \cdot \mathcal B^{t_q, \Phi (B_2)}_{H^\circ} = 
\Z_{U_x}(t_q) \cdot \mathcal B^{t_q, \Phi (B_2)}_{H^\circ} .
\end{equation}
For any $B \in \mathcal B^{t_q, \Phi (B_2)}_{H^\circ}$ and $u \in \Z_{U_x}(t_q)$ it is clear that
\[
u \cdot B \in \mathcal B^{t_q, \Phi (B_2)}_{H^\circ} \; \Longleftrightarrow \;
\Phi (B_2) \subset u B u^{-1} \; \Longleftrightarrow \;
u^{-1} \Phi (B_2) u \subset B .
\]
Furthermore, since $\Phi (B_2) \subset B$ is generated by $x$ and 
$\{ \Phi \matje{\alpha}{0}{0}{\alpha^{-1}} \mid \alpha \in \C^\times \}$, 
the right hand side is equivalent to 
\[
u^{-1} \Phi \matje{\alpha}{0}{0}{\alpha^{-1}} u \in B \quad \forall \alpha \in \C^\times .
\]
In Lie algebra terms this can be reformulated as 
\[
\mathrm{Ad}_{u^{-1}} (d \Phi \matje{\alpha}{0}{0}{-\alpha}) \in \mathrm{Lie} \, B
\quad \forall \alpha \in \C . 
\]
Because $u$ is unipotent, this happens if and only if
\[
\mathrm{Ad}_{u^\lambda} (d \Phi \matje{\alpha}{0}{0}{-\alpha}) \in \mathrm{Lie} \, B
\quad \forall \lambda,\alpha \in \C .  
\]
By the reverse chain of arguments the last statement is equivalent with
\[
u^\lambda \cdot B \in \mathcal B^{t_q, \Phi (B_2)}_{H^\circ} \quad \forall \lambda \in \C .
\]
Thus $\{ u \in \Z_{U_x}(t_q) \mid u \cdot B \in \mathcal B^{t_q, \Phi (B_2)}_{H^\circ} \}$ 
is contractible for all $B \in \mathcal B^{t_q, \Phi (B_2)}_{H^\circ}$, and we already knew 
that $\Z_{U_x}(t_q)$ is contractible. Together with \eqref{eq:S.2} and \eqref{eq:S.3} these 
imply that $\mathcal B^{t_q, \Phi (B_2)}_{H^\circ} \to \mathcal B^{t_q,x}_{H^\circ}$ 
is a homotopy equivalence.
\end{proof}

For the affine Springer correspondence we will need more precise information on the
relation between the varieties for $G$, for $H$ and for $M^\circ$.

\begin{prop}\label{prop:UP}
\begin{enumerate}
\item The variety $\mathcal B^{t,x}_{H^\circ}$ is isomorphic to $[\mathcal W^{H^\circ} : 
\mathcal W^{M^\circ}]$ copies of $\mathcal B^x_{M^\circ}$, and 
$\mathcal B^{t, \Phi (B_2)}_{H^\circ}$ is isomorphic to the same number of copies of 
$\mathcal B^{\Phi (B_2)}_{M^\circ}$.
\item The group $\Z_{H^\circ}(t,x) / \Z_{M^\circ}(x)$ permutes these two
sets of copies freely.
\item The variety $\mathcal B_G^{\Phi (\mathbf W_F \times B_2)}$ is isomorphic to
$[\mathcal W^G : \mathcal W^{H^\circ}]$ copies of $\mathcal B^{t, \Phi (B_2)}_{H^\circ}$. 
The group $\Z_G (\Phi) / \Z_{H^\circ}(\Phi)$ permutes these copies freely.
\end{enumerate}
\end{prop}
\begin{proof}
(1) Let $A$ be a subgroup of $T$ such that $M^\circ = \Z_{H^\circ} (A)^\circ$ and let 
$\mathcal B_{H^\circ}^A$ denote the variety of all Borel subgroups of $H^\circ$ 
which contain $A$. With an adaptation of \cite[p.471]{CG} we will prove that, for any 
$B \in \mathcal B_{H^\circ}^A ,\; B \cap M^0$ is a Borel subgroup of $M^0$. 

Since $B \cap M^\circ \subset B$ is solvable, it suffices to show that its Lie algebra
is a Borel subalgebra of Lie $M^\circ$. Write Lie $T = \mathfrak t$ and let 
\[
\text{Lie } H^\circ = \mathfrak n \oplus \mathfrak t \oplus \mathfrak n_- 
\]
be the triangular decomposition, where Lie $B = \mathfrak n \oplus \mathfrak t$.
Since $A \subset B$, it preserves this decomposition and
\begin{align*}
& \text{Lie } M^\circ = (\text{Lie } H )^A = 
\mathfrak n^A \oplus \mathfrak t \oplus \mathfrak n_-^A ,\\
& \text{Lie } B \cap M^\circ = \text{Lie } B^A = \mathfrak n^A \oplus \mathfrak t .
\end{align*}
The latter is indeed a Borel subalgebra of Lie $M^\circ$.
Thus there is a canonical map 
\begin{equation} \label{eqn:(7)} 
\mathcal B_{H^\circ}^A \to \Flag \, M^0, \quad B \mapsto B \cap M^0 .
\end{equation}
The group $M$ acts by conjugation on $\mathcal B_{H^\circ}^A$ and \eqref{eqn:(7)} 
clearly is $M$-equivariant. By \cite[p. 471]{CG} the $M^\circ$-orbits form a partition
\begin{equation}\label{eq:componentsBA}
\mathcal B_{H^\circ}^A = \mathcal B_1 \sqcup \mathcal B_2 \sqcup \cdots \sqcup \mathcal B_m .
\end{equation}
At the same time these orbits are the connected components of $\mathcal B_{H^\circ}^A$ 
and the irreducible components of the projective variety $\mathcal B_{H^\circ}^A$. 
The argument from \cite[p. 471]{CG} also shows that \eqref{eqn:(7)}, restricted to any one 
of these orbits, is a bijection from the $M^0$-orbit onto Flag $M^0$. 

The number of components $m$ can be determined as in the proof of
\cite[Corollary 3.12.a]{Ste1965}. The collection of Borel subgroups of $M^\circ$ that 
contain the maximal torus $T$ is in bijection with the Weyl group $\mathcal W^{M^\circ}$. 
Retracting via \eqref{eqn:(7)}, we find that every component $\mathcal B_i$ has precisely 
$|\mathcal W^{M^\circ}|$ elements that contain $T$. On the other hand, since $A \subset T ,\;
\mathcal B_{H^\circ}^A$ has $|\mathcal W^{H^\circ}|$ elements that contain $T$, so
\[
m = [\mathcal W^{H^\circ} : \mathcal W^{M^\circ}] .
\]
To obtain our desired isomorphisms of varieties, we let $A$ be the group generated by
$t$ and we restrict $\mathcal B_i \to \text{Flag } M^\circ$ to Borel subgroups
that contain $t,x$ (respectively $t,\Phi (B_2)$). \\ 
(2) By Proposition \ref{prop:S.1} 
\[
\Z_{H^\circ}(t,x) / \Z_{M^\circ}(x) \cong \Z_{H^\circ}(t, \Phi (B_2)) / \Z_{M^\circ}(\Phi (B_2)) .
\]
Since the former is a subgroup of $M / M^\circ$ and the copies under
consideration are in $M$-equivariant bijection with the components \eqref{eq:componentsBA},
it suffices to show that $M / M^\circ$ permutes these components freely.
Pick $B,B'$ in the same component $\mathcal B_i$ and assume that $B' = h B h^{-1}$ 
for some $h \in M$. Since $\mathcal B_i$ is $M^\circ$-equivariantly isomorphic 
to the flag variety of $M^\circ$ we can find $m \in M^\circ$ such that $B' = m^{-1} B m$.
Then $m h$ normalizes $B$, so $m h \in B$. As $B$ is connected, this implies
$m h \in M^\circ$ and $h \in M^\circ$. \\
(3) Apply the proofs of parts 1 and 2 with $A = \Phi (\mathbf I_F) ,\; G$ in the role of 
$H^\circ ,\; H^\circ$ in the role of $M^\circ$ and $t \Phi (B_2)$ in the role of $x$.
\end{proof}

\section{Comparison of different parameters}
\label{sec:comppar}

In the following sections we will make use of several different but related 
kinds of parameters. 
\vspace{2mm}

\noindent \textbf{Kazhdan--Lusztig--Reeder parameters (KLR parameters)}\\  
For a Langlands parameter as in \eqref{eqn:Phi}, the variety of Borel subgroups \\
$\mathcal B_G^{\Phi (\mathbf W_F \times B_2)}$ is nonempty, and the centralizer
$\Z_G (\Phi)$ of the image of $\Phi$ acts on it. Hence the group of components
$\pi_0 (\Z_G (\Phi))$ acts on the homology $H_* \big( \mathcal B_G^{\Phi (\mathbf W_F 
\times B_2)} ,\C \big)$. We call an irreducible representation $\rho$ of 
$\pi_0 (\Z_G (\Phi))$ geometric if it appears in $H_* \big( \mathcal B_G^{\Phi 
(\mathbf W_F \times B_2)} ,\C \big)$. We define a Kazhdan--Lusztig--Reeder parameter 
for $G$ to be a such pair $(\Phi,\rho)$. The group $G$ acts on these parameters by 
\begin{equation}\label{eq:defKLRparameter}
g \cdot (\Phi,\rho) = (g \Phi g^{-1}, \rho \circ \mathrm{Ad}_g^{-1}) 
\end{equation}
and we denote the corresponding equivalence class by $[\Phi,\rho ]_G$.
\vspace{2mm}

\noindent \textbf{Affine Springer parameters}\\ 
As before, suppose that $t \in G$ is semisimple and that $x \in \Z_G (t)$ is unipotent. 
Then $\Z_G (t,x)$ acts on $\mathcal B_G^{t,x}$ and $\pi_0 (\Z_G (t,x))$ acts on the
homology of this variety. In this setting we say that $\rho_1 \in \Irr \big( 
\pi_0 (\Z_G (t,x)) \big)$ is geometric if it appears in $H_{\top}( \mathcal B_G^{t,x} ,\C )$,
where $\top$ refers to highest degree in which the homology is nonzero, the real
dimension of $\mathcal B_G^{t,x}$.
We call such triples $(t,x,\rho_1)$ affine Springer parameters for $G$,
because they appear naturally in the representation theory of the affine Weyl group
associated to $G$. The group $G$ acts on such parameters by conjugation, and we 
denote the conjugacy classes by $[t,x,\rho_1]_G$.
\vspace{2mm}

\noindent \textbf{Kazhdan--Lusztig triples}\\ 
Next we consider a unipotent element $x \in G$ and a semisimple element $t_q \in G$ 
such that $t_q x t_q^{-1} = x^q$. As above, $\Z_G (t_q,x)$ acts on the variety
$\mathcal B_G^{t_q,x}$ and we call $\rho_q \in \Irr \big( \pi_0 (\Z_G (t_q,x)) \big)$
geometric if it appears in $H_* \big( \mathcal B_G^{t_q,x} ,\C \big)$. We refer to
triples $(t_q,x,\rho_q)$ of this kind as Kazhdan--Lusztig triples for $G$.
Again they are endowed with an obvious $G$-action and we denote the equivalence classes
by $[t_q,x,\rho_q]_G$.
\vspace{2mm}

We note that in all cases the representations of the component groups stem from
the action of $G$ on a variety of Borel subgroups. The centre of $G$ acts trivially
on such a variety, so in all three above cases an irreducible representation of the
appropriate component group can only be if all elements coming from $\Cent (G)$ 
act trivially.

In \cite{KL,R} there are some indications that these three kinds of parameters are
essentially equivalent. Proposition \ref{prop:S.1} allows us to make this precise
in the necessary generality.

\begin{lem}\label{lem:compareParameters}
Let $\fs$ be a Bernstein component in the principal series, associate 
$c^\fs \colon \fo_F^\times \to T$ 
to it as in Lemma \ref{lem:cBernstein} and write $H = \Z_G (c^\fs(\fo_F^\times))$.
There are natural bijections between $H^\circ$-equivalence classes of:
\begin{itemize}
\item Kazhdan--Lusztig--Reeder parameters for $G$ with 
$\Phi \big|_{\fo_F^\times} = c^\fs$ and\\ 
$\Phi (\varpi_F) \in H^\circ$;
\item affine Springer parameters for $H^\circ$;\
\item Kazhdan--Lusztig triples for $H^\circ$.
\end{itemize}
\end{lem}
\begin{proof}
Since $\SL_2 (\C)$ is connected and commutes with $\fo_F^\times$, its image under $\Phi$ must
be contained in the connected component of $H$. Therefore KLR-parameters with these 
properties are in canonical bijection with KLR parameters for $H^\circ$ and it suffices to 
consider the case $H^\circ = G$.

As in \eqref{eqn:Phi} and \eqref{eq:S.12}, any KLR parameter gives rise to the 
ingredients $t,x$ and $t_q$ for the other two kinds of parameters. As we discussed
after \eqref{eqn:Phi}, the pair $(t,x)$ is enough to recover the conjugacy class
of $\Phi$. A refined version of the Jacobson--Morozov theorem says that the same
goes for the pair $(t_q,x)$, see \cite[\S 2.4]{KL} or \cite[Section 4.2]{R}.

To complete $\Phi, (t,x)$ or $(t_q,x)$ to a parameter of the appropriate kind,
we must add an irreducible representation $\rho ,\rho_1$ or $\rho_q$.
For the affine Springer parameters it does not matter whether we 
consider the total homology or only the homology in top degree. Indeed, it follows 
from Propositions \ref{prop:S.1} and \ref{prop:UP} and \cite[bottom of page~296 and 
Remark 6.5]{Shoji} that any irreducible representation $\rho_1$ which appears in 
$H_* \big( \mathcal B_G^{t,x} ,\C \big)$, already appears 
in the top homology of this variety.

This and Proposition \ref{prop:S.1} show that there is a natural correspondence 
between the possible ingredients $\rho,\rho_1$ and $\rho_q$. 
\end{proof}

\section{The affine Springer correspondence}
\label{sec:affSpringer}

An interesting instance of Section \ref{sec:celldec} arises when $M$ is the 
centralizer of a semisimple element $t$ in a connected reductive 
complex group $G$. As before we assume that $t$ lies in a maximal torus $T$ of 
$G$ and we write $\mathcal W^G = W(G,T)$. By Lemma \ref{lem:centrals}
\begin{equation}\label{eq:S.8}
\mathcal W^M := \Nor_M (T) / \Z_M (T) \cong \cW^{M^\circ} \rtimes \pi_0 (M) 
\end{equation}
is the stabilizer of $t$ in $\mathcal W^G$, so the role of $\Gamma$ is played by the 
component group $\pi_0 (M)$. In contrast to the setup in Section \ref{sec:celldec},
it is possible that some elements of $\pi_0 (M) \setminus \{1\}$ fix $W$ pointwise.
This poses no problems however, as such elements never act trivially on $T$. 
For later use we record the following consequence of \eqref{eq:S.9}:
\begin{equation}\label{eq:23.7}
\pi_0 (M)_{\tau (x,\rho)} \cong \big( \Z_M (x) / \Z_{M^\circ}(x) \big)_\rho .
\end{equation}
Recall from Section \ref{sec:extquot} that 
\begin{align*}
& \widetilde{T}_2: = \{(t,\sigma) \,:\, t \in T, \sigma \in \Irr(\cW_t^G)\}, \\
&(T\q \cW^G)_2: = \widetilde{T}_2/\cW^G.
\end{align*}
We note that the rational characters of the complex torus $T$ span the regular 
functions on the complex variety $T$:
\[
\mathcal{O}(T) = \C [X^*(T)].
\]
From \eqref{EXT1}, \eqref{EXT2}, Lemma \ref{lem:Clifford_abelian} and 
Proposition \ref{prop:S.2} we infer the 
following rough form of the extended Springer correspondence for the affine 
Weyl group $X^*(T) \rtimes \cW^G$. 

\begin{thm} 
There are bijections
\[
(T\q \cW^G )_2 \simeq \Irr \, (X^*(T) \rtimes \cW^G ) \simeq 
\{(t,\tau(x, \varrho) \rtimes \psi) \}/\cW^G
\]
with $t \in T, \tau(x, \varrho) \in
\Irr \,\cW^{M^0}, \psi \in \Irr(\pi_0(M)_{\tau (x,\varrho)})$.
\end{thm}

Now we recall the geometric realization of irreducible representations of 
$X^* (T) \rtimes \mathcal W^G$ by Kato \cite{Kat}. 
For a unipotent element $x \in M^\circ$ let $\mathcal B^{t,x}_G$ be the variety of Borel 
subgroups of $G$ containing $t$ and $x$. Fix a Borel subgroup $B$ of $G$ containing $T$ 
and let $\theta_{G,B} : \mathcal B^{t,x}_G \to T$ be the morphism defined by
\begin{equation}\label{eq:thetaGB}
\theta_{G,B}(B') = g^{-1} t g \text{ if } B' = g B g^{-1} \text{ and } t \in g T g^{-1}.
\end{equation}
The image of $\theta_{G,B}$ is $\mathcal W^G t$, the map is constant on the irreducible
components of $\mathcal B^{t,x}_G$ and it gives rise to an action of $X^* (T)$ on the 
homology of $\mathcal B^{t,x}_G$. Furthermore $\Q [ \mathcal W^G ] \cong H (\mathcal Z_G)$ 
acts on $H_{d(x)}(\mathcal B^{t,x}_G,\C)$ via the convolution product in Borel--Moore 
homology, as described in \eqref{eq:S.5}. Both actions commute with the action of 
$\Z_G (t,x)$ induced by conjugation of Borel subgroups. By homotopy invariance, the latter
action factors through $\pi_0 (\Z_G (t,x))$.

Let $\rho_1 \in \Irr \big( \pi_0 (\Z_G (t,x)) \big)$. By \cite[Theorem 4.1]{Kat} the 
$X^* (T) \rtimes \mathcal W^G$-module 
\begin{equation}\label{eq:KatoMod}
\tau (t,x,\rho_1) := 
\mathrm{Hom}_{\pi_0 (\Z_G (t,x))} \big( \rho_1, H_{d(x)}(\mathcal B^{t,x}_G,\C) \big)  
\end{equation}
is either irreducible or zero. Moreover every irreducible representation of 
$X^* (T) \rtimes \mathcal W^G$ is obtained is in this way, and the data $(t,x,\rho_1)$
are unique up to $G$-conjugacy. This generalizes the Springer correspondence for
finite Weyl groups, which can be recovered by considering the representations on 
which $X^* (T)$ acts trivially.

Propositions \ref{prop:S.2} and \ref{prop:UP} shine some new light on this:

\begin{thm}\label{thm:S.3}
\begin{enumerate}
\item There are bijections between the following sets:
\begin{itemize} 
\item $\mathrm{Irr}(X^* (T) \rtimes \mathcal W^G) = 
\mathrm{Irr}(\mathcal O (T) \rtimes \mathcal W^G)$;
\item $(T // \mathcal W^G )_2  = \big\{ (t, \tilde \tau ) \mid t \in T , 
\tilde \tau \in \mathrm{Irr}(\mathcal W^M) \big\} / \mathcal W^G$;
\item $\big\{ (t, \tau, \sigma) \mid t \in T , \tau \in \mathrm{Irr}(\mathcal W^{M^\circ}), 
\sigma \in \mathrm{Irr}(\pi_0 (M)_\tau) \big\} / \mathcal W^G$;
\item $\big\{ (t,x,\rho,\sigma) \mid t \in T , x \in M^\circ \, \mathrm{unipotent} , 
\rho \in \mathrm{Irr} \big( \pi_0 (\Z_{M^\circ}(x)) \big) \\ \mathrm{geometric}, 
\sigma \in \mathrm{Irr} \big( \pi_0 (M)_{\tau (x,\rho)} \big) \big\} / G$;
\item $\big\{ (t,x,\rho_1) \mid t \in T, x \in M^\circ \, \mathrm{unipotent}, 
\rho_1 \in \mathrm{Irr} \big( \pi_0 (\Z_G (t,x)) \big) \\ \mathrm{geometric} \big\} / G$.
\end{itemize}
Here a representation of $\pi_0 (\Z_{M^\circ}(x))$ (or $\pi_0 (\Z_G (t,x))$) is 
called geometric if it appears in $H_{d(x)}(\mathcal B^x_{M^\circ} ,\C)$ (respectively 
$H_{d(x)}(\mathcal B^{t,x}_G,\C)$). Apart from the third and fourth sets, these bijections
are natural.
\item The $X^*(T) \rtimes \mathcal W^G$-representation corresponding to $(t,x,\rho_1)$ 
via these bijections is Kato's module \eqref{eq:KatoMod}.
\end{enumerate}
\end{thm}

We remark that in the fourth and fifth sets it would be more natural to allow $t$ to be any
semisimple element of $G$. In fact that would give the affine Springer parameters from
Lemma \ref{lem:compareParameters}.
Clearly $G$ acts on the set of such more general parameters
$(t,x,\rho,\sigma)$ or $(t,x,\rho_1)$, which gives equivalence relations $/G$. The two above 
$/G$ refer to the restrictions of these equivalence relations to parameters with $t \in T$.

\begin{proof} 
(1) Recall that the isotropy group of $t$ in $\mathcal W^G$ is 
\[
\mathcal W^G_t = \mathcal W^M = \mathcal W^{M^\circ} \rtimes \pi_0 (M) .
\]
Hence the bijection between the first two sets is an instance of Clifford theory, see 
Lemma \ref{lem:Clifford_algebras}. The second and third sets are 
in bijection by Proposition \ref{prop:S.2}.
The Springer correspondence for $\mathcal W^{M^\circ}$ provides the bijection with the
fourth collection. To establish a bijection with the fifth collection, we first 
observe that
\begin{equation}\label{eq:N.1}
\begin{split}
\pi_0 (\Z_G (t,x)) = \pi_0 (\Z_M (x)) \cong 
\pi_0 \big( \Z_{M^\circ} (x) \rtimes \pi_0 (M)_{[x]_{M^\circ}} \big) \\
= \pi_0 (\Z_{M^\circ}(x)) \rtimes \pi_0 (M)_{[x]_{M^\circ}} .
\end{split}
\end{equation}
Furthermore $\pi_0 (M)_{\tau (x,\rho)} = \pi_0 (M)_{[x,\rho]_{M^\circ}}$ by \eqref{eq:S.9}. 
From that and Proposition \ref{prop:S.2} it follows that every irreducible representation 
of \eqref{eq:N.1} is of the form $\rho \rtimes \sigma$ (see Notation \ref{not:rtimes}), 
with $\rho$ and $\sigma$ as in the fourth set. By Proposition \ref{prop:UP}
\begin{equation}\label{eq:S.26}
H_* (\mathcal B^{t,x}_G, \C) \cong H_* (\mathcal B^x_{M^\circ} ,\C) \otimes 
\C [\Z_G (t,x) / \Z_{M^\circ}(x)] \otimes \C^{ [\mathcal W^G : \mathcal W^G_t]}
\end{equation}
as $\Z_G (t,x)$-representations. By \cite[\S 3.1]{R} 
\[
\Z_G (t,x) / \Z_{M^\circ}(x) \cong \pi_0 (M)_{[x]_{M^\circ}}
\] 
is abelian. Hence 
$\mathrm{Ind}_{\pi_0 (M)_{[x,\rho]_{M^\circ}}}^{\pi_0 (M)_{[x]_{M^\circ}}} (\sigma)$ 
appears exactly once in the regular representation of this group and
\begin{multline}\label{eq:S.34}
\mathrm{Hom}_{\pi_0 (\Z_G(t,x))} \big( \rho \rtimes \sigma ,H_{d(x)} 
(\mathcal B^{t,x}_G, \C) \big) \cong \\
\mathrm{Hom}_{\pi_0 (\Z_{M^\circ}(x))} \big( \rho ,H_{d(x)} (\mathcal B^x_{M^\circ}, 
\C) \big) \rtimes \sigma \otimes \C^{[\mathcal W^G : \mathcal W^G_t]} .
\end{multline}
In particular we see that $\rho$ is geometric if and only if $\rho \rtimes \sigma$ is 
geometric, which establishes the final bijection. Now the resulting bijection between
the second and fifth sets is natural by Theorem \ref{thm:SpringerExtended}.\\
(2) The $X^* (T) \rtimes \mathcal W^G$-representation constructed from 
$(t,x,\rho \rtimes \sigma)$ by means of our bijections is
\begin{equation}\label{eq:S.28}
\mathrm{Ind}_{X^*(T) \rtimes \mathcal W^G_t}^{X^*(T) \rtimes \mathcal W^G} 
\big( \mathrm{Hom}_{\pi_0 (\Z_{M^\circ}(x))} 
\big( \rho ,H_{d(x)} (\mathcal B^x_{M^\circ}, \C) \big) \rtimes \sigma \big) .
\end{equation}
On the other hand, by \cite[Proposition 6.2]{Kat}
\begin{multline} \label{eq:S.35}
H_* (\mathcal B^{t,x}_G,\C) \cong \mathrm{Ind}_{X^*(T) \rtimes 
\mathcal W^{M^\circ}}^{X^*(T) \rtimes \mathcal W^G} (H_* (\mathcal B^x_{M^\circ} ,\C)) \\
\cong \mathrm{Ind}_{X^*(T) \rtimes \mathcal W^G_t}^{X^*(T) \rtimes \mathcal W^G} 
\big( H_* (\mathcal B^x_{M^\circ} ,\C) \otimes \C [\Z_G (t,x) / \Z_{M^\circ}(x)] \big)
\end{multline}
as $\Z_G (t,x) \times X^* (T) \rtimes \mathcal W^G$-representations. Together with
the proof of part 1 this shows that $\tau (t,x,\rho \rtimes \sigma)$
is isomorphic to \eqref{eq:S.28}. 
\end{proof}

We can extract a little more from the above proof. Recall that $\cO_x$ 
denotes the conjugacy class of $x$ in $M$. Let us agree that the 
affine Springer parameters with a fixed $t \in T$ are partially ordered by
\[
(t,x,\rho_1) < (t,x',\rho'_1) \quad \text{when} \quad
\overline{\cO_x} \subsetneq \overline{\cO_{x'}} .
\]
\begin{lem}\label{lem:S.7}
There exist multiplicities $m_{t,x,\rho_1,x',\rho'_1} \in {\mathbb Z}_{\geq 0}$ such that
\begin{multline*}
\mathrm{Hom}_{\pi_0 (\Z_G(t,x))} \big( \rho_1 ,
H_* (\mathcal B^{t,x}_G, \C) \big) \cong \\
\tau (t,x,\rho_1) \oplus \bigoplus_{(t,x',\rho'_1) > (t,x,\rho_1)}
m_{t,x,\rho_1,x',\rho'_1} \,\tau (t,x',\rho'_1) .
\end{multline*}
\end{lem}
\begin{proof}
It follows from \eqref{eq:S.35}, \eqref{eq:S.26} and \eqref{eq:S.34} that
\begin{multline}\label{eq:S.36}
\mathrm{Hom}_{\pi_0 (\Z_G(t,x))} \big( \rho \rtimes \sigma ,
H_* (\mathcal B^{t,x}_G, \C) \big) \cong \\
\mathrm{Ind}_{X^*(T) \rtimes \mathcal W^G_t}^{X^*(T) \rtimes \mathcal W^G} 
\mathrm{Ind}_{\cW^{M^\circ} \rtimes \pi_0 (M)_{[x,\rho]_{M^\circ}}}^{\cW^G_t}
\big( \mathrm{Hom}_{\pi_0 (\Z_{M^\circ}(x))} 
\big( \rho ,H_{d(x)} (\mathcal B^x_{M^\circ}, \C) \big) \otimes \sigma \big) .
\end{multline}
The functor $\mathrm{Ind}_{X^*(T) \rtimes \mathcal W^G_t}^{X^*(T) \rtimes \mathcal W^G}$
provides an equivalence between the categories 
\begin{itemize}
\item $X^*(T) \rtimes \mathcal W^G_t$-representations with $\cO (T)^{\cW^G_t}$-character $t$;
\item $X^*(T) \rtimes \mathcal W^G$-representations with $\cO (T)^{\cW^G}$-character $\cW^G t$.
\end{itemize}
Therefore we may apply Lemma \ref{lem:S.6} to the right hand side of \eqref{eq:S.36},
which produces the required formula.
\end{proof}

Let us have a look at the representations with an affine Springer parameter of
the form $(t,x=1,\rho_1 = \mathrm{triv})$. Equivalently, the fourth parameter in 
Theorem \ref{thm:S.3} is $(t,x=1,\rho = \mathrm{triv},\sigma = \mathrm{triv})$.
The $\mathcal W^{M^\circ}$-representation with Springer parameter $(x=1,\rho = 
\mathrm{triv})$ is the trivial representation, so $(x=1,\rho = \mathrm{triv},
\sigma = \mathrm{triv})$ corresponds to the trivial representation of 
$\mathcal W^G_t$. With \eqref{eq:S.28} we conclude that the $X^* (T) \rtimes 
\mathcal W^G$-representation with affine Springer parameter $(t,1,\mathrm{triv})$ is
\begin{equation} \label{eq:trivWaff}
\tau (t,1,\triv) = \mathrm{Ind}_{X^*(T) \rtimes \mathcal W^G_t}^{X^*(T) 
\rtimes \mathcal W^G} \big( \mathrm{triv}_{\mathcal W^G_t} \big) .
\end{equation}
Notice that this is the only irreducible $X^* (T) \rtimes \mathcal W^G$-representation 
with an $X^*(T)$-weight $t$ and nonzero $\mathcal W^G$-fixed vectors.

\section{Geometric representations of affine Hecke algebras}
\label{sec:repAHA}

Let $G$ be a connected reductive complex group, $B$ a Borel subgroup and $T$ 
a maximal torus of $G$ contained in $B$. Let $\mathcal H (G)$ be the affine Hecke 
algebra with the same based root datum as $(G,B,T)$ and with a parameter 
$q \in \C^\times$ which is not a root of unity.

As we will have to deal with disconnected reductive groups, we include
some additional automorphisms in the picture. In every root subgroup $U_\alpha$ with 
$\alpha \in \Delta (B,T)$ we pick a nontrivial element $u_\alpha$. Let $\Gamma$ be a 
finite group of automorphisms of $(G,T,(u_\alpha)_{\alpha \in \Delta (B,T)})$. Since $G$ need 
not be semisimple, it is possible that some elements of $\Gamma$ fix the entire root system 
of $(G,T)$. Notice that $\Gamma$ acts on the Weyl group $\mathcal W^G = W(G,T)$ and on
$X^* (T)$ because it stabilizes $T$. 
Furthermore $\Gamma$ acts on the standard basis of $\cH (G)$ by
\[
\gamma (T_w) = T_{\gamma (w)}, \text{ where } \gamma \in \Gamma, w \in X^* (T) \rtimes \cW^G .
\]
Since $\Gamma$ stabilizes $B$, it determines an algebra automorphism of $\cH (G)$.
We form the crossed product algebra $\mathcal H (G) \rtimes \Gamma$ with respect to this
$\Gamma$-action. It follows from the Bernstein presentation of $\cH (G)$ \cite[\S 3]{LuGrad} 
that $Z (\cH (G) \rtimes \Gamma) \cong \mathcal O (T / \cW^G \rtimes \Gamma)$.

We define a Kazhdan--Lusztig triple 
for $\mathcal H (G) \rtimes \Gamma$ to be a triple $(t_q,x,\rho)$ such that:
\begin{itemize}
\item $t_q \in G$ is semisimple, $x \in G$ is unipotent and $t_q x t_q^{-1} = x^q$;
\item $\rho$ is an irreducible representation of the component group \\
$\pi_0 (\Z_{G \rtimes \Gamma} (t_q,x))$, such that every irreducible subrepresentation
of the restriction of $\rho$ to $\pi_0 (\Z_G (t_q,x))$ appears in $H_* (\mathcal B^{t_q,x},\C)$.
\end{itemize}
The group $G \rtimes \Gamma$ acts on such triples by conjugation, and we denote the
conjugacy class of a triples by $[t_q,x,\rho]_{G \rtimes \Gamma}$. 
Now we generalize \cite[Theorem 7.12]{KL} and \cite[Theorem 3.5.4]{R}:

\begin{thm}\label{thm:S.5}
There exists a natural bijection between $\mathrm{Irr}(\mathcal H (G) \rtimes \Gamma)$ 
and $G \rtimes \Gamma$-conjugacy classes of Kazhdan--Lusztig triples. 

The $\cH (G) \rtimes \Gamma$-module $\pi (t_q,x,\rho)$ has central character 
$(\cW^G \rtimes \Gamma) t_q$ and is the unique irreducible quotient of the 
$\mathcal H (G)\rtimes \Gamma$-module 
\[
\mathrm{Hom}_{\pi_0 (\Z_{G \rtimes \Gamma} (t_q,x))} 
\big( \rho, H_* (\mathcal B^{t_q,x},\C) \otimes \C[\Gamma] \big) .
\]
\end{thm}
\begin{proof}
First we recall the geometric constructions of $\mathcal H (G)$-modules by Kazhdan, Lusztig
and Reeder, taking advantage of Lemma \ref{lem:S.3} to simplify the presentation somewhat. 
As in \cite[\S 1.5]{R}, let
\begin{equation}\label{eq:N.27}
1 \to C \to \tilde G \to G \to 1
\end{equation}
be a finite central extension such that $\tilde G$ is a connected reductive group with
simply connected derived group. The kernel $C$ acts naturally on $\mathcal H (\tilde G)$ and 
\begin{equation}\label{eq:S.17}
\mathcal H (\tilde G )^C \cong \mathcal H (G) .
\end{equation}
The action of $\Gamma$ on the based root datum of $(G,B,T)$ lifts uniquely to an action on
the corresponding based root datum for $\tilde G$, so the $\Gamma$-actions on $G$ and on
$\mathcal H (G)$ lift naturally to actions on $\tilde G$ and $\mathcal H (\tilde G)$.
Let $\mathcal H_{\mathbf q} (\tilde G)$ be the variation on 
$\mathcal H (\tilde G)$ with scalars $\C [\mathbf q ,\mathbf q^{-1}]$ where $\mathbf q$
is a formal variable (instead of scalars $\C$ and $q \in \C^\times$). 
In \cite[Theorem 3.5]{KL} an isomorphism
\begin{equation}\label{eq:S.18}
\mathcal H_{\mathbf q} (\tilde G) \cong K^{\tilde G \times \C^\times}(\mathcal Z_{\tilde G}) 
\end{equation}
is constructed, where the right hand side denotes the $\tilde G \times \C^\times$-equivariant
K-theory of the Steinberg variety $\mathcal Z_{\tilde G}$ of $\tilde G$. Since 
$G \rtimes \Gamma$ acts via conjugation on $\tilde G$ and on $\mathcal Z_{\tilde G}$, 
it also acts on $K^{\tilde G \times \C^\times}(\mathcal Z_{\tilde G})$. However, the 
connected group $G$ acts trivially, so the action factors via $\Gamma$. Now the definition 
of the generators in \cite[Theorem 3.5]{KL} shows that \eqref{eq:S.18} is 
$\Gamma$-equivariant. In particular it specializes to $\Gamma$-equivariant isomorphisms
\begin{equation}\label{eq:S.19}
\mathcal H (\tilde G) \cong \mathcal H_{\mathbf q} (\tilde G) 
\otimes_{\C [\mathbf q,\mathbf q^{-1}]} \C_q \cong K^{\tilde G \times \C^\times}
(\mathcal Z_{\tilde G}) \otimes_{\C [\mathbf q,\mathbf q^{-1}]} \C_q .
\end{equation}
Let $(\tilde t_q,\tilde x) \in (\tilde G)^2$ be a lift of $(t_q,x) \in G^2$ with
$\tilde x$ unipotent. The $\tilde G$-conjugacy class of $\tilde t_q$ defines a central 
character of $\mathcal H (\tilde G)$ and 
\[
\mathcal H (\tilde G) \otimes_{\Z (\mathcal H (\tilde G))} \C_{\tilde t_q} \cong
K^{\tilde G \times \C^\times}(\mathcal Z_{\tilde G}) 
\otimes_{R (\tilde G \times \C^\times)} \C_{\tilde t_q,q} .
\]
According to \cite[Proposition 8.1.5]{CG} there is an isomorphism
\begin{equation}\label{eq:CG}
K^{\tilde G \times \C^\times}(\mathcal Z_{\tilde G}) \otimes_{R (\tilde G \times \C^\times)} 
\C_{\tilde t_q,q} \cong H_* (\mathcal Z_{\tilde G}^{\tilde t_q,q} ,\C) . 
\end{equation}
Moreover \eqref{eq:CG} is $\Gamma$-equivariant, because all the maps involved in the
proof of \cite[Proposition 8.1.5]{CG} are functorial with respect to isomorphisms of
algebraic varieties. To be precise, one should note that throughout \cite[Chapter 8]{CG} it
is assumed that $\tilde G$ is simply connected. However, as we already have \eqref{eq:S.19} 
at our disposal, \cite[\S 8.1]{CG} also applies whenever the derived group of $\tilde G$
is simply connected.

Any Borel subgroup of $\tilde G$ contains $C$, so $\mathcal B^{\tilde t_q,\tilde x} = 
\mathcal B^{\tilde t_q,\tilde x}_{\tilde G}$ and $\mathcal B^{t_q,x} = \mathcal B^{t_q,x}_G$ 
are isomorphic algebraic varieties. From \cite[p. 414]{CG} we see
that the convolution product in Borel--Moore homology leads to an action of 
$H_* (\mathcal Z^{\tilde t_q,q}_{\tilde G} ,\C)$ on $H_* ( \mathcal B^{\tilde t_q,\tilde x}, 
\C)$. Notice that for $\tilde h \in H_* (\mathcal Z^{\tilde t_q,q}_{\tilde G} ,\C)$ and 
$g \in G \rtimes \Gamma$ we have
\[
g \cdot \tilde h \in H_* (\mathcal Z^{g \tilde t_q g^{-1},q}_{\tilde G} ,\C) \cong 
\mathcal H (\tilde G) \otimes_{\Z (\mathcal H (\tilde G))} \C_{g \tilde t_q g^{-1}} .
\]
An obvious generalization of \cite[Lemma 8.1.8]{CG} says that all these 
constructions are compatible with the above actions of $G \rtimes \Gamma$,
in the sense that the following diagram commutes:
\begin{equation}\label{eq:S.20}
\begin{array}{ccc}
H_* ( \mathcal B^{\tilde t_q,\tilde x}, \C)
& \xrightarrow{\; \tilde h \;} & H_* ( \mathcal B^{\tilde t_q,\tilde x}, \C) \\
\downarrow \scriptstyle{H_* (\mathrm{Ad}_g)} & & 
\downarrow \scriptstyle{H_* (\mathrm{Ad}_g)} \\
H_* ( \mathcal B^{g \tilde t_q g^{-1},g \tilde x g^{-1}}, \C) & 
\xrightarrow{\;g \cdot \tilde h\;} & 
H_* ( \mathcal B^{g \tilde t_q g^{-1},g \tilde x g^{-1}}, \C) .
\end{array}
\end{equation}
In particular the component group $\pi_0 (\Z_{\tilde G}(\tilde t_q,\tilde x))$ acts on
$H_* ( \mathcal B^{\tilde t_q,\tilde x}, \C)$ by $\mathcal H (\tilde G)$-intertwiners.
Let $\tilde \rho$ be an irreducible representation of this component group, appearing
in $H_* ( \mathcal B^{\tilde t_q,\tilde x}, \C)$. In other words, $(\tilde t_q,
\tilde x,\tilde \rho)$ is a Kazhdan--Lusztig triple for $\mathcal H (\tilde G)$. 
According to \cite[Theorem 7.12]{KL}
\begin{equation}\label{eq:S.16}
\mathrm{Hom}_{\pi_0 (\Z_{\tilde G}(\tilde t_q,\tilde x))} 
(\tilde \rho , H_* (\mathcal B^{\tilde t_q,\tilde x},\C)) 
\end{equation}
is a $\mathcal H (\tilde G)$-module with a unique irreducible quotient, say 
$V_{\tilde t_q,\tilde x,\tilde \rho}$. 

Following \cite[\S 3.3]{R} we define a group $R_{\tilde t_q,\tilde x}$ by
\begin{equation}\label{eq:S.21}
1 \to \pi_0 (\Z_{\tilde G}(\tilde t_q,\tilde x)) \to \pi_0 (\Z_{\tilde G} (t_q,x)) \to 
R_{\tilde t_q,\tilde x} \to 1 .
\end{equation}
Obviously $\Z_{\tilde G}(\tilde t_q,\tilde x)$ contains $\Z (\tilde G)$, so the sequence
\begin{equation}\label{eq:N.2}
1 \to \pi_0 (\Z_{\tilde G}(\tilde t_q,\tilde x) / \Z (\tilde G) )  \to 
\pi_0 (\Z_{\tilde G} (t_q,x) / \Z (\tilde G)) \to 
R_{\tilde t_q,\tilde x} \to 1
\end{equation}
is also exact. For the middle term we have
\[
\Z_{\tilde G}(t_q, x) / \Z (\tilde G) \cong 
\Z_{G}(t_q,x) / \Z (G)
\]
Since the derived group of $\tilde G$ is simply connected, $\Z_{\tilde G} (t_q)^\circ = 
\Z_{\tilde G} (\tilde t_q)$. In the second term of \eqref{eq:N.2} we get
\[
\Z_{\tilde G}(\tilde t_q,\tilde x) / \Z (\tilde G) \cong
\Z_{\Z_{\tilde G}(t_q)^\circ} (\tilde x) / \Z (\tilde G) \cong
\Z_{\Z_G (t_q)^\circ} (x) / \Z (G) .
\]
Let us abbreviate $M = \Z_G (t_q)$. Then \eqref{eq:N.2} can be written as
\begin{equation*}
1 \to \pi_0 (\Z_{M^\circ}(x) / \Z(G)) \to 
\pi_0 (\Z_M (x) / \Z(G)) \to R_{\tilde t_q,\tilde x} \to 1 .
\end{equation*}
Like in \eqref{eq:N.2} we can derive another short exact sequence
\begin{equation}\label{eq:N.3}
1 \to \pi_0 (\Z_{M^\circ}(x) / \Z (M^\circ)) \to 
\pi_0 (\Z_M (x) / \Z (M^\circ)) \to R_{\tilde t_q,\tilde x} \to 1 .
\end{equation}
It can also be obtained from \eqref{eq:S.21} by the dividing the two appropriate
groups by the inverse image of $\Z (M^\circ)$ in $\tilde G$. 
From Lemma \ref{lem:S.3} (with the trivial representation of 
$\pi_0 (\Z_G (t_q)^\circ (x))$ in the role of $\rho$) 
we know that \eqref{eq:N.3} splits. By Proposition \ref{prop:UP} and
\eqref{eq:componentsBA} $\Z (M^\circ)$ acts trivially on 
$H_* ( \mathcal B^{\tilde t_q,\tilde x}, \C)$.
Hence all the 2-cocycles of subgroups of 
$R_{\tilde t_q,\tilde x}$ appearing associated to \eqref{eq:S.16} are trivial.

Let $\tilde \sigma$ be any irreducible representation of $R_{\tilde t_q,\tilde x,\tilde \rho}$,
the stabilizer of the isomorphism class of $\tilde \rho$ in $R_{\tilde t_q,\tilde x}$.
Clifford theory for \eqref{eq:N.3} produces $\tilde \rho \rtimes \tilde \sigma \in
\Irr \big(\pi_0 (\Z_M (x) / \Z (M^\circ)) \big)$, a representation which lifts to
$\pi_0 (\Z_G (t_q,x))$. Moreover by \cite[Lemma 3.5.1]{R}
it appears in $H_* (\mathcal B^{t_q,x},\C)$, and conversely every irreducible representation
with the latter property is of the form $\tilde \rho \rtimes \tilde \sigma$.

With the above in mind, \cite[Lemma 3.5.2]{R} says that the $\mathcal H (G)$-module
\begin{equation}\label{eq:S.27}
\begin{split}
M(t_q,x,\tilde \rho \rtimes \tilde \sigma) := &\; \mathrm{Hom}_{\pi_0 (\Z_G (t_q,x))} 
\big( \tilde \rho \rtimes \tilde \sigma, H_* (\mathcal B^{t_q,x},\C) \big) \\
= &\; \mathrm{Hom}_{R_{\tilde t_q,\tilde x,\tilde \rho}} \big( \tilde \sigma, 
\mathrm{Hom}_{\pi_0 (\Z_{\tilde G}(\tilde t_q,\tilde x))} (\tilde \rho , 
H_* (\mathcal B^{\tilde t_q,\tilde x},\C))  \big)
\end{split} 
\end{equation}
has a unique irreducible quotient
\begin{equation}\label{eq:S.29}
\pi (t_q,x,\tilde \rho \rtimes \tilde \sigma) = 
\mathrm{Hom}_{R_{\tilde t_q,\tilde x,\tilde \rho}} 
(\tilde \sigma , V_{\tilde t_q,\tilde x,\tilde \rho}) .
\end{equation}
According  \cite[Lemma 3.5.3]{R} this sets up a bijection between 
Irr$(\mathcal H (G))$ and $G$-conjugacy classes of Kazhdan--Lusztig triples for $G$.
Since $H_* (\mathcal B^{\tilde t_q,\tilde x},\C)$ has $Z (\cH (\tilde G))$-character 
$\cW^G \tilde{t}_q$, the $\cH (G)$-modules $H_* (\mathcal B^{t_q,x},\C)$ and 
\eqref{eq:S.29} have $Z (\cH (G))$-character $\cW^G t_q$.

\begin{rem}
The module \eqref{eq:S.27} is well-defined for any $q \in \C^\times$, although for roots of
unity it may have more than one irreducible quotient. For $q=1$ the algebra $\mathcal H (G)$
reduces to $\C [X^* (T) \rtimes \mathcal W^G]$ and \cite[Section 8.2]{CG} shows that Kato's
module \eqref{eq:KatoMod} is a direct summand of $M (t_1,x,\rho_1)$. 
\end{rem}

Next we study what $\Gamma$ does to all these objects. There is natural action of $\Gamma$
on Kazhdan--Lusztig triples for $G$, namely 
\[
\gamma \cdot (t_q,x, \rho_q) =  \big( \gamma t_q \gamma^{-1}, 
\gamma x \gamma^{-1}, \rho_q\circ \mathrm{Ad}_\gamma^{-1} \big) .
\]
From \eqref{eq:S.20} and \eqref{eq:S.27} we deduce that the diagram
\begin{equation}\label{eq:S.22}
\begin{array}{ccc}
\pi (t_q,x,\rho_q ) & \xrightarrow{\; h \;} & \pi (t_q,x,\rho_q) \\
\downarrow \scriptstyle{H_* (\mathrm{Ad}_g)} & & 
\downarrow \scriptstyle{H_* (\mathrm{Ad}_g)} \\
\!\! \pi \big( g t_q g^{-1},g x g^{-1}, \rho_q \circ \mathrm{Ad}_g^{-1} \big) & 
\xrightarrow{\; \gamma (h) \;} & 
\pi \big( g t_q g^{-1},g x g^{-1},\rho_q \circ \mathrm{Ad}_g^{-1} \big)
\end{array}
\end{equation}
commutes for all $g \in G \gamma$ and $h \in \mathcal H (G)$. Hence 
\begin{equation}\label{eq:S.23}
\text{Reeder's parametrization of } \Irr (\mathcal H (G)) \text{ is }\Gamma\text{-equivariant.} 
\end{equation}
Let $\pi \in \Irr (\mathcal H (G))$ and choose a Kazhdan--Lusztig triple such that
$\pi$ is equivalent with $\pi (t_q,x,\rho_q)$. Composition with $\gamma^{-1}$ on $\pi$ 
gives rise to a 2-cocycle $\natural (\pi)$ of $\Gamma_\pi$. Clifford theory tells us that 
every irreducible representation of $\mathcal H (G) \rtimes \Gamma$ is of the form 
$\pi \rtimes \rho_2$ for some $\pi \in \Irr (\mathcal H (G))$, unique up to 
$\Gamma$-equivalence, and a unique $\rho_2 \in \Irr (\C [\Gamma_\pi, \natural (\pi)])$. 
By the above the stabilizer of $\pi$ in $\Gamma$ equals the stabilizer of the $G$-conjugacy 
class $[t_q,x,\rho_q]_G$. Thus we have parametrized $\Irr (\mathcal H (G) \rtimes \Gamma)$ 
in a natural way with $G \rtimes \Gamma$-conjugacy classes of quadruples 
$(t_q,x,\rho_q,\rho_2)$, where $(t_q,x,\rho_q)$ is a Kazhdan--Lusztig triple for $G$ and 
$\rho_2 \in \Irr \big( \C [\Gamma_{[t_q,x,\rho_q]_G}, \natural(\pi(t_q,x,\rho_q))] \big)$.

The short exact sequence
\begin{equation}\label{eq:S.24}
1 \to \pi_0 (\Z_G (t_q,x)) \to \pi_0 (\Z_{G \rtimes \Gamma} (t_q,x)) \to
\Gamma_{[t_q,x]_G} \to 1
\end{equation}
yields an action of $\Gamma_{[t_q,x]_G}$ on Irr$\big( \pi_0 (\Z_G (t_q,x)) \big)$.
Restricting this to the stabilizer of $\rho_q$, we obtain another 2-cocycle 
$\natural(t_q,x,\rho_q)$ of $\Gamma_{[t_q,x,\rho_q]_G}$, which we want to compare to 
$\natural(\pi(t_q,x,\rho_q))$. Let us decompose
\[
H_* (\mathcal B^{t_q,x},\C) \cong \bigoplus\nolimits_{\rho_q} \rho_q \otimes 
M(t_q,x,\rho_q)
\]
as $\pi_0 ((\Z_G (t_q,x)) \times \mathcal H(G)$-modules. We sum over all $\rho_q \in
\Irr \big( \pi_0 (\Z_G (t_q,x)) \big)$ for which the contribution is nonzero, and we know
that for such $\rho_q$ the $\mathcal H (G)$-module $M(t_q,x,\rho_q)$ has a unique irreducible 
quotient $\pi (t_q,x,\rho_q)$. Since $\pi_0 (\Z_{G \rtimes \Gamma} (t_q,x))$ acts 
(via conjugation of Borel subgroups) on $H_* (\mathcal B^{t_q,x},\C)$, any splitting of 
\eqref{eq:S.24} as sets provides a 2-cocycle $\natural$ for the action of 
$\Gamma_{[t_q,x,\rho_q]_G}$ on $\rho_q \otimes M(t_q,x,\rho_q)$.
Unfortunately we cannot apply Lemma \ref{lem:S.3} to find a splitting of \eqref{eq:S.24}
as groups, because $\Z_G (t_q)$ need not be connected. Nevertheless $\natural$ 
can be used to describe the actions of $\Gamma_{[t_q,x,\rho_q]_G}$ on both $\rho_q$ and 
$\pi (t_q,x,\rho_q)$, so 
\begin{equation}\label{eq:S.25}
\natural(t_q,x,\rho_q) = \natural = \natural(\pi(t_q,x,\rho_q)) 
\text{ as 2-cocycles of } \Gamma_{[t_q,x,\rho_q]_G} .
\end{equation}
It follows that every irreducible representation $\rho$ of $\pi_0 (\Z_{G \rtimes \Gamma} 
(t_q,x))$ is of the form $\rho_q \rtimes \rho_2$ for $\rho_q$ and $\rho_2$ as above.
Moreover $\rho$ determines $\rho_q$ up to $\Gamma_{[t_q,x]_G}$-equivalence and $\rho_2$
is unique if $\rho_q$ has been chosen. Finally, if $\rho_q$ appears in 
$H_{\top}(\mathcal B^{t_q,x},\C)$ then every irreducible $\pi_0 (\Z_G (t_q,x))
$-subrepresentation of $\rho$ does, because $\pi_0 (\Z_{G \rtimes \Gamma} (t_q,x))$ acts 
naturally on $H_* (\mathcal B^{t_q,x},\C)$. Therefore we may replace the above 
quadruples $(t_q,x,\rho_q,\rho_2)$ by Kazhdan--Lusztig triples $(t_q,x,\rho)$. 

The module associated to $(t_q,x,\rho_q,\rho_2)$ via the above constructions is the
unique irreducible quotient of the $\cH (G) \rtimes \Gamma$-module
\begin{equation}\label{eq:HGmod2}
\mathrm{Hom}_{\pi_0 (\Z_G (t_q,x))} 
\big( \rho_q, H_* (\mathcal B^{t_q,x},\C) \big) \rtimes \rho_2 .
\end{equation}
The same reasoning as in the proof of Theorem \ref{thm:SpringerExtended} shows that
\eqref{eq:HGmod2} is isomorphic to
\begin{equation}\label{eq:HGmod1}
\mathrm{Hom}_{\pi_0 (\Z_{G \rtimes \Gamma} (t_q,x))} 
\big( \rho, H_* (\mathcal B^{t_q,x},\C) \otimes \C[\Gamma] \big) .
\end{equation}
Since the $\cH (G)$-module $H_* (\mathcal B^{t_q,x},\C)$ depends in a natural way
on $(t_q,x)$, so does the unique irreducible quotient of \eqref{eq:HGmod1}. As 
$H^* (\mathcal B^{t_q,x},\C)$ has $Z(\ch (G))$-character $\cW^G t_q$, \eqref{eq:HGmod1} 
has $Z(\cH (G) \rtimes \Gamma)$-character $(\cW^G \rtimes \Gamma) t_q$.
\end{proof}
 
For use in Section \ref{sec:general} we discuss some analytic properties of 
$\cH (G) \rtimes \Gamma$-modules. Let $\{ \theta_x T_w : x \in X^* (T), w \in \cW^G \}$ be
the Bernstein basis of $\cH (G)$ \cite[\S 3]{LuGrad}. Recall from \cite[Lemma 2.20]{Opd}
that a finite-dimensional $\cH (G)$-module $(\pi,V)$ is tempered if and only if
all eigenvalues of the $\pi (\theta_x)$ with $x$ in the closed positive cone $X^* (T)^+
\subset X^* (T)$ have absolute value $\leq 1$. Similarly,
$(\pi,V)$ is square-integrable if and only if the eigenvalues of the $\pi (\theta_x)$ with 
$x \in X^* (T) \setminus \{0\}$ have absolute value $< 1$ \cite[Lemma 2.22]{Opd}. 
We say that $(\pi,V)$ is essentially square-integrable if its restriction to
$\cH ( G / Z(G))$ is square-integrable.

We will treat these criteria as definitions of temperedness and (essential) 
square-integrability. We define that a finite-dimensional $\cH (G) \rtimes \Gamma$-module 
is tempered or (essentially) square-integrable if and only if its restriction to 
$\cH (G)$ is so. One can check that square-integrability is equivalent to temperedness
plus essential square-integrability.

\begin{prop}\label{prop:tempered}
Let $(t_q,x,\rho)$ be a Kazhdan--Lusztig triple for $\cH (G) \rtimes \Gamma$.
\begin{enumerate}
\item Lift $(t_q,x,\rho)$ to a KLR parameter $(\Phi,\rho)$ as in Lemma 
\ref{lem:compareParameters} and write $t = \Phi (\varpi_F) \in T$. The 
$\cH (G)\rtimes \Gamma$-module $\pi (t_q,x,\rho)$ is tempered if and only if
$t$ is contained in a compact subgroup of $T$.
\item $\pi (t_q,x,\rho)$ is essentially square-integrable if and only if $\{t_q,x\}$ 
is not contained in any Levi subgroup of a proper parabolic subgroup of $G$.
\end{enumerate}
\end{prop}
\begin{proof}
(1) Since $X^* (G / G_\der) +  X^* (T / \Cent (G))$ is a sublattice of finite index in 
$X^* (T)$, a finite dimensional $\cH (G)$-module is tempered if and only its 
restrictions to $\cH (G / G_\der)$ and to $\cH (T / \Cent (G))$ are both tempered. 

As $G / G_\der$ is a torus, $\pi_z := \pi (t_q,x,\rho) |_{\cH (G / G_\der)}$ 
is tempered if and only if all the eigenvalues of the $\pi_z (\theta_x)$ with 
$x \in X^* (G / G_\der)$ have absolute value 1. By Theorem \ref{thm:S.5} the central
character of $\pi (t_q,x,\rho)$ is $(\cW^G \rtimes \Gamma) t_q$, so the eigenvalues 
of $\pi_z (\theta_x)$ belong to $x ((\cW^G \rtimes \Gamma) t_q)$. Because
$\cW^G \rtimes \Gamma$ acts on $T$ by algebraic automorphisms, it preserves the unique
maximal compact subgroup $T_\cpt$. Therefore $\pi_z$ is tempered if and only if
$t_q \in T_\cpt / T_\cpt \cap G_\der$. Since $t_q \in t \Phi (\SL_2 (\C))$ and 
$\Phi (\SL_2 (\C)) \subset G_\der$, this is equivalent to $t$ being compact in 
$T / T \cap G_\der$.

We denote the restriction of $\pi(t_q,x,\rho)$ to $\cH (G / \Cent (G))$ by $\pi_d$.
Since the varieties of Borel subgroups for $G$ and for $G/ \Cent (G)$ are isomorphic and the
relevant irreducible modules are made from the homologies of such varieties, $\pi_d$ is
a direct sum of representations $\pi (t_q,x,\rho_i)$ with $\rho_i \in \Irr \big( \pi_0 
(Z_{G /\Cent (G)}(t_q,x)) \big)$. 

We would like to apply \cite[Theorem 8.2]{KL}, but it is only proven for simply connected
groups. Let $\tilde H$ be the simply connected cover of $G / \Cent (G)$. Since $\SL_2 (\C)$ 
is simply connected, $\Phi|_{\SL_2 (\C)}$ lifts in unique way to a homomorphism 
$\tilde{\Phi} : \SL_2 (\C) \to \tilde{H}$. Lift $t$ to $\tilde{t} \in \tilde{H}$.
Then $\tilde{t}$ centralizes $\tilde{\Phi} (\SL_2 (\C))$, because it centralizes the
Lie algebra of that group. Now $\tilde{t}_q := t \tilde{\Phi}\matje{q^{1/2}}{0}{0}{q^{-1/2}}
\in \tilde{H}$ is a lift of $t_q$.

From \eqref{eq:HGmod2}, \eqref{eq:HGmod1} and \eqref{eq:S.27} we see that $\pi_d$ is a
$\cH (G / \Cent (G))$-summand of a sum of irreducible $\cH (\tilde H)$-modules 
$\pi (\tilde{t}_q,\tilde x,\tilde{\rho}_i)$ with $\tilde{\rho}_i \in \Irr \big( \pi_0 
(Z_{\tilde{H}}(\tilde{t}_q,\tilde{x})) \big)$. All irreducible $\cH (G / \Cent (G))$-submodules
of $\pi (\tilde{t}_q,\tilde x,\tilde{\rho}_i)$ are conjugate by invertible elements of
$\cH (\tilde{H})$, so $\pi (\tilde{t}_q,\tilde x,\tilde{\rho}_i)$ is tempered if and only
if any of its $\cH (G / \Cent (G))$-constituents is tempered. It follows that $\pi_d$ is
tempered if and only if $\pi (\tilde{t}_q,\tilde x,\tilde{\rho}_i)$ is tempered for every
relevant $\tilde{\rho}_i$.

According to \cite[Theorem 8.2]{KL}, the latter condition is equivalent to ``$P = G$'',
which by \cite[\S 7.1]{KL} is the same as ``all eigenvalues of Ad$(\tilde{t})$ on
Lie$(\tilde{H})$ have absolute value 1''. Since $\tilde{H} \to G / \Cent (G)$ is a central
extension, this statement is equivalent to the analgous one for Ad$(t)$ acting on
Lie$(G / \Cent (G))$. We showed that $\pi_d$ is tempered if and only if all the eigenvalues
of Ad$(t)$ on Lie$(G / \Cent (G))$ have absolute value $\leq 1$. The nontrivial eigenvalues
come in pairs $\alpha (t), \alpha (t)^{-1}$ with $\alpha \in \Phi (G,T)$, so we may 
replace the condition $\leq 1$ by $= 1$.

We conclude with a chain of equivalences.
\begin{align*}
& \pi (t_q,x,\rho) \text{ is tempered } \Longleftrightarrow \\
& \pi_z \text{ is tempered and } \pi_d \text{ is tempered } \Longleftrightarrow \\
& t \text{ is compact in } T / T \cap G_\der \text{ and } 
|\alpha (t)| = 1 \; \forall \alpha \in \Phi (G,T) \Longleftrightarrow \\
& |x(t)| = 1 \; \forall x \in X^* (G / G_\der) \text{ and }
|x(t)| = 1 \; \forall x \in X^* (T / Z(G)) \Longleftrightarrow \\
& |x(t)| = 1 \; \forall x \in X^* (T) \Longleftrightarrow \\
& t \text{ belongs to the maximal compact subgroup of } T .
\end{align*}
(2) In the same way as in the proof of part (1), this can be reduced to irreducible
modules of $\cH (\tilde{H})$. Since $\tilde{H}$ is simply connected, essential
square-integrability is now the same as square-integrability, and \cite{KL} applies.
According to \cite[Theorem 8.3]{KL} a $\cH (\tilde{H})$-module $\pi (\tilde{t}_q,
\tilde{x},\tilde{\rho})$ is square-integrable if and only if $\{ \tilde{t}_q, \tilde{x}\}$
is not contained in any Levi subgroup of a proper parabolic subgroup of $\tilde{H}$.
There is a canonical bijection between parabolic (respectively Levi) subgroups
of $\tilde{H}$ and of $G$, so this statement is equivalent to the analogous
one for $\{t_q,x\}$ and $G$.
\end{proof}

\section{Spherical representations}
\label{sec:spherical}

Let $G,B,T$ and $\Gamma$ be as in the previous section.
Let $\mathcal H (\cW^G )$ be the Iwahori--Hecke algebra of the Weyl group 
$\cW^G$, with a parameter $q \in \C^\times$ which is not a root of unity. This
is a deformation of the group algebra $\C [\cW^G]$ and a subalgebra of the affine Hecke 
algebra $\cH (G)$. The multiplication is defined in terms of the basis 
$\{ T_w \mid w \in \cW^G\}$ by
\begin{eqnarray} \label{eq:defHA}
T_x T_y= T_{xy}, \quad \text{if $\ell(xy)=\ell(x)+\ell(y)$, and}\cr
(T_s-q)(T_s+1)=0, \quad \text{if $s$ is a simple reflection.}
\end{eqnarray}
Recall that $\cH (G)$ also has a commutative subalgebra $\cO (T)$, 
such that the multiplication maps
\begin{equation}\label{eq:multmaps}
\cO (T) \otimes \cH (\cW^G) \longrightarrow \cH (G) \longleftarrow 
\cH (\cW^G) \otimes \cO (T)
\end{equation}
are bijective.

The trivial representation of $\cH (\cW^G) \rtimes \Gamma$ is defined as
\begin{equation}
\triv (T_w \gamma) = q^{\ell (w)} \quad w \in \cW^G, \gamma \in \Gamma .
\end{equation}
It is associated to the idempotent
\[
p_\triv \: p_\Gamma := \sum_{w \in \cW^G} T_w P_{\cW^G}(q)^{-1} \, 
\sum_{\gamma \in \Gamma} \gamma |\Gamma |^{-1} \quad \in \cH (\cW^G) \rtimes \Gamma ,
\]
where $P_{\cW^G}$ is the Poincar\'e polynomial
\[
P_{\cW^G}(q) = \sum\nolimits_{w \in \cW^G} q^{\ell (w)} .
\]
Notice that $P_{\cW^G}(q) \neq 0$ because $q$ is not a root of unity.
The trivial representation appears precisely once in the regular representation of
$\cH (\cW^G) \rtimes \Gamma$, just like for finite groups.

An $\cH (G) \rtimes \Gamma$-module $V$ is called spherical if it is generated by
the subspace $p_\triv p_\Gamma V$ \cite[(2.5)]{HeOp}. This admits a nice
interpretation for the unramified principal series representations. Recall that
$\cH (G) \cong \cH (\cG,\mathcal I)$ for an Iwahori subgroup $\mathcal I \subset \cG$. 
Let $\mathcal K \subset \cG$ be a good maximal compact subgroup containing $\mathcal I$. 
Then $p_\triv$ corresponds to averaging over $\mathcal K$ and $p_\triv \cH (\cG,\mathcal I) 
p_\triv \cong \cH (\cG,\mathcal K)$, see \cite[Section 1]{HeOp}. Hence spherical 
$\cH (\cG,\mathcal I)$-modules correspond to smooth
$\cG$-representations that are generated by their $\mathcal K$-fixed vectors, also known as
$\mathcal K$-spherical $\cG$-representations. By the Satake transform
\begin{equation} \label{eq:Satake}
p_\triv \cH (\cG,\mathcal I) p_\triv \cong \cH (\cG,\mathcal K) \cong \cO (T / \cW^G) ,
\end{equation}
so the irreducible spherical modules of $\cH (G) \cong \cH (\cG,\mathcal I)$ are parametrized
by $T / \cW^G$ via their central characters. We want to determine the Kazhdan--Lusztig 
triples (as in Theorem \ref{thm:S.5}) of these representations.

\begin{prop}\label{prop:spherical}
For every central character $(\cW^G \rtimes \Gamma) t \in T / (\cW^G \rtimes \Gamma)$ 
there is a unique irreducible spherical $\cH (G) \rtimes \Gamma$-module,
and it is associated to the Kazhdan--Lusztig triple $(t,x=1,\rho = \triv)$.
\end{prop}
\begin{proof}
We will first prove the proposition for $\cH (G)$, and only then consider $\Gamma$.

By the Satake isomorphism \eqref{eq:Satake} there is a unique irreducible spherical
$\cH (G)$-module for every central character $\cW^G t \in T / \cW^G$. The equivalence
classes of Kazhdan--Lusztig triples of the form $(t,x=1,\rho = \triv)$ are also in
canonical bijection with $T / \cW^G$. Therefore it suffices to show that $\pi (t,1,\triv)$
is spherical for all $t \in T$.

The principal series of $\cH (G)$ consists of the modules 
$\mathrm{Ind}_{\cO (T)}^{\cH (G)} \C_t$ for $t \in T$. This module admits a central
character, namely $\cW^G t$. By \eqref{eq:multmaps} every such module is isomorphic to 
$\cH (\cW^G)$ as a $\cH (\cW^G)$-module. In particular it contains the trivial 
$\cH (\cW^G)$-representation once and has a unique irreducible spherical subquotient. 

As in Section \ref{sec:repAHA}, let $\tilde G$ be a finite central extension of $G$ with
simply connected derived group. Let $\tilde T ,\tilde B$ be the corresponding extensions 
of $T,B$. We identify the roots and the Weyl groups of $\tilde G$ and $G$. 
Let $\tilde t \in \tilde T$ be a lift of $t \in T$.
From the general theory of Weyl groups it is known that there is a unique
$t^+ \in \cW^G \tilde T$ such that $|\alpha (t^+)| \geq 1$ for all $\alpha \in 
R (\tilde B, \tilde T) = R(B,T)$.
By \eqref{eq:S.20} 
\[
H_* \big( \mathcal B_{\tilde G}^{\tilde t},\C \big) \cong H_*
\big( \mathcal B_{\tilde G}^{t^+} ,\C \big)
\]
as $\cH (\tilde G)$-modules. These $t^+,\tilde B$ fulfill \cite[Lemma 2.8.1]{R}, so by
\cite[Proposition 2.8.2]{R}
\begin{equation}
M_{\tilde t,\tilde x = 1,\tilde \rho = \triv} = H_* (\mathcal B_{\tilde G}^{t^+},\C)
\cong \mathrm{Ind}_{\cO (\tilde T)}^{\cH (\tilde G)} \C_{t^+} .
\end{equation}
According to \cite[(1.5)]{ReeWhittaker}, which applies to $t^+$, the spherical vector $p_\triv$
generates $M_{\tilde t,1,\triv}$. Therefore it cannot lie in any proper 
$\cH (\tilde G)$-submodule of $M_{\tilde t,1,\triv}$ and represents a nonzero element
of $\pi (\tilde t,1,\triv)$. We also note that the central character of $\pi (\tilde t,1,
\triv)$ is that of $M_{\tilde t,1,\triv} ,\; \cW^G \tilde t = \cW^G t^+$. 

Now we analyse this is an $\cH (G)$-module. The group $R_{\tilde t,1} = 
R_{\tilde t,\tilde x = 1, \tilde \rho = \triv}$ from \eqref{eq:S.21} is just the component 
group $\pi_0 (\Z_G (t))$, so by \eqref{eq:S.29}
\[
\pi (\tilde t,1,\triv) \cong \bigoplus\nolimits_\rho \mathrm{Hom}_{\pi_0 (\Z_G (t))}
(\rho, \pi (\tilde t,1,\triv)) = \bigoplus\nolimits_\rho \pi (t,1,\triv) .
\]
The sum runs over $\Irr \big( \pi_0 (\Z_G (t)) \big)$, all these representations $\rho$
contribute nontrivially by \cite[Lemma 3.5.1]{R}. Recall from Lemma \ref{lem:centrals}
that $\pi_0 (\Z_G (t))$ can be realized as a subgroup of $\cW^G$ and from \eqref{eq:Satake}
that $p_\triv \in \pi (\tilde t,1,\triv)$ can be regarded as a function on $\tilde \cG$
which is bi-invariant under a good maximal compact subgroup $\tilde{\mathcal K}$. This brings
us in the setting of \cite[Proposition 4.1]{Cas}, which says that $\pi_0 (\Z_G (t))$ fixes 
$p_\triv \in \pi (\tilde t,1,\triv)$. Hence $\pi (t,1,\triv)$ contains $p_\triv$ and is a
spherical $\cH (G)$-module. Its central character is the restriction of the central character
of $\pi (\tilde t,1,\triv)$, that is, $\cW^G t \in T / \cW^G$.

Now we include $\Gamma$. Suppose that $V$ is a irreducible spherical
$\cH (G) \rtimes \Gamma$-module. By Clifford theory its restriction to $\cH (G)$ is a
direct sum of irreducible $\cH (G)$-modules, each of which contains $p_\triv$. Hence $V$
is built from irreducible spherical $\cH (G)$-modules. By \eqref{eq:S.23}
\[
\gamma \cdot \pi (t,1,\triv) = \pi (\gamma t,1,\triv) ,
\]
so the stabilizer of $\pi (t,1,\triv) \in \Irr (\cH (G))$ in $\Gamma$ equals the stabilizer
of $\cW^G t \in T / \cW^G$ in $\Gamma$. Any isomorphism of $\cH (G)$-modules
\[
\psi_\gamma : \pi (t,1,\triv) \to \pi (\gamma t,1,\triv) 
\]
must restrict to a bijection between the onedimensional subspaces of spherical vectors
in both modules. We normalize $\psi_\gamma$ by $\psi_\gamma (p_\triv) = p_\triv$. Then
$\gamma \mapsto \psi_\gamma$ is multiplicative, so the 2-cocycle of $\Gamma_{\cW^G t}$
is trivial. With Theorem \ref{thm:S.5} this means that the irreducible $\cH (G) \rtimes 
\Gamma$-modules whose restriction to $\cH (G)$ is spherical are parametrized by 
equivalence classes of triples $(t,1,\triv \rtimes \sigma)$ with $\sigma \in
\Irr (\Gamma_{\cW^G t})$. The corresponding module is
\begin{equation*}
\pi (t,1,\triv \rtimes \sigma) = \pi (t,1,\triv) \rtimes \sigma =
\mathrm{Ind}_{\cH (G) \rtimes \Gamma_{\cW^G t}}^{\cH (G) \rtimes \Gamma}
\big( \pi (t,1,\triv) \otimes \sigma \big) .
\end{equation*}
Clearly $\pi (t,1,\triv \rtimes \sigma)$ contains the spherical vector
$p_\triv p_\Gamma$ if and only if $\sigma$ is the trivial representation. It follows
that the irreducible spherical $\cH (G) \rtimes \Gamma$-modules are parametrized by
equivalence classes of triples $\big( t,1,\triv_{\pi_0 (\Z_{G \rtimes \Gamma}(t))} \big)$,
that is, by $T / (\cW^G \rtimes \Gamma)$.
\end{proof}

\section{From the principal series to affine Hecke algebras}
\label{sec:Roche}

Let $\chi$ be a smooth character of the maximal torus $\cT \subset \cG$. We recall that
\begin{align*}
\fs & = [\cT,\chi]_{\cG},\\
c^\fs & = \hat \chi \big|_{\fo_F^\times},\\
H & = \Cent_G(\im c^\fs) , \\
W^\fs &  = \Z_{\cW^G} (\im c^\fs) .
\end{align*}
Let \{KLR parameters$\}^\fs$ be the collection of Kazhdan--Lusztig--Reeder parameters 
for $G$ such that $\Phi \big|_{\fo_F^\times} = c^\fs$. Notice that the condition forces
$\Phi (\mathbf W_F \times \SL_2 (\C)) \subset H$. This collection is not closed under 
conjugation by elements of $G$, only $H = \Z_G (\im c^\fs)$ acts naturally on it.

Recall that $T^\fs$ and $T^\fs / W^\fs$ are Bernstein's torus and Bernstein's centre
associated to $\fs$. Clearly $T$ acts simply
transitively on $T^\fs$, but we need a little more. Consider the bijections
\begin{equation}\label{eq:N.22}
T^\fs \longrightarrow \{ \text{L-parameters } \Phi \text{ for } \cT \text{ with } 
\Phi |_{\fo_F^\times} = c^\fs \} \xrightarrow{\; \mathrm{ev}_{\varpi_F} \;} T ,
\end{equation}
where the first map is the restriction of the local Langlands correspondence for $\cT$
to $T^\fs$ and the second map sends $\Phi$ to $\Phi (\varpi_F)$. The latter is not
natural because it depends on our choice of $\varpi_F$, but since we use the same
uniformizer everywhere this is not a problem.

As $T^\fs$ is a maximal torus in $H$, every semisimple element of $H^\circ$
is conjugate to one in $T$. By Lemma \ref{lem:centrals} $W^\fs \cong N_H (T) / T$, 
so we can identify $T^\fs / W^\fs$ with the space $c(H)_{\ss}$ of semisimple 
conjugacy classes in $H$ that consist of elements if $H^\circ$.

In general $H$ need not be connected. Recall from Lemma \ref{lem:pinning} that any 
choice of a pinning of $H^\circ$ determines a splitting of the short exact sequence
\begin{equation} \label{eq:splitH}
1 \to H^\circ / \Z(H^\circ) \to H / \Z(H^\circ) \to \pi_0 (H) \to 1 .
\end{equation}
Lemma \ref{lem:centrals} shows that
\begin{equation}\label{eq:WsH}
W^\fs = \mathcal W^G_{\im c^\fs} \cong \mathcal W^{H^\circ} \rtimes \pi_0 (H) .
\end{equation}
We fix a Borel subgroup $B \subset G$ containing $T$, and a pinning of $H^\circ$ with 
$T$ as maximal torus and $B_H = B \cap H^\circ$ as Borel subgroup.
This determines a conjugation action of $\pi_0 (H)$ on $H^\circ$, and hence on objects
associated to $H^\circ$. Like in Section \ref{sec:repAHA}, let $\cH (H^\circ)$ be the
affine Hecke algebra with the same based root datum as $(H^\circ,B)$, and with parameter 
$q$ equal to the cardinailty of the residue field of $F$. By our conventions $\pi_0 (H)$ 
normalizes $B$, so it acts on $\cH (H^\circ)$ by algebra automorphisms. 
Following \cite[Section 8]{Roc} we define
\begin{equation}\label{eq:HeckeH}
\cH (H) = \cH (H^\circ) \rtimes \pi_0 (H) . 
\end{equation}
We denote the Hecke algebra of $\cG$ by $\cH (\cG)$. Recall that its consists of
all locally constant compactly supported functions $\cG \to \C$ and is endowed with
the convolution product. The category $\Rep (\cG)$ of smooth $\cG$-representations
is naturally equivalent to the category of nondegenerate $\cH (\cG)$-modules.
Let $\Rep (\cG)^\fs$ be the block of $\Rep (\cG)$ associated to $\fs$. 

The link between these representations and Section \ref{sec:repAHA}
is provided by results of Roche. In \cite[p. 378--379]{Roc} Roche imposes some 
conditions on the residual characteristic of the field.

\begin{con}\label{con:char}
If the root system $R (H,T)$ is irreducible, 
then the restriction on the residual characteristic $p$ of $F$ is as follows:
\begin{itemize}
\item for type $A_n  \quad p > n+1$
\item for types $B_n, C_n, D_n \quad p \neq 2$
\item for type $F_4 \quad p \neq 2,3$
\item for types $G_2, E_6 \quad p \neq 2,3,5$
\item for types $E_7, E_8 \quad p \neq 2,3,5,7.$
\end{itemize}
If $R (H,T)$ is reducible, one excludes primes attached to each of its 
irreducible factors.
\end{con}
Since $R (H,T)$ is a subset of $R (G,T) \cong R (\cG,\cT)^\vee$, 
these conditions are fulfilled when they hold for $R (\cG,\cT)$.

\begin{thm}\label{thm:Roche} 
Assume that Condition \ref{con:char} holds. 
There exists an equivalence of categories
\[
\mathrm{Rep}(\cG)^\fs \longleftrightarrow \Mod (\cH (H)) 
\]
such that:
\begin{enumerate}
\item The cuspidal support of an irreducible $\cG$-representation corresponds 
to the central character of the associated $\cH (H)$-module via the canonical 
bijection $T^\fs / W^\fs \to c(H)_\ss$.
\item It does not depend on the choice of $\chi$ with $[\cT,\chi]_\cG = \fs$.
\end{enumerate}
\end{thm}
\begin{proof}
First we note that, although Roche \cite{Roc} works with a $p$-adic field, 
it follows from \cite{AdRo} that his 
arguments apply just as well over local fields of positive characteristic. 
By \cite[Corollary 7.9]{Roc} there exists a type $(J,\tau)$ for 
$\fs = [\cT,\chi]_\cG$, where $\tau$ is a character. Then the $\tau$-spherical 
Hecke algebra $\cH (\cG,\tau)$ of $\cH (\cG)$ (see \cite[\S 2]{BKtyp}) equals 
$e_\tau \cH (\cG) e_\tau$, where $e_\tau \in \cH (J)$ is the central idempotent 
corresponding to $\tau$. According to \cite[Theorem 4.3]{BKtyp} there exists 
an equivalence of categories
\begin{equation}\label{eq:N.7}
\mathrm{Rep} (\cG)^\fs \to \Mod ( \cH (\cG,\tau)) : V \mapsto V^\tau ,
\end{equation}
where $V^\tau = e_\tau V$ is the $\tau$-isotypical subspace of $V |_J$.
From the proof of \cite[Proposition 3.3]{BKtyp} we see that the inverse of
\eqref{eq:N.7} is given by
\begin{equation}\label{eq:N.45}
\Mod ( \cH (\cG,\tau)) \to \mathrm{Rep} (\cG)^\fs :
M \mapsto \cH (\cG) \otimes_{\cH (\cG,\tau)} M .
\end{equation}
Theorem 8.2 of \cite{Roc} says that there exists a support preserving 
algebra isomorphism
\begin{equation}\label{eq:N.6}
\cH (H) \to \cH (\cG,\tau) .
\end{equation}
The combination of \eqref{eq:N.7} and \eqref{eq:N.6} yields the desired 
equivalence of categories. It satisfies property (1) by \cite[Theorem 9.4]{Roc}.

In \cite[\S 9]{Roc} it is shown that $(J,\tau)$ is a cover of the type 
$(\cT_0 ,\chi |_{\cT_0})$, in the sense of \cite[\S 8]{BKtyp}. With 
\cite[Theorem 9.4]{Roc} one sees that the above equivalence of categories
does not change if one twists $\chi$ by an unramified character of $\cT$,
basically because that does not effect $\chi |_{\cT_0}$.

Every other character of $\cT$ determining the same inertial equivalence
class $\fs$ can be obtained from $\chi$ by an unramified twist and conjugation
by an element of $W^\fs$. Reeder \cite[\S 6]{R} checked that the latter 
operation does not change Roche's equivalence of categories. We note that in
\cite{R} it is assumed that $H$ is connected. Fortunately this does not play 
a role in \cite[\S 6]{R}, because all the underlying results from \cite{Roc}
and \cite{MENS} are known irrespective of the connectedness.
\end{proof}

We emphasize that Theorem \ref{thm:Roche} is the only cause of our conditions on
the residual characteristic. If one can prove Theorem \ref{thm:Roche} for a 
particular Bernstein component and a $p$ which is excluded by Condition 
\ref{con:char}, then everything in our paper (except possibly 
Lemma \ref{lem:Roche}) holds for that case. 

For example, for unramified 
characters $\chi$ Theorem \ref{thm:Roche} is already classical, proven 
without any restrictions on $p$ by Borel \cite{Bor}. As Roche remarks in
\cite[4.14]{Roc}, all the main results of \cite{Roc} (and hence Theorem 
\ref{thm:Roche}) are valid without restrictions on $p$ when $\cG = \GL_n (F)$ 
or $\cG = \SL_n (F)$. For $\GL_n (F)$ this is easily seen, for $\SL_n (F)$
one can use \cite{GoRo}.

Theorems \ref{thm:Roche} and \ref{thm:S.5} provide a bijection 
\begin{equation}\label{eq:N.8}
\Irr (\cG)^\fs \to \Irr (\cH (H)) \to \{ \text{KLR-parameters} \}^\fs / H . 
\end{equation}
Unfortunately this bijection is not entirely canonical in general.

\begin{ex}\label{ex:noncanonical}
Consider the unramified principal series representations of $\SL_2 (F)$.
Then the type is the trivial representation of an Iwahori subgroup 
$\mathcal I \subset \SL_2 (F)$ and Theorem \ref{thm:Roche} reduces to \cite{Bor}. 
The functor sends a $\SL_2 (F)$-representation to its space of $\mathcal I$-fixed 
vectors. The Iwahori subgroup is determined by the choice of a maximal compact
subgroup and a Borel subgroup of $\SL_2 (F)$, and these data also determine
the isomorphism $\cH (\SL_2 (F),\mathrm{triv}_{\mathcal I}) \cong \cH (H)$.

However, there are two conjugacy classes of maximal compact subgroups in 
$\SL_2 (F)$. If we pick a maximal compact subgroup in the other class and perform
the same operations, we obtain an alternative map \eqref{eq:N.8}. The difference
is not big, for almost all $\SL_2 (F)$-representations the two maps have the same
image. But look at the parabolically induced representation 
$\pi = I_\cB^{\SL_2 (F)}(\chi_{-1})$, where $\chi_{-1}$ denotes the unique 
unramified character of $\cT$ of order 2. It is well-known that $\pi$ is the direct
sum of two inequivalent irreducible representations, say $\pi_+$ and $\pi_-$.
It turns out that the difference between our two candidates for \eqref{eq:N.8} 
is just interchanging $\pi_+$ and $\pi_-$.
\end{ex}

We will determine in Section \ref{sec:can} how canonical \eqref{eq:N.8} is precisely.

\section{Main result (special case)} 
\label{sec:main} 

In the current section we will study the relations between $\Irr (\cG)^\fs$ 
and $(T^\fs /\!/ W^\fs )_2$, in the case that $H$ is connected. 
This happens for most $\fs$, a sufficient condition is:

\begin{lem} \label{lem:Roche}   
Suppose that $G$ has simply connected derived group and that the residual 
characteristic $p$ satisfies Condition \ref{con:char} for $R(G,T)$.
Then $H$ is connected.
\end{lem} 
\begin{proof}   
We consider first the case where $\fs=[\cT,1]_\cG$. 
Then we have  $c^\fs = 1, H=G$ and $W^\fs=\cW$.

We assume now that $c^\fs \ne 1$. Then $\im c^\fs$ is a finite abelian subgroup of 
$T$ which has the following structure: the direct product of a finite abelian $p$-group 
$A_p$ with a cyclic group $B_{q-1}$ whose order divides $q - 1$.  This follows from the  
well-known structure theorem for the group $\fo_F^\times$, see \cite[\S 2.2]{I}:
\[ 
\im c^\fs = A_p \cdot B_{q-1}.
\] 
We have
\[
H = \Cent_{H_A}(B_{q-1}) \quad\text{where}\quad H_A := \Cent_G (A_p).
\]
Since $G$ has simply connected derived group, $A_p$ is a $p$-group and $p$ is not a 
torsion prime for the root system $R(G,T)$, it follows from Steinberg's
connectedness theorem \cite[2.16.b]{Steinbergtorsion} that the group $H_A$
is connected. It was shown in \cite[p. 397]{Roc} that $H_A = \Cent_G (x)$ for
a well-chosen $x \in T$. Then \cite[2.17]{Steinbergtorsion} says that the derived
group of $H_A^\circ = H_A$ is simply connected.

Now  $B_{q-1}$ is cyclic. Applying Steinberg's connectedness theorem to the group $H_A$, 
we get that $H$ itself is connected. 
\end{proof}

\begin{rem}
Notice that $H$ does not necessarily have a simply connected derived group in setting
of Lemma \ref{lem:Roche}. For instance, if $G$ is the exceptional group of type $\rG_2$
and $\chi$ is the tensor square of a ramified quadratic character of $F^\times$, 
then $H = \SO_4 (\Cset)$.
\end{rem}

In the remainder of this section we will assume that $H$ is connected, Then Lemma
\ref{lem:centrals} shows that $W^\fs$ is the Weyl group of $H$. 

\begin{thm} \label{thm:ps}  
Let $\cG$ be a  split reductive $p$-adic group and let $\fs = [\mathcal T, \chi]_\cG$
be a point in the Bernstein spectrum of the principal series of $\cG$. Assume that $H$
is connected and that Condition \ref{con:char} holds. Then there is a commutative 
diagram of bijections, in which the triangle is canonical:
\[ 
\xymatrix{ &  & (T^\fs /\!/W^\fs)_2 \ar[dr]\ar[dl] & \\  
\Irr(\mathcal{G})^\fs \ar[r] & \Irr (\cH (H)) \ar[rr] & & 
\{\KLR\:\:\mathrm{parameters}\}^\fs / H} 
\] 
In the triangle the right slanted map stems from Kato's affine Springer correspondence
\cite{Kat}. The bottom horizontal map is the bijection established by Reeder \cite{R}
and the left slanted map can be constructed via the asymptotic Hecke algebra of Lusztig.
\end{thm}
\begin{proof}
Theorem \ref{thm:Roche} provides the bijection
$\Irr(\mathcal{G})^\fs \to \Irr (\mathcal H (H))$. 

The right slanted map is the composition of Theorem \ref{thm:S.3}.1 (applied
to $H$) and Lemma \ref{lem:compareParameters} (with the condition $\Phi(\varpi_F)=t$). 
We can take as the horizontal map the parametrization of irreducible 
$\mathcal H (H)$-modules
by Kazhdan, Lusztig and Reeder as described in Section \ref{sec:repAHA}. These are
both canonical bijections, so there is a unique left slanted map which makes the 
diagram commute, and it is also canonical. 
We want to identify it in terms of Hecke algebras.

Fix a KLR parameter $(\Phi,\rho)$ and recall from Theorem \ref{thm:S.3}.2 that the 
corresponding $X^* (T) \rtimes \mathcal W^H$-representation is  
\begin{equation}
\tau (t,x,\rho) = 
\mathrm{Hom}_{\pi_0 (\Z_H (t,x))} \big( \rho, H_{d(x)}(\mathcal B^{t,x}_H,\C) \big) .
\end{equation}
Similarly, by Theorem \ref{thm:S.5} the corresponding $\mathcal H (H)$-module 
is the unique irreducible quotient of the $\mathcal H (H)$-module
\begin{equation}\label{eq:Mtqx}
\mathrm{Hom}_{\pi_0 (\Z_H (t_q,x))} \big( \rho_q, H_*(\mathcal B^{t_q,x}_H,\C) \big) .
\end{equation}
In view of Proposition \ref{prop:S.1} both spaces are unchanged if we replace 
$t$ by $t_q$ and $\rho$ by $\rho_q$, and the vector space \eqref{eq:Mtqx} 
is also naturally isomorphic to
\begin{equation}\label{eq:N.4}
\mathrm{Hom}_{\pi_0 (\Z_H (\Phi))} 
\big( \rho, H_* (\mathcal B_H^{t,\Phi (B_2)} ,\C) \big) .
\end{equation}
Recall the asymptotic Hecke algebra $\mathcal J (H)$ from \cite{LuCellsIII}. 
We remark that, although in \cite{LuCellsIII} the underlying reductive group $H$ is 
supposed to be semisimple, this assumption is shown to be unnecessary in 
\cite{LuCellsIV}. Lusztig constructs canonical bijections
\begin{equation}\label{eq:bijectionsIrr}
\Irr (\mathcal H (H)) \longleftrightarrow \Irr (\mathcal J (H)) \longleftrightarrow
\Irr (X^* (T) \rtimes \cW^H ) 
\end{equation}
which we will analyse with our terminology.
According to \cite[Theorem 4.2]{LuCellsIV} 
$\Irr (\mathcal J (H))$ is naturally parametrized by the set of $H$-conjugacy classes
of Kazhdan--Lusztig triples for $H$. By Lemma \ref{lem:compareParameters} we can
also use KLR parameters, so may call the $\mathcal J (H)$-module with parameters 
$(t_q,x,\rho_q) \; \tilde \pi (\Phi,\rho)$. Its retraction to $\mathcal H (H)$ via
\begin{equation}\label{eq:N.10}
\cH(H)\overset{\phi_q}\longrightarrow  \cJ(H) 
\overset{\phi_1}\longleftarrow X^* (T) \rtimes \cW^H 
\end{equation}
is described in \cite[2.5]{LuCellsIV}. It is essentially the $\rho_q$-isotypical part of 
the $\langle t_q \rangle \times \C^\times$-equivariant K-theory of the variety $\cB^{t_q,x}$. 
With \cite[Theorem 6.2.4]{CG} this can be translated to the terminology of 
Section \ref{sec:repAHA}, and one can see that it is none other than \eqref{eq:Mtqx}. 

Recall that $\mathbf q$ is an indeterminate and let $\mathcal H_v (H) = \cH_{\mathbf q} (H) 
\otimes_{\C [\mathbf q ,\mathbf q^{-1}]} \C_v$ be the affine Hecke algebra with the same 
based root datum as $H$ and with parameter $v \in \C^\times$. Thus 
\[
\mathcal H_q (H) = \mathcal H (H) \quad  \text{and} \quad
\mathcal H_1 (H) = \C [X^* (T) \rtimes \mathcal W^H] .
\]
Like in \eqref{eq:N.27}, let $\tilde H$ be a central finite extension of $H$ whose derived
group is simply connected. By \eqref{eq:S.17} and \eqref{eq:S.19} 
\begin{equation}\label{eq:N.5}
\cH_v (H) \cong \big( K^{\tilde H \times \C^\times} (\mathcal Z_{\tilde H}) 
\otimes_{\C [\mathbf q,\mathbf q^{-1}]} \C_v \big)^{\ker (\tilde H \to H)} .
\end{equation}
The above, in particular \eqref{eq:Mtqx}, describes the retraction $\tilde \pi (\Phi,\rho)
\in \Irr (\mathcal J (H))$ to $\mathcal H_v (H)$ for any $v \in \C^\times$. 

In \cite[Corollary 3.6]{LuCellsIII} the $a$-function 
is used to single out a particular irreducible quotient $\mathcal H_v (H)$-module of 
\eqref{eq:Mtqx}. This applies when $v=1$ or $v$ is not a root of unity.
For $\cH_q (H)$ we saw in \eqref{eq:S.27} that there is only one such quotient,
which by definition is $\pi (t_q,x,\rho_q)$. This is our description of the left hand
side of \eqref{eq:bijectionsIrr}.

For $v=1$ we need a different argument. 
By the above and \eqref{eq:N.4} we obtain the $\mathcal H_1 (H)$-module
\begin{equation}\label{eq:S.38}
\mathrm{Hom}_{\pi_0 (\Z_H (t,x))} \big( \rho, H_* (\mathcal B^{t,x}_H,\C) \big) 
\end{equation}
with the action coming from \eqref{eq:N.5}, \eqref{eq:CG} and the convolution product
in Borel--Moore homology. Let us compare this with Kato's action \cite{Kat}, as 
described in Section \ref{sec:affSpringer}. On the subalgebra $\C [\cW^H]$ both are
defined in terms of Borel--Moore homology, respectively with $K^{\tilde H \times \C^\times}
(\mathcal Z_{\tilde H})$ and with $H (\mathcal Z_{\tilde H})$. 
It follows from \cite[(7.2.12)]{CG} that they agree. An element $\lambda \in X^* (T)$ 
acts via \eqref{eq:N.5} on K-theory as tensoring with a line bundle over 
$\cB_{\tilde H}^{\tilde x}$ canonically associated to $\lambda$, see \cite[p. 395]{CG}
or \cite[Theorem 3.5]{KL}. From the descriptions given in \cite[p. 420]{CG} and
\cite[\S 3]{Kat} we see that on \eqref{eq:S.38} this reduces to the action coming
from \eqref{eq:thetaGB}. In other words, we checked that the $\mathcal H_1 (H)$-module 
\eqref{eq:S.38} contains Kato's module \eqref{eq:KatoMod}, as the homology in
top degree. 

We want to see what the right hand bijection in \eqref{eq:bijectionsIrr} does to 
$\tilde \pi (\Phi,\rho)$. By construction it produces
a certain irreducible quotient of \eqref{eq:S.38}, namely the unique one with minimal 
$a$-weight. Unfortunately this is not so easy to analyse directly. Therefore we consider 
the opposite direction, starting with an irreducible $\mathcal H_1 (H)$-module 
$V$ with $a$-weight $a_V$. According to \cite[Corollary 3.6]{LuCellsIII} the 
$\mathcal J (H)$-module 
\[
\tilde V := \mathcal H_1 (H)^{a_V} \otimes_{\mathcal H_1 (H)} V ,
\]
is irreducible and has $a$-weight $a_V$. See \cite[Lemma 1.9]{LuCellsIII} for the precise
definition of $\tilde V$. 

Now we fix $t \in T$ and we will prove with induction to $\dim \cO_x$ that 
$\widetilde{\tau (t,x,\rho)}$ is none other than $\tilde \pi (\Phi,\rho)$. 
Our main tool is Lemma 
\ref{lem:S.7}, which says that the constituents of \eqref{eq:S.38} are $\tau (t,x,\rho)$
and irreducible representations corresponding to larger affine Springer parameters
(with respect to the partial order defined via the unipotent classes $\cO_x \subset M$).
For $\dim \cO_{x_0} = 0$ we see immediately that only the $\mathcal J (H)$-module
$\tilde \pi (t,x_0,\rho_0)$
can contain $\tau (t,x_0,\rho_0)$, so that must be $\widetilde{\tau (t,x_0,\rho_0)}$.
For $\dim \cO_{x_n} = n$ Lemma \ref{lem:S.7} says that \eqref{eq:S.38} can only contain 
$\tau (t,x_n,\rho_n)$ if $x \in \overline{\cO_{x_n}}$. But when $\dim \cO_x < n$ 
\[
\widetilde{\tau (t,x_n,\rho_n)} \not\cong \tilde \pi (\Phi,\rho) ,
\]
because the right hand side already is $\widetilde{\tau (t,x,\rho)}$, by the induction
hypothesis and the bijectivity of $V \mapsto \tilde V$. So the parameter of 
$\widetilde{\tau (t,x_n,\rho_n)}$ involves an $x$ with $\dim \cO_x = n$.
Then another look at Lemma \ref{lem:S.7} shows that moreover $(x,\rho)$ must be 
$M$-conjugate to $(x_n,\rho_n)$. Hence $\widetilde{\tau (t,x,\rho)}$ is indeed
\eqref{eq:S.38}.

We showed that the bijections \eqref{eq:bijectionsIrr} work out as
\begin{equation}\label{eq:N.17}
\begin{array}{ccccc}
\Irr (\mathcal H (H)) & \leftrightarrow & \Irr (\mathcal J (H)) & \leftrightarrow
& \Irr (X^* (T) \rtimes \cW^H ) \\
\pi (t_q,x,\rho_q) & \leftrightarrow & \tilde \pi (\Phi,\rho) &
\leftrightarrow & \tau (t,x,\rho) ,
\end{array}
\end{equation}
where all the objects in the bottom line are determined by the KLR parameter $(\Phi,\rho)$.
\end{proof}

\section{Main result (Hecke algebra version)}
\label{sec:mainH}

In this section $q \in \C^\times$ is allowed to be any element of infinite order.  
We study how Theorem \ref{thm:ps} can be extended to the algebras and modules from Section 
\ref{sec:repAHA}. So let $\Gamma$ be a group of automorphisms of $G$ that
preserves a chosen pinning, which involves $T$ as maximal torus. 
With the disconnected group $G \rtimes \Gamma$ we associate three kinds of parameters:
\begin{itemize}
\item The extended quotient of the second kind $(T /\!/ \mathcal W^G \rtimes \Gamma )_2$.
\item The space $\Irr (\mathcal H_q (G) \rtimes \Gamma)$ of equivalence classes of 
irreducible representations of the algebra $\mathcal H_q (G) \rtimes \Gamma$.
\item Equivalence classes of unramified Kazhdan--Lusztig--Reeder parameters. 
Let $\Phi : \mathbf W_F \times \SL_2 (\C) \to G$ be a group homomorphism with 
$\Phi (\mathbf I_F) = 1$ and $\Phi (\mathbf W_F) \subset T$. As in Section \ref{sec:Borel}, 
the component group 
\[
\pi_0 (\Z_{G \rtimes \Gamma}(\Phi)) = 
\pi_0 (\Z_{G \rtimes \Gamma}(\Phi (\mathbf W_F \times B_2)))
\]
acts on $H_* (\mathcal B^{\Phi (\mathbf W_F \times B_2)}_G ,\C)$. We take $\rho \in$
$\Irr \big( \pi_0 (\Z_{G \rtimes \Gamma}(\Phi)) \big)$ such that every irreducible
$\pi_0 (\Z_G (\Phi))$-subrepresentation of $\rho$ appears in $H_* (\mathcal B^{\Phi 
(\mathbf W_F \times B_2)}_G ,\C)$. The set \{KLR parameters for $G \rtimes \Gamma \}^{\unr}$ 
of pairs $(\Phi,\rho)$ carries an action of $G \rtimes \Gamma$ by conjugation. We consider 
the collection \{KLR parameters for $G \rtimes \Gamma \}^{\unr} / G \rtimes \Gamma$ 
of conjugacy classes $[\Phi,\rho ]_{G \rtimes \Gamma}$.
\end{itemize}

As in the proof of Theorem \ref{thm:ps}, let $\mathcal J (G)$ be the asymptotic Hecke 
algebra of $G$. The group $\Gamma$ acts on the extended affine Weyl group 
$X^*(T) \rtimes \cW^G$ in a length-preserving way. Hence every $\gamma \in \Gamma$ 
naturally determines an automorphism of $\mathcal J (G)$, as described in 
\cite[\S 1]{LuCellsIV}. 
This enables us to form the crossed product $\mathcal J (G) \rtimes \Gamma$.

\begin{thm}\label{thm:S.4}
There exists a commutative diagram of natural bijections
\[
\xymatrix{ 
& \hspace{-7mm} (T /\!/ \mathcal W^G \rtimes \Gamma )_2 \hspace{-7mm} \ar[dr]\ar[dl] & \\  
\Irr (\mathcal H_q (G) \rtimes \Gamma)    \ar[rr] & & 
\{ \textup{KLR parameters for } G \rtimes \Gamma \}^{\unr} / G \rtimes \Gamma } 
\]
It restricts to bijections between the following subsets:
\begin{itemize}
\item the ordinary quotient $T / (\cW^G \rtimes \Gamma) \subset
(T /\!/ \mathcal W^G \rtimes \Gamma )_2$,
\item the collection of spherical representations in $\Irr (\mathcal H_q (G) \rtimes \Gamma)$,
\item equivalence classes of KLR parameters $(\Phi,\rho)$ for $G \rtimes \Gamma$ with\\
$\Phi (\mathbf I_F \times \SL_2 (\C)) = 1$ and 
$\rho = \triv_{\pi_0 (\Z_{G \rtimes \Gamma}(\Phi))}$.
\end{itemize}
Moreover the left slanted map can be constructed via the (irreducible representations of)
the algebra $\mathcal J (G) \rtimes \Gamma$.
\end{thm}
\begin{proof}
The corresponding statement for $G$, proven in Theorem \ref{thm:ps}, 
is the existence of natural bijections
\begin{equation}\label{eq:S.13}
\xymatrix{
\Irr (\mathcal J (G)) \ar[d] \ar[rr] &  & (T /\!/ \mathcal W^G )_2 \ar[d] \ar[dll] \\  
\Irr (\mathcal H_q (G) )    \ar[rr] & & \{ \text{KLR parameters for } G \}^{\unr} / G} 
\end{equation}
Although in Section \ref{sec:main} $q$ was a prime power, we notice that among the
objects in \eqref{eq:S.13} only the algebra $\mathcal H_q (G)$ depends on $q$. Fortunately 
the bottom, slanted and left hand vertical maps in \eqref{eq:S.13} are defined equally 
well for our more general $q \in \C^\times$, as can be seen from the proofs of 
Theorems \ref{thm:S.5} and \ref{thm:ps}. Thus we may use \eqref{eq:S.13} 
as our starting point.

\emph{Step 1. The bijections in \eqref{eq:S.13} are $\Gamma$-equivariant.}\\
The action of $\Gamma$ on $(T // \mathcal W^G )_2$ can be written as
\begin{equation}\label{eq:S.14}
\gamma \cdot [t,\tilde \tau]_{\mathcal W^G} = 
[\gamma (t),\tilde \tau \circ \mathrm{Ad}_\gamma^{-1}]_{\mathcal W^G} .
\end{equation}
In terms of the multiplication in $G \rtimes \Gamma$, the action on KLR parameters is
\begin{equation}\label{eq:S.15}
\gamma \cdot [\Phi ,\rho_1]_G = 
[\gamma \Phi \gamma^{-1}, \rho_1 \circ \mathrm{Ad}_\gamma^{-1} ]_G
\end{equation}
We recall the right hand vertical map in \eqref{eq:S.13} from Theorem \ref{thm:S.3}. 
Write $M = \Z_G (t)$ and $\mathcal W^G_t = W (M^\circ,T) \rtimes \pi_0 (M)$. Then the
$\mathcal W^G_t$-representation $\tilde \tau$ can be written as $\tau (x,\rho_3) \rtimes 
\sigma$ for a unipotent element $x \in M^\circ$, a geometric $\rho_3 \in$ Irr$(\Z_{M^\circ}(x))$ 
and a $\sigma \in$ Irr$(\pi_0 (M)_{\tau (x,\rho_3)})$. The associated KLR parameter is 
$[\Phi,\rho_3 \rtimes \sigma]_G$, where $\Phi \matje{1}{1}{0}{1} = x$ and $\Phi$ maps a 
Frobenius element of $\mathbf W_F$ to $t$.
From \eqref{eq:S.7} we see that $\tau (x,\rho_3) \circ \mathrm{Ad}_\gamma^{-1}$ is equivalent
with $\tau (\gamma x \gamma^{-1},\rho_3 \circ \mathrm{Ad}_\gamma^{-1})$, so 
\[
\tilde \tau \circ \mathrm{Ad}_\gamma^{-1} \text{ is equivalent with } \tau (\gamma x \gamma^{-1},
\rho_3 \circ \mathrm{Ad}_\gamma^{-1}) \rtimes (\sigma \circ \mathrm{Ad}_\gamma^{-1}) .
\]
Hence \eqref{eq:S.14} is sent to the KLR parameter \eqref{eq:S.15}, which means that the right
hand vertical map in \eqref{eq:S.13} is indeed $\Gamma$-equivariant.

In view of Proposition \ref{prop:S.1} and \eqref{eq:S.15}, we already showed in \eqref{eq:S.23} 
that the lower horizontal map in \eqref{eq:S.13} is $\Gamma$-equivariant. By the commutativity 
of the triangle, so is the slanted map. 

As we checked in the proof of Theorem \ref{thm:ps},
the left hand vertical map is retraction along $\phi_q : \cH (G) \to \mathcal J (G)$ followed
by taking the unique irreducible quotient. The algebra homomorphism $\phi_q$ is 
$\Gamma$-equivariant because $\Gamma$ respects the entire setup in \cite[\S 1]{LuCellsIV}. 
Therefore the left hand vertical map is also $\Gamma$-equivariant.

\emph{Step 2. Suppose that $\tilde \pi (\Phi,\rho), [t,\tilde \tau]_{\mathcal W^G},\pi$ 
and $[\Phi,\rho_1]_G$ are four corresponding objects in \eqref{eq:S.13}. 
Then their stabilizers in $\Gamma$ coincide:}
\[
\Gamma_{\tilde \pi (\Phi,\rho)} = \Gamma_{[t,\tilde \tau]_{\mathcal W^G}} 
= \Gamma_\pi = \Gamma_{[\Phi,\rho_1]_G} .
\]
This follows immediately from step 1.

\emph{Step 3. Clifford theory produces 2-cocycles $\natural \big( \tilde \pi (\Phi,\rho) \big) 
,\; \natural \big( [t,\tilde \tau]_{\mathcal W^G} \big), \natural (\pi) $ and 
$\natural \big([\Phi,\rho_1]_G \big)$ of $\Gamma_x$. 
We can choose the same cocycle for all four of them.}\\
For $\natural (\pi)$ and $\natural \big([\Phi,\rho_1]_G \big)$ this was already checked in
\eqref{eq:S.25}, where we use Proposition \ref{prop:S.1} to translate between $\Phi$ and
$(t_q,x)$. 

From \eqref{eq:N.17}, and Theorems \ref{thm:S.3} and \ref{thm:S.5} we see that 
$\tilde \pi (\Phi,\rho), [t,\tilde \tau]_{\mathcal W^G}$ and $\pi$ come from three rather 
similar representations. The difference is that $\tilde \pi (\Phi,\rho)$ is built from the 
entire homology of a variety, whereas the other two are quotients thereof. The 
$\Gamma_\pi$-actions on these three modules are defined in the
same way, so the two cocycles can be chosen equal.

We remark that $\natural \big( [t,\tilde \tau]_{\mathcal W^G} \big)$ is trivial by
Proposition \ref{prop:S.2}, so the other 2-cocycles are also trivial.

\emph{Step 4. Upon applying $X \mapsto (X /\!/ \Gamma)^\natural_2$ to the commutative diagram 
\eqref{eq:S.13} we obtain the corresponding diagram for $G \rtimes \Gamma$.}\\
Here $\natural$ denotes the family of 2-cocycles constructed in steps 2 and 3.
For $\Irr (\mathcal J (G) ,\; (T /\!/ \mathcal W^G )_2$ and $\Irr(\mathcal H_q (G))$ 
we know from Lemmas \ref{lem:Clifford} and \ref{lem:Clifford_algebras}
that this procedure yields the correct parameters. That it works for Kazhdan--Lusztig--Reeder
parameters was checked in the last part of the proof of Theorem \ref{thm:S.5}. By steps 1 and
3 the construction used in \eqref{eq:twisting} yields the same homomorphisms between
the twisted group algebras (called $\phi_{\gamma,x}$ in Section \ref{sec:extquot}) in all four
settings. Hence the maps from \eqref{eq:S.13} 
can be lifted in a natural way to the diagram for $G \rtimes \Gamma$.

The ordinary quotient is embedded in $(T /\!/ \mathcal W^G \rtimes \Gamma )_2$ as the
collection of pairs $\big( t,\triv_{(\cW^G \rtimes \Gamma)_t} \big)$. By an obvious
generalization of \eqref{eq:trivWaff} these correspond to the affine Springer parameters
$(t,x=1,\rho = \triv)$. It is clear from the above construction that they are mapped to
KLR parameters $(\Phi,\triv)$ with $\Phi (\mathbf I_F \times \SL_2 (\C)) = 1$ and 
$\Phi (\varpi_F) = t$. By Proposition \ref{prop:spherical} the latter correspond
to the spherical irreducible $\cH (G) \rtimes \Gamma$-modules. 
\end{proof}

\section{Canonicity}
\label{sec:can}

We return to the notation from Section \ref{sec:Roche}. We would like to combine
Theorems \ref{thm:ps} and \ref{thm:S.4} to a version that applies to $\Irr (\cG)^\fs$
irrespective of the (dis)connectedness of $H = \Z_G (\mathrm{im} c^\fs)$.
We have observed already that everything in Theorem \ref{thm:S.4} is canonical, but
we do not know yet how canonical Theorem \ref{thm:Roche} is. Unfortunately a discussion
of this issue is avoided in the sources \cite{Roc} and \cite{R}.

For this purpose we need some technical results about the extended affine Hecke algebra 
$\cH (H)$. Let us denote the elements of the Bernstein basis of $\cH (H)$ by
$\theta_\lambda T_w$, where $\lambda \in X^* (T)$ and $w \in \cW^H$. 
The algebra $\cH (T)$ is canonically isomorphic to $\cO (T) = \C [X^* (T)]$, 
so it has a basis $\{[\lambda] : \lambda \in X^* (T)\}$. 
The assignment $[\lambda] \mapsto \theta_\lambda$ determines an algebra injection
\[
t_U : \cH (T) \cong \cO (T) \to \cH (H) . 
\]
It is canonical in the sense that it depends only on the based root datum of $(H,T)$,
which was fixed by the choice of a Borel subgroup $B_H = B \cap H$. 
Via $t_U$ we regard $\cO (T)$ as a subalgebra of $\cH (H)$. It is well-known from
\cite[\S 3]{LuGrad} that the centre of $\cH (H)$ is $\cO (T)^{\cW^H}$. Let
$\C (T)$ be the field of rational functions on $T$, the quotient field of 
$\cO (T)$. Then $\cH (H) \otimes_{\Z (\cH (H))} \C (T)^{\cW^H}$
carries a natural algebra structure, and as a vector space it is simply
\[
\cH (H) \otimes_{\cO (T)} \C (T) \cong \C (T) \rtimes \cW^H . 
\]
By \cite[\S 6]{LuGrad} or \cite[Proposition 1.5.1]{Sol} there is an algebra
isomorphism
\begin{equation}\label{eq:N.12}
\cH (H) \otimes_{\Z (\cH (H))} \C (T)^{\cW^H}  \cong \C (T) \rtimes \cW^H ,
\end{equation}
which is the identity on $\cO (T)$. 

\begin{prop}\label{prop:autH}
Let $\phi$ be an automorphism of $\cH (H)$ which is the identity on $\cO (T)$.
\begin{enumerate}
\item It induces an automorphism (also denoted by $\phi$) of $\C (T) \rtimes \cW^H$.
\item There exist $z_w \in \C^\times$ and $\lambda_w \in X^* (T)$ such that
$\phi (w) = z_w \theta_{\lambda_w} w$ for all $w \in \cW^H$.
\item For every reflection $s_\alpha$ with $\alpha \in R(H^\circ ,T)$ we have
$\lambda_{s_\alpha} \in \mathbb Z \alpha$.
\item $z_w = 1$ for $w \in \cW^{H^\circ}$, and $w \mapsto z_w$ is a character of
$\pi_0 (H) \cong \cW^H / \cW^{H^\circ}$.
\end{enumerate}
\end{prop}
\begin{proof}
(1) is a direct consequence of \eqref{eq:N.12}. 
By assumption $\phi$ is the identity on the quotient field $\C (T)$ of $\cO (T)$.
Hence it is of the form
\begin{equation}\label{eq:N.13}
\phi : \sum_{w \in \cW^H} f_w w \mapsto \sum_{w \in \cW^H} f_w \Phi_w w 
\end{equation}
for suitable $\Phi_w \in \C (T)$. Let $\imath^\circ_w \in \cH (H) 
\otimes_{\Z (\cH (H))} \C (T)^{\cW^H}$ be the image of $w \in \cW^H$ under 
\eqref{eq:N.12}. An explicit formula in the case of a simple reflection $s_\alpha$
is given in \cite[(1.25)]{Sol}:
\begin{equation} \label{eq:N.14}
1 + T_{s_\alpha} = q \frac{\theta_\alpha - q^{-1}}{\theta_\alpha - 1} 
(1 + \imath^\circ_{s_\alpha}) .
\end{equation}
Since $\phi$ preserves $\cH (H)$, we see from \eqref{eq:N.13} and \eqref{eq:N.14}
that $\Phi_{s_\alpha} \in \cO (T)^\times = \C^\times X^* (T)$. Say 
$\Phi_{s_\alpha} = z \theta_\lambda$. Then we calculate in $\C(T) \rtimes \cW^H$:
\[
1 = s_\alpha^2 = \phi (s_\alpha)^2 = z \theta_\lambda s_\alpha z \theta_\lambda 
s_\alpha = z^2 \theta_\lambda \theta_{s_\alpha (\lambda)} s_\alpha^2 =
z^2 \theta_{\lambda + s_\alpha (\lambda)} .
\]
Therefore $z = \pm 1$ and $s_\alpha (\lambda) = - \lambda$, which means that
$\lambda \in \Q \alpha \cap X^* (T)$. Now
\begin{align*}
\phi (1 + T_{s_\alpha}) & = \phi \Big( \frac{q \theta_\alpha - 1}{\theta_\alpha - 1} 
(1 + \imath^\circ_{s_\alpha}) \Big) \\
& = \frac{q \theta_\alpha - 1}{\theta_\alpha - 1} 
(1 + z \theta_\lambda \imath^\circ_{s_\alpha}) \\
& = \frac{q \theta_\alpha - 1}{\theta_\alpha - 1} ( 1 - z \theta_\lambda ) +
z \theta_\lambda \frac{q \theta_\alpha - 1}{\theta_\alpha - 1} 
(1 + \imath^\circ_{s_\alpha}) \\
& = \frac{q \theta_\alpha - 1}{\theta_\alpha - 1} ( 1 - z \theta_\lambda ) +
z \theta_\lambda (1 + T_{s_\alpha}) .
\end{align*}
This is an element of $\cH (H)$ and $q > 1$, so $\theta_\alpha - 1$ divides 
$1 - z \theta_\lambda$ in $\cO (T)$. We deduce that $z = +1$ and 
$\lambda = \lambda_{s_\alpha} \in \mathbb Z \alpha$. In particular
\[
\phi (\imath^\circ_{s_\alpha}) = 
\theta_{\lambda_{s_\alpha}} \imath^\circ_{s_\alpha} ,
\]
which directly implies that for every $w \in \cW^{H^\circ}$ there exists a
$\lambda_w \in X^* (T)$ with $\phi (\imath^\circ_w) = \theta_{\lambda_w}
\imath^\circ_w$. \\
If $w \in \cW^{H^\circ}$ is any reflection, then $w = s_\beta$ for some
$\beta \in R(H^\circ,T)$ and $w$ is conjugate to some
simple reflection $s_\alpha$, say by $v \in \cW^{H^\circ}$. Then
\begin{equation}
\begin{split}
\theta_{\lambda_{s_\beta}} s_\beta = \phi (s_\beta) = \phi (v s_\alpha v^{-1}) =  
\theta_{\lambda_v} v \theta_{\lambda_{s_\alpha}} s_\alpha v^{-1} 
\theta_{- \lambda_v} \\
= \theta_{\lambda_v} \theta_{v(\lambda_{s_\alpha})} 
\theta_{v s_\alpha v^{-1} (-\lambda_v)} v s_\alpha v^{-1} =
\theta_{v(\lambda_{s_\alpha}) + \lambda_v - s_\beta (\lambda_v)} s_\beta \\
= \theta_{v(\lambda_{s_\alpha}) + \langle \beta^\vee , \lambda_v \rangle \beta}
s_\beta .
\end{split}
\end{equation}
As $v (\lambda_{s_\alpha}) \in v (\mathbb Z \alpha) = \mathbb Z \beta$, we
see that $\theta_{\lambda_{s_\beta}} \in \mathbb Z \beta$. This proves
(ii), (iii) and (iv) on $\cW^{H^\circ}$. 

Recall from Lemma \ref{lem:disconnected} that $\cW^H \cong \cW^{H^\circ} 
\rtimes \pi_0 (H)$, with $\pi_0 (H)$ preserving the simple roots. For $w
\in \pi_0 (H)$ we have $\imath^\circ_w = T_w$ by \cite[Proposition 1.5.1]{Sol},
so the argument from \eqref{eq:N.13} and \eqref{eq:N.14} shows that $\Phi_w \in
\cO (T)^\times$. Therefore (ii) holds on $\cW^H$. Knowing this, the 
multiplication rules in $\C (T) \rtimes \cW^H$ entail that $w \mapsto z_w$ 
must be a character of $\cW^H$.
\end{proof}

To investigate the effect of automorphisms as in Proposition \ref{prop:autH}
on $\cH (H)$-modules, we take a closer look at \eqref{eq:N.10}. Let
$\cH_{\mathbf{\sqrt q}}(H)$ be the affine Hecke algebra with the same data 
as $\cH (H)$, but with a formal parameter $\mathbf q$ and over the ground 
ring $\C [\mathbf{q^{\pm 1/2}}]$. This algebra also has a Bernstein presentation
and a Bernstein basis like $\cH (H)$, only over $\C [\mathbf{q^{\pm 1/2}}]$.
Any $\phi$ as in Proposition \ref{prop:autH} lifts to a 
$\C [\mathbf{q^{\pm 1/2}}]$-linear automorphism of $\cH_{\mathbf{\sqrt q}}(H)$,
just use the same formula as in part (ii).

Like in \eqref{eq:HeckeH} we define 
\[
\mathcal J (H) = \mathcal J (H^\circ) \rtimes \pi_0 (H) . 
\]
In \cite[\S 2.4]{LuCellsII} a homomorphism
of $\C [\mathbf{q^{\pm 1/2}}]$-algebras
\[
\cH_{\mathbf{\sqrt q}}(H^\circ) \to 
\mathcal J (H^\circ) \otimes_\C \C [\mathbf{q^{\pm 1/2}}]
\]
is constructed, which induces \eqref{eq:N.10} by specialization of 
$\mathbf{q^{1/2}}$ at $q^{1/2}$ or at 1. Because the actions of $\pi_0 (H)$
preserve all the data used to construct these algebras, it induces a
homomorphism of $\C [\mathbf{q^{\pm 1/2}}]$-algebras
\[
\tilde \phi : \cH_{\mathbf{\sqrt q}}(H) \to 
\mathcal J (H) \otimes_\C \C [\mathbf{q^{\pm 1/2}}] .
\]
From the $\mathcal J (H)$-module $\tilde \pi (\Phi,\rho)$ and $\tilde \phi$ 
we obtain the $\cH_{\mathbf{\sqrt q}}(H)$-module 
\begin{equation}\label{eq:N.15}
\tilde \pi (\Phi,\rho) \otimes_\C \C [\mathbf{q^{\pm 1/2}}] .
\end{equation}
We call modules of this form, for any KLR parameter $(\Phi,\rho)$ with 
$\Phi \big|_{\fo_F^\times} = c^\fs$, standard $\cH_{\mathbf{\sqrt q}}(H)$-modules.

\begin{lem}\label{lem:Smod}
Let $\phi$ be any automorphism of the $\C [\mathbf{q^{\pm 1/2}}]$-algebra
$\cH_{\mathbf{\sqrt q}}(H)$. The induced map $\phi^*$ on 
$\Mod (\cH_{\mathbf{\sqrt q}}(H))$ sends standard modules to standard modules.
\end{lem}
\begin{proof}
We saw in the proof of Theorem \ref{thm:ps} that for every generic $v \in \C^\times$ 
the specialization of \eqref{eq:N.15} at $\mathbf{q^{1/2}}$ at $v^{1/2}$ is an 
irreducible $\cH_v (H)$-module, namely $\pi (t_v,x,\rho_v)$. All these modules
have the same underlying vector space 
\[
\mathrm{Hom}_{\pi_0 (\Z_H (\Phi))} 
\big( \rho, H_* (\mathcal B_H^{t,\Phi (B_2)} ,\C) \big) ,
\]
and the action of $\cH_v (H)$ depends algebraically on $v^{\pm 1/2}$. It follows 
from Section \ref{sec:repAHA} that, for generic $v$, there is only one way to embed
$\pi (t_v,x,\rho_v)$ in a family of irreducible $\cH_v (H)$-modules that depends 
algebraically on $v^{\pm 1/2}$ (varying in this generic set of parameters). Since
$\phi$ is a $\C [\mathbf{q^{\pm 1/2}}]$-algebra automorphism, 
\begin{equation}\label{eq:N.16}
\phi^* \big( \tilde \pi (\Phi,\rho) \otimes_\C \C [\mathbf{q^{\pm 1/2}}] \big)
\end{equation}
has irreducible specializations at all generic $v \in \C^\times$, and these still
depend algebraically on $v^{\pm 1/2}$. So \eqref{eq:N.16} looks like a standard
module as long as only generic parameters are considered, say like 
$\tilde \pi (\Phi',\rho') \otimes_\C \C [\mathbf{q^{\pm 1/2}}]$. But the set of
generic parameters in dense in $\C^\times$, so
\[
\phi^* \big( \tilde \pi (\Phi,\rho) \otimes_\C \C [\mathbf{q^{\pm 1/2}}] \big)
= \tilde \pi (\Phi',\rho') \otimes_\C \C [\mathbf{q^{\pm 1/2}}]. \qedhere
\]
\end{proof}

Recall the parametrization of irreducible $\cH (H)$-modules in Theorem \ref{thm:S.5}.

\begin{lem}\label{lem:actionOnParameters}
Let $\phi$ be an automorphism of $\cH (H)$ which is the identity on $\cO (T)$.
For every Kazhdan--Lusztig triple $(t_q,x,\rho_q)$ there exists a geometric
$\rho'_q \in \Irr \big( \pi_0 (\Cent_G (t_q,x)) \big)$
such that $\phi^* (\pi (t_q,x,\rho_q)) = \pi (t_q,x,\rho'_q)$.
\end{lem}
\begin{proof}
Consider the standard $\cH_{\mathbf{\sqrt q}}(H)$-module \eqref{eq:N.15}, where
$(\Phi,\rho)$ is associated to $(t_q,x,\rho_q)$ via Lemma \ref{lem:compareParameters}.
Its specialization at $\mathbf{q^{1/2}} = q^{1/2}$ is a $\cH_q (H)$-module with
$\pi (t_q,x,\rho_q)$ as unique irreducible quotient. On the other hand, the 
specialization at $\mathbf{q^{1/2}} = 1$ is the $\C [X^* (T) \rtimes \cW^H]$-module
\[
\mathrm{Hom}_{\pi_0 (\Z_H (t,x))} \big( \rho_1, H_* (\mathcal B^{t,x}_H,\C) \big) . 
\]
By Theorem \ref{thm:S.3} its component in top homological degree $d(x)$ is
\begin{align*}
& \tau (t,x,\rho_1) = \mathrm{ind}_{X^* (T) \rtimes \cW^M}^{X^* (T) \rtimes \cW^H}
\tau^\circ (t,x,\rho_1)) \in \Irr (X^* (T) \rtimes \cW^H) , \\
& \tau^\circ (t,x,\rho_1) = \mathrm{Hom}_{\pi_0 (\Z_M (x))} 
\big( \rho_1, H_{d(x)} (\mathcal B^x_{M^\circ},\C) \otimes \C [\pi_0 (M)] \big) ,
\end{align*}
where $M = \Z_H (t)$. Via Proposition \ref{prop:autH} $\phi$ determines an 
automorphism of $\cH_1 (H) = \C [X^* (T) \rtimes \cW^H]$, and we want to know
$\phi^* (\tau (t,x,\rho_1))$. Since $\phi$ is the identity on $\cO (T)$, composition
with it does not change the parameter $t \in T$. 

In $\tau^\circ (t,x,\rho_1) \in \Irr (\cW^M)$ the unipotent class of $x \in M^\circ$ 
is already determined by the action of $\cW^{M^\circ}$, see Proposition \ref{prop:S.2} 
and Theorem \ref{thm:SpringerExtended}. Recall from \cite[\S 4.1]{SpringerSteinberg}
that $M^\circ$ is generated by the reflections $s_\alpha$ with $\alpha (t) = 1$. By
Propostion \ref{prop:autH}.(4) $\phi (s_\alpha) = \theta_{\lambda_{s_\alpha}} s_\alpha$
for some $\lambda_{s_\alpha} \in \mathbb Z \alpha$. We calculate
\begin{multline*}
\tau^\circ (t,x,\rho_1) \circ \phi^{-1} (s_\alpha) = \tau^\circ (t,x,\rho_1) 
(s_\alpha \theta_{\lambda_{s_\alpha}}) \\
= \tau^\circ (t,x,\rho_1) (s_\alpha) \theta_{\lambda_{s_\alpha}} (t) 
= \tau^\circ (t,x,\rho_1) (s_\alpha) .
\end{multline*}
Thus $\tau^\circ (t,x,\rho_1) \circ \phi^{-1} \big|_{\cW^{H^\circ}} =
\tau^\circ (t,x,\rho_1) \big|_{\cW^{H^\circ}}$ and
\begin{equation*}
\tau^\circ (t,x,\rho_1) \circ \phi^{-1} = \tau^\circ (t,x,\rho'_1)
\end{equation*}
for a geometric $\rho'_1 \in \Irr \big( \pi_0 (\Cent_M (x)) \big)$. It follows that
\begin{equation}\label{eq:N.19}
\phi^* \big( \tau (t,x,\rho_1) \big) = \tau (t,x,\rho'_1) .
\end{equation}
Lift $\phi$ to automorphism of $\cH_{\mathbf{\sqrt q}}(H)$, using Proposition
\ref{prop:autH}. By Lemma \ref{lem:Smod} 
\begin{equation}\label{eq:N.18}
\phi^* \big( \tilde \pi (\Phi,\rho) \otimes_\C \C [\mathbf{q^{\pm 1/2}}] \big)
\end{equation}
is again a standard $\cH_{\mathbf{\sqrt q}}(H)$-module. But there is only one
standard $\cH_{\mathbf{\sqrt q}}(H)$-module whose specialization at 
$\mathbf{q^{1/2}} = 1$ has $\tau (t,x,\rho'_1)$ as component in top homological
degree, namely the one with parameter $(\Phi,\rho')$. By \eqref{eq:N.18} the module
\eqref{eq:N.19} must be isomorphic to
\[
\tilde \pi (\Phi,\rho') \otimes_\C \C [\mathbf{q^{\pm 1/2}}] .
\]
In particular its specialization at $\mathbf{q^{1/2}} = q^{1/2}$ is 
\[
\phi^* \Big( \mathrm{Hom}_{\pi_0 (\Z_H (t_q,x))} 
\big( \rho_q, H_*(\mathcal B^{t_q,x}_H,\C) \big) \Big)
\cong \mathrm{Hom}_{\pi_0 (\Z_H (t_q,x))} 
\big( \rho'_q, H_*(\mathcal B^{t_q,x}_H,\C) \big) ,
\]
which has $\pi (t_q,x,\rho'_q)$ as unique irreducible quotient. Consequently\\
$\phi^* (\pi (t_q,x,\rho_q)) \cong \pi (t_q,x,\rho'_q)$.
\end{proof}

We shall apply Lemma \ref{lem:actionOnParameters} to Theorem \ref{thm:Roche}.
The main role in the proof of that Theorem is played by a cover $(J,\tau)$
of $(\cT_0,\chi |_{\cT_0})$. Let $I_\cB^\cG$ denote the normalized parabolic
induction functor, starting from the Borel subgroup $\cB \subset \cG$
corresponding to $B \subset G$. 
As shown in \cite[\S 9]{Roc}, there exists an algebra injection
\[
t_B : \cH (\cT,\chi |_{\cT_0}) \to \cH (\cG, \tau) 
\]
such that the following diagram commutes
\begin{equation}\label{eq:N.11}
\xymatrix{
\Rep (\cG)^\fs \ar[r]^\sim & \Mod (\cH (\cG,\tau)) \ar[r]^\sim & \Mod (\cH (H)) \\
\Rep (\cT)^{[\cT,\chi]_\cT} \ar[u]^{I_\cB^\cG} \ar[r]^\sim &
\Mod (\cH(\cT,\chi |_{\cT_0})) \ar[u]^{t_{B*}} \ar[r]^\sim &
\Mod (\cH (T)) \ar[u]^{t_{U*}}
}
\end{equation}
Here $t_{U*}(V) = \mathrm{Hom}_{\cH T)}(\cH (H),V)$, and similarly for $t_{B*}$.

\begin{lem}\label{lem:compareTypes}
Let $B'$ be a Borel subgroup of $G$ such that $B' \cap H^\circ = B \cap H^\circ$.
Suppose that $(J',\tau')$ is another $\fs$-type which covers $(\cT_0,\chi |_{\cT_0})$,
and that there exists an isomorphism $\cH (\cG,\tau') \cong \cH (H)$ that makes
the diagram analogous to \eqref{eq:N.11}, but with primes, commute. Then the map
\[
\Irr (\cG)^\fs \to \Irr (\cH (\cG,\tau')) \to \Irr (\cH (H)) \to
\{\text{$\KLR$ parameters} \}^\fs / H 
\]
can only differ from its counterpart for $(J,\tau)$ in third ingredient $\rho$
of a KLR parameter.
\end{lem}
\begin{proof}
Let us denote the copy of $\cH (H)$ obtained from $\cH (\cG,\tau')$ by $\cH' (H)$,
to distinguish it from the earlier $\cH (H)$. The assumptions of the lemma entail
an equivalence of categories 
\begin{equation}\label{eq:N.20}
\Mod (\cH (H)) \longleftrightarrow \Mod (\cH' (H)) ,
\end{equation}
which sends any $\cH (H)$-module induced from $\cH (T)$ to an isomorphic
$\cH' (H)$-module. In particular the regular representation of $\cH (H)$ is mapped
to an $\cH' (H)$-module isomorphic to the regular representation. Let $\mathcal M$
be a Morita $\cH' (H) - \cH (H)$-bimodule that implements \eqref{eq:N.20}. Then 
$\cH (H) \cong \mathrm{End}_{\cH' (H)}(\mathcal M)$ as algebras and
\[
\mathcal M \cong \mathcal M \otimes_{\cH (H)} \cH (H) \cong \cH' (H) 
\]
as $\cH' (H)$-modules. Hence
\begin{equation}\label{eq:N.21}
\cH (H) \cong \mathrm{End}_{\cH' (H)}(\cH' (H)) \cong \cH' (H) , 
\end{equation}
providing an algebra isomorphism that has the same effect as \eqref{eq:N.20}. 
The map $t_U : \cH (T) \to \cH (H)$ depends only the choice of a positive 
system in $R(H^\circ,T)$, so it is the same for $B'$ and $B$. From that and
\eqref{eq:N.11} we see that \eqref{eq:N.21} is the identity on 
$\cO (T) = t_U (\cH (T))$. 
By Lemma \ref{lem:actionOnParameters} composition with this isomorphism sends an
irreducible $\cH (H)$-module $\pi (t_q,x,\rho_q)$ to $\pi (t_q,x,\rho'_q)$ for
some $\rho'_q$. This statement is just another way to formulate the lemma. 
\end{proof}

Now we can answer the questions raised by \eqref{eq:N.8} and 
Example \ref{ex:noncanonical}.

\begin{prop}\label{prop:canonicity}
Assume that Condition \ref{con:char} holds. Consider the bijections 
\[
\Irr (\cG)^\fs \to \Irr (\cH (H)) \to \{ \text{$\KLR$ parameters} \}^\fs / H 
\]
from Theorems \ref{thm:Roche} and \ref{thm:S.5}. Suppose that $\pi \in 
\Irr (\cG)^\fs$ is mapped to $[t_q,x,\rho_q]_H$. Then the $H$-conjugacy class
of $(t_q,x)$ is uniquely determined by the condition that the equivalence of
categories $\Rep (\cG)^\fs \cong \Mod (\cH (H))$ comes from an $\fs$-type
which is a cover of $(\cT_0, \chi |_{\cT_0} )$ for some $\chi \in \Irr (\cT)$ 
with $[\cT,\chi]_\cG = \fs$.
\end{prop}
\begin{proof}
Most of the work was done in \ref{lem:compareTypes} and Theorem 
\ref{thm:Roche}. We only need to show that it does not depend on the choice
of a Borel subgroup $\cT \subset \cB \subset \cG$, or equivalently of a Borel
subgroup $T \subset B \subset G$. Any other Borel subgroup of $\cG$ containing
$\cT$ is of the form $\cB' = w \cB w^{-1}$ for a unique $w \in \cW^G$. If we
would use $\cB'$ instead of $\cB$, we would end up studying the induced 
representation $I_{\cB'}^\cG (\chi)$ with the extended affine Hecke algebra
$\cH (H,w B w^{-1} \cap H^\circ)$, whose based root datum is that of 
$(H, w B w^{-1} \cap H^\circ)$. An irreducible constituent $\pi'$ of 
$I_{\cB'}^\cG (\chi)$ would then produce a KLR parameter $[t'_q,x',\rho'_q]_H$.
In this setting $w B w^{-1} \cap H^\circ$ is conjugate to $B_H = B \cap H^\circ$
by an element $h \in N_{H^\circ} (T)$, unique up to $T$. Let $\gamma$ be its image
in $\cW^{H^\circ} \subset \cW^G$. The map 
\[
\theta_\lambda T_u \mapsto \theta_{\gamma (\lambda)} T_{\gamma u \gamma^{-1}}
\]
on the Bernstein bases determines an algebra isomorphism 
\[
\cH (H,w B w^{-1} \cap H^\circ) \to \cH (H) = \cH (H,B_H). 
\]
Conjugating the entire situation by $h$, we obtain a constituent 
\[
\gamma \cdot \pi' \cong \pi' \quad \text{of} \quad I_{\gamma \cB' \gamma^{-1}}^{\cG}
(\gamma \cdot \chi) \cong I_{\cB'}^\cG (\chi) ,
\]
and an $\cH (H)$-module with KLR parameter 
\[
[h t'_q h^{-1},h x' h^{-1}, h \cdot \rho'_q]_H = [t'_q,x',\rho'_q]_H .
\]
Notice that the Borel subgroup $B'' := h B' h^{-1} = 
h w B w^{-1} h^{-1}$ satisfies $B'' \cap H^\circ = B \cap H^\circ$.
Any type which we used the produce the KLR parameters can also be conjugated
by a lift of $\gamma$ in $N_\cG (\cT)$, and that yield a type which covers
$(\cT_0, \gamma \cdot \chi |_{\cT_0} )$. By Theorem \ref{thm:Roche}.(2) that
is just as good as a cover of $(\cT_0, \chi |_{\cT_0} )$. Thus we have reduced
to the situation of Lemma \ref{lem:compareTypes}.
\end{proof}

We remark that the conditions imposed in Proposition \ref{prop:canonicity}
seem quite reasonable. We need the algebra $\cH (H)$ to arrive at the right
parameter space, and we use the covering of $(\cT_0,\chi |_{\cT_0})$ to get
a relation between $\cH (T) \to \cH (H)$ and the normalized parabolic induction 
functor $I_\cB^\cG$.

\section{Main result (general case)}
\label{sec:general}

The preparations for our main theorem are now complete.

\begin{thm}\label{thm:main}
Let $\cG$ be a  split reductive $p$-adic group and let $\fs = [\mathcal T, \chi]_\cG$ 
be a point in the Bernstein spectrum of the principal series of $\cG$. Assume that 
Condition \ref{con:char} holds. Then there is a commutative triangle of bijections 
\[ 
\xymatrix{   & (T^\fs/\!/W^\fs)_2 \ar[dr]\ar[dl] & \\  
\Irr(\mathcal{G})^\fs    \ar[rr] & & \{\KLR\:\:\mathrm{parameters}\}^\fs / H} 
\] 
The slanted maps are generalizations of the slanted maps in Theorem \ref{thm:ps}
and the horizontal map stems from Theorem \ref{thm:S.5}. 
The right slanted map is natural. 
If $\pi \in \Irr (\cG)^\fs$ corresponds to a KLR parameter $(\Phi,\rho)$, then the
Langlands parameter $\Phi$ is determined canonically by $\pi$.

We denote the irreducible $\cG$-representation associated to a KLR parameter 
$(\Phi,\rho)$ by $\pi (\Phi,\rho)$.
\begin{enumerate}
\item The infinitesimal central character of $\pi (\Phi,\rho)$ is the $H$-conjugacy class
\[
\Phi \big( \varpi_F , \matje{q^{1/2}}{0}{0}{q^{-1/2}} \big) 
\in c(H)_{\ss} \cong T^\fs / W^\fs .
\]
\item $\pi (\Phi,\rho)$ is tempered if and only if $\Phi (\mathbf W_F)$ is bounded, which
is the case if and only if $\Phi (\varpi_F)$ lies in a compact subgroup of $H$.
\end{enumerate}
\end{thm}
\begin{proof}
By Proposition \ref{prop:canonicity} any $\pi \in \Irr (\cG)^\fs$
canonically determines a Langlands parameter $\Phi$.
The larger part of the commutative triangle was already discussed in \eqref{eq:WsH}, 
\eqref{eq:HeckeH} and Theorem \ref{thm:S.4}. It remains to show that the set 
\{KLR parameters$\}^\fs / H$ (as defined on page \pageref{eq:defKLRparameter}) 
is naturally in bijection with
\{KLR parameters for $H^\circ \rtimes \pi_0 (H) \}^{\unr} / H^\circ \rtimes \pi_0 (H)$.

By \eqref{eq:splitH} we are taking conjugacy classes with respect to the
group $H / \Z(H^\circ)$ in both cases.
It is clear from the definitions that that in both sets the ingredients $\Phi$
are determined by the semisimple element $\Phi (\varpi_F) \in H$. This provides the
desired bijection between the $\Phi$'s in the two collections, so let us focus on the
ingredients $\rho$. 

For $(\Phi,\rho) \in \{\text{KLR parameters}\}^\fs$ the irreducible representation $\rho$ 
of the component group $\pi_0 (\Z_H (\Phi)) = \pi_0 (\Z_G (\Phi))$ must appear in \\
$H_* \big( \mathcal B_G^{\Phi (\mathbf W_F \times B_2)} ,\C \big)$. By Proposition
\ref{prop:UP}.3 this space is isomorphic, as a $\pi_0 (\Z_G (\Phi))$-representation, to
a number of copies of
\[
\mathrm{Ind}_{\pi_0 (\Z_{H^\circ} (\Phi))}^{\pi_0 (\Z_G (\Phi))} H_* \big( 
\mathcal B_{H^\circ}^{\Phi (\mathbf W_F \times B_2)} ,\C \big) .
\]
Hence the condition on $\rho$ is equivalent to requiring that every irreducible
$\pi_0 (\Z_{H^\circ} (\Phi))$-subrepresentation of $\rho$ appears in
$H_* \big( \mathcal B_{H^\circ}^{\Phi (\mathbf W_F \times B_2)} ,\C \big)$. That 
is exactly the condition on $\rho$ in an unramified KLR parameter for 
$H^\circ \rtimes \pi_0 (H)$. This establishes the properties of the commutative diagram.\\
(1) By Theorem \ref{thm:S.5} the $\mathcal H (H^\circ)$-module with Kazhdan--Lusztig 
triple $(t_q,x,\rho_q)$ has central character
\[
t_q = \Phi \big( \varpi_F , \matje{q^{1/2}}{0}{0}{q^{-1/2}} \big) 
\in c(H^\circ)_{\ss} \cong T / \mathcal W^{H^\circ} .
\]
It follows that the $\mathcal H (H)$-module with parameter $(t_q,x,\rho_q)$ or $(\Phi,\rho)$ 
has central character $t_q \in c(H)_{\ss} \cong T / \cW^H$. Via \eqref{eq:N.22} we can 
also consider it as an element of $T^\fs / W^\fs$. In view of \eqref{eq:N.6} and 
\eqref{eq:N.45} the corresponding $\cG$-representation is
\begin{equation}\label{eq:N.46}
\pi (\Phi,\rho) = \cH (\cG) \otimes_{\cH (H)} \pi (t_q,x,\rho_q) .
\end{equation}
This tensor product defines an equivalence between $\Mod (\cH (H))$ and $\Rep (\cG)^\fs$,
which by definition transforms the central character into the infinitesimal character.\\
(2) It was checked in \cite[Theorem B]{BHK} that a $\pi (\Phi,\rho) \in \Irr (\cG)^\fs$ 
is tempered if and only if the corresponding $\mathcal H (H)$-module $\pi (t_q,x,\rho_q)$ 
is tempered. By Proposition \ref{prop:tempered} the latter is tempered if and only if 
$\Phi (\varpi_F) \in T$ is compact.
Since $\Phi (\mathbf W_F)$ is generated by the finite group $\Phi (\mathbf I_F)$ and 
$\Phi (\varpi_F)$, the above condition on $\Phi (\varpi_F)$ is equivalent to 
boundedness of $\Phi (\mathbf W_F)$.
\end{proof}

\section{A local Langlands correspondence}
\label{sec:LLC}

As in the introduction, $\Irr( \cG, \cT)$ denotes the space of all irreducible 
$\cG$-representations in the principal series. Considering Theorem \ref{thm:main}
for all Bernstein components in the principal series simultaneously, we will 
parametrize $\Irr (\cG, \cT)$.

\begin{prop}\label{prop:LLC}
Let $\cG$ be a split reductive $p$-adic group, with restrictions on the residual
characteristic as in Condition \ref{con:char}.
There exists a commutative, bijective triangle 
\[ 
\xymatrix{   & (\Irr \, \mathcal T /\!/ \mathcal W^G )_2 \ar[dr]\ar[dl] & \\  
\Irr ( \cG, \mathcal T)    \ar[rr] & & \{\text{$\KLR$ parameters for } G \} / G} 
\]
The right slanted map is natural, and via the bottom map any $\pi \in 
\Irr (\cG ,\cT)$ canonically determines a Langlands parameter $\Phi$ for $\cG$.

The restriction of this diagram to a single Bernstein component recovers Theorem 
\ref{thm:main}. In particular the bottom arrow generalizes the Kazhdan--Lusztig
parametrization of the irreducible $\cG$-representations in the unramified
principal series.
\end{prop}
\begin{proof}
Let us work out what happens if in Theorem \ref{thm:main} we take the union over 
all Bernstein components $\fs \in \mathfrak B (\cG,\cT)$.

On the left we obtain (by definition) the space $\Irr (\cG ,\mathcal T)$. 
Notice that in Theorem \ref{thm:main}, instead of \{KLR parameters$\}^\fs / H$ we 
could just as well take $G$-conjugacy classes of KLR parameters $(\Phi,\rho)$ such 
that $\Phi \big|_{\mathbf I_F}$ is $G$-conjugate to $c^\fs$. The union of those clearly is 
the space of all $G$-conjugacy classes of KLR parameters for $G$. For the space at
the top of the diagram, choose a smooth character $\chi_\fs$ of $\mathcal T$ such that 
$(\mathcal T,\chi_\fs) \in \fs$. By definition the $T^\fs$ in $(T^\fs /\!/ W^\fs )_2$ equals 
\[
T^\fs := \{ \chi_\fs \otimes t \mid t \in T \},
\]
where $t$ is considered as an unramified character of $\mathcal T$. On the other hand,
$\Irr \, \mathcal T$ can be obtained by picking representatives $\chi_\fs$ for
$\Irr (\cT_0) = ( \Irr \, \mathcal T )/ T$ and taking the union of the corresponding $T^\fs$.
Two such spaces $T^\fs$ give rise to the same Bernstein component for $\cG$ 
if and only if they are conjugate by an element of $N_{\cG} (\mathcal T)$, 
or equivalently by an element of $\mathcal W^G$. Therefore
\[
(\Irr \, \mathcal T /\!/ \mathcal W^G )_2 = \big( \bigcup_{\fs \in \mathfrak B (\cG,\cT)} 
\mathcal W^G \cdot T^\fs /\!/ \mathcal W^G \big)_2 
= \bigcup_{\fs \in \mathfrak B (\cG,\cT)} \big( T^\fs /\!/ W^\fs \big)_2 .
\]
Hence the union of the spaces in the commutative triangles from Theorem \ref{thm:main} 
is as desired. The right slanted arrows in these triangles combine to a natural bijection
\[
(\Irr \, \mathcal T /\!/ \cW^G )_2 \to \{\text{KLR parameters for } G \} / G ,
\] 
because the $\cW^G$-action is compatible with the $G$-action. Suppose that $(\cT ,\chi'_\fs)$
is another base point for $\fs$. Up to an unramified twist, we may assume that $\chi'_\fs =
w \chi_\fs$ for some $w \in \cW^G$. Then the Hecke algebras $\cH (H)$, and $\cH (H')$ are
isomorphic by a map that reflects conjugation by $w$ and by Theorem \ref{thm:Roche}.(2)
this is compatible with the bijections between $\Irr (\cG)^\fs , \Irr (\cH (H))$ and
$\Irr (\cH (H'))$. It follows that the bottom maps in the triangles from Theorem 
\ref{thm:main} paste to a bijection 
\[
\Irr (\cG ,\cT) \to \{\text{KLR parameters for } G \} / G .
\]
Finally, the map
\[
(\Irr \, \mathcal T /\!/ \cW^G )_2 \to \Irr (\cG ,\cT)
\]
can be defined as the composition of the other two bijections in the above triangle. 
Then it is the combination the left slanted maps from Theorem \ref{thm:main} 
because the triangles over there are commutative. 
\end{proof}

The bottom rows of Theorem \ref{thm:main} and Proposition \ref{prop:LLC} can be considered
as a Langlands correspondence for $\Irr (\cG,\cT)$. In other words, for a Langlands 
parameter $\Phi$ as in \eqref{eqn:Phi} we define the principal series part of the
L-packet $\Pi_\Phi (\cG)$ as
\begin{equation}\label{eq:N.24}
\{ \pi (\Phi,\rho) \mid  
\rho \in \Irr \big( \pi_0 (\Cent_G (\Phi)) \big) \text{ geometric } \} .
\end{equation}
It is expected that $\Pi_\Phi (\cG)$ contains one $\cG$-representation for every
irreducible representation of $\pi_0 (\Cent_G (\Phi))$. Therefore we believe that
\eqref{eq:N.24} exhausts $\Pi_\Phi (\cG)$ if and only if every
irreducible representation of $\pi_0 (\Cent_G (\Phi))$ appears in 
$H_* \big( \mathcal B_G^{\Phi (\mathbf W_F \times B_2)} ,\C \big)$.

To support our partial Langlands correspondence, we will show that it satisfies 
Borel's ``desiderata'' \cite[Section 10]{BorAut}. Let us recall them here:

\begin{con}[Borel's desiderata]\label{cond:Borel} \
\begin{enumerate}
\item Let $\chi_\Phi$ be the character of $Z(\cG)$ canonically associated to $\Phi$.
Then any $\pi \in \Pi_\Phi (\cG)$ has central character $\chi_\Phi$.
\item Let $c$ be a one-cocycle of $\mathbf W_F$ with values in $Z(G)$ and let 
$\chi_c$ be the associated character of $\cG$. Then $\Pi_{c \Phi}(\cG) =
\{ \pi \otimes \chi_c \mid \pi \in \Pi_\Phi (\cG) \}$.
\item If one element of $\Pi_\Phi (\cG)$ is essentially square-integrable,
then all elements are. This happens if and only if the image of $\Phi$ is not
contained in any proper Levi subgroup of the Langlands dual group ${}^L \cG$.
\item If one element of $\Pi_\Phi (\cG)$ is tempered, then all elements are.
This is equivalent to $\Phi (\mathbf W_F)$ being bounded in $G$.
\item Suppose that $\eta : \tilde \cG \to \cG$ is a morphism of connected reductive
$F$-groups with commutative kernel and cokernel. Let ${}^L \eta : {}^L \cG \to
{}^L \tilde \cG$ be the dual morphism and let $\pi \in \Pi_\Phi (\cG)$. 
Then $\pi \circ \eta$ is a direct sum of some 
$\tilde \pi \in \Pi_{{}^L \eta \circ \Phi}(\tilde \cG)$. 
\end{enumerate}
\end{con}

Part (4) of Condition \ref{cond:Borel} has already been proved in
part (2) of Theorem \ref{thm:main}. Here we check the desiderata 1, 2 and 3.

\begin{lem}\label{lem:desiderata1}
Borel's desideratum (1) holds for our Langlands correspondence for $\Irr (\cG,\cT)$.
\end{lem}
\begin{proof}
The infinitesimal central character of $\pi (\Phi,\rho)$, as described
in Theorem \ref{thm:main}.(1), is its cuspidal support. For representations
in the principal series this boils down to a character of $\cT$, uniquely
determined up to $\cW^G$. Since $\Cent (\cG) \subset \cT$, the central
character of $\pi (\Phi,\rho)$ is just the restriction of
\begin{equation}\label{eq:N.23}
\Phi \big( \varpi_F , \matje{q^{1/2}}{0}{0}{q^{-1/2}} \big) \in T / W^\fs 
\cong T^\fs / W^\fs
\end{equation}
to $\Cent (\cG)$. Recall that $\Phi \big( \matje{q^{1/2}}{0}{0}{q^{-1/2}} \big)$
comes from a homomorphism $\SL_2 (\C) \to G$. The image of $\Phi |_{\SL_2 (\C)}$
is generated by unipotent elements, so it is contained in the derived group
$G_{\der}$. That is the complex dual group of $\cG / \Cent (\cG)$, so
$\Phi \big( \matje{q^{1/2}}{0}{0}{q^{-1/2}} \big)$ does not effect the
$\Cent (\cG)$-character of $\pi (\Phi,\rho)$ and we may consider $\Phi (\varpi_F)$
instead of \eqref{eq:N.23}. In view of our identification $T \cong T^\fs$
from \eqref{eq:N.22}, this means that the central character of $\pi (\Phi,\rho)$
is just the restriction to $\Cent (\cG)$ of the $\cT$-character determined by
$\Phi |_{\mathbf W_F}$ via the local Langlands correspondence for (split) tori.
This agrees with Borel's construction given in \cite[\S 10.1]{BorAut}. 
\end{proof}

\begin{lem}\label{lem:desideratum2}
Desideratum (2) holds for $\Irr (\cG,\cT)$. More precisely, if $c$ and $\chi_c$ 
are as in (2) and $(\Phi,\rho)$ is a KLR parameter for $\cG$, then 
$\pi (c \Phi,\rho) \cong \pi (\Phi,\rho) \otimes \chi_c$.
\end{lem}
\begin{proof}
Since $c$ takes values in $\Cent (G) \subset T$, we can multiply any Langlands
parameter for $\cT$ with $c$ and obtain another Langlands parameter for $\cT$.
If we transfer this map to $\Irr (\cT)$ via the local Langlands correspondence
we get $\chi \mapsto \chi \otimes \chi_c$, see \cite[\S 10.2]{BorAut}.
The composition 
\begin{equation}\label{eq:N.25}
T \to T^\fs \xrightarrow{\; \otimes \chi_c \;} T^{\fs \chi_c} \to T , 
\end{equation}
where both outer maps come from \eqref{eq:N.22}, is of the form $t \mapsto 
c_T t$ for a unique $c_T \in \Cent (G) \subset T$.

Because the image of $c$ is contained in $\Cent (G)$, the Langlands parameter
$c \Phi$ has the same centralizer in $G$ as $\Phi$. Hence $(c \Phi,\rho)$ is
a well-defined KLR parameter and the groups $H$ and $H_c$, associated 
respectively to $\Phi |_{\mathbf W_F}$ and to $c \Phi |_{\mathbf W_F}$,
coincide. Then \eqref{eq:N.25} gives rise to an isomorphism
\begin{equation}\label{eq:N.50}
\phi_c : \cH (H_c) \to \cH (H) \quad \text{with} \quad 
\phi_c (T_w \theta_\lambda) = \lambda (c_T) T_w \theta_\lambda
\end{equation}
for $w \in \cW^H$ and $\lambda \in X^* (T)$. The induced map on irreducible
representations is
\begin{equation}\label{eq:N.26}
\begin{aligned}
\phi_c^* : \Irr (\cH (H)) & \to \Irr (\cH (H_c)) , \\
\pi (t_q,x,\rho_q) & 
\mapsto \pi (t_q,x,\rho_q) \otimes c_T = \pi (t_q c_T,x,\rho_q) .
\end{aligned}
\end{equation}
Since $\chi_c$ is a character of $\Cent (\cG)$, the $\fs \chi_c$-type
$(J_c, \tau_c)$ from \cite{Roc} equals to $(J,\tau \otimes \chi_c)$.
Therefore the composition of $\phi_c$ with the two instances of \eqref{eq:N.6} is 
\[
\cH (\cG,\tau \otimes \chi_c) \to \cH (\cG,\tau) : f \mapsto \chi_c f .
\]
(Here $\chi_c f$ denotes pointwise multiplication of functions $\cG \to \C$, 
not a convolution product.) It follows that the composition of $\phi_c^*$ with 
the two instances of Theorem \ref{thm:Roche} is just
\[
\Irr (\cG)^\fs \xrightarrow{\; \otimes \chi_c \;} \Irr (\cG)^{\fs \chi_c} .
\]
This and \eqref{eq:N.26} show that $\pi (c \Phi,\rho) \cong \pi (\Phi,\rho)
\otimes \chi_c$.
\end{proof}

Recall that $\cG$ only has irreducible square-integrable representations if $\Z (\cG)$ 
is compact. A $\cG$-representation is called essentially square-integrable if its 
restriction to the derived group of $\cG$ is square-integrable. This is more general 
than square-integrable modulo centre, because for that notion $\Z(\cG)$ needs to act by 
a unitary character. Also recall that (essential) square-integrability for 
$\cH (H)$-modules was defined just before Proposition \ref{prop:tempered}.

To show that Roche's equivalence of categories Theorem \ref{thm:Roche} preserves
essential square-integrability, we will first characterize in another way.
Let $\Irr (\cG)_\temp$ be the set of tempered representations in $\Irr (\cG)$ or,
equivalently the support of the Plancherel measure in $\Irr (\cG)$ or the dual
of the reduced $C^*$-algebra of $\cG$. The space of tempered irreducible
$\cH (H)$-modules $\Irr (\cH (H))_\temp$ (denoted $\hat{\mathfrak C}$ in \cite{Opd}) 
admits the analogous descriptions.

\begin{lem}\label{lem:essL2}
\begin{enumerate}
\item A tempered irreducible $\cG$-representation is essentially square-integrable if and
only if it is contained in a connected component of $\Irr (\cG)_\temp$ of minimal dimension, 
namely $\dim (\Cent (G))$.
\item A tempered irreducible $\cH (H)$-module is essentially square-integrable if and
only if it is contained in a connected component of \\
$\Irr (\cH (H))_\temp$ of minimal dimension, namely $\dim (H / H_\der)$. 
\end{enumerate}
\end{lem}
\begin{proof}
(1) This follows from Harish-Chandra's the description of $\Irr (\cG)_\temp$ in terms 
of essentially square-integrable representations of Levi subgroups of $\cG$ 
\cite[Proposition III.4.1]{Wal}. Let us make it concrete. 

If $\pi \in \Irr (\cG)_\temp$ 
is essentially square-integrable then its central character is unitary, so it is
``carr\'e int\'egrable'' in the sense of \cite{Wal}. Its connected component in
$\Irr (\cG)_\temp$ consists of the twists of $\pi$ by unitary unramified characters
of $\cG$, so it has dimension $\dim (\Cent (G)) = \dim (\Cent (\cG))$. 

Suppose now that $\pi \in \Irr (\cG)_\temp$ is not essentially square-integrable.
By \cite[Proposition III.4.1]{Wal} there exist a proper parabolic subgroup 
$\cP \subset \cG$ with Levi factor $\cM$ and an essentially square-integrable 
representation $\omega \in \Irr (\cM)_\temp$ such
that $\pi$ is a quotient of $I_\cP^\cG (\omega)$. Moreover $(M,\omega)$ is unique
up to $\cG$-conjugation. The connected component of $\pi$ in $\Irr (\cG)_\temp$
consists of all the subquotients of $I_\cP^\cG (\omega \otimes \chi_\cM)$, where
$\chi_\cM$ runs through the unitary unramified characters of $\cM$. The space of
such characters has dimension $\dim (\Cent (\cM))$, which is larger then 
$\dim (\Cent (\cG))$ because $\cP \subsetneq \cG$. \\
(2) Recall that $\cH (H) = \cH (H^\circ) \rtimes \pi_0 (H)$ and that essential 
square-integrability of $\cH(H)$-modules depends only on their restriction to 
$\cH (H^\circ)$. For the affine Hecke algebra $\cH (H^\circ)$ the claim can be 
proven in the same way as part (1), using the Plancherel theorem and the ensuing 
description of $\Irr_\temp (\cH (H^\circ))$
\cite[Corollary 5.7]{DeOp}. From there it can be generalized to $\cH (H)$ with 
the comparison between $\Irr (\cH (H))$ and $\Irr (\cH (H^\circ))$ provided
by Clifford theory for crossed products with finite groups \cite[Appendix]{RamRam}.
\end{proof}

\begin{prop}\label{prop:desideratum3}
Let $(\Phi,\rho)$ be a KLR-parameter for $G$. Then $\pi (\Phi,\rho) \in \Irr (\cG)$ is
essentially square-integrable if and only of the image of $\Phi$ is not contained in
any Levi subgroup of a proper parabolic subgroup of $G$.

In other words, Borel's desideratum (3) holds for our Langlands correspondence 
for $\Irr (\cG,\cT)$.
\end{prop}
\begin{proof}
Let $(\Phi,\rho)$ be a KLR-parameter for $G$ and let $\chi_\Phi$ is the central character 
of $\pi (\Phi,\rho) \in \Irr (\cG)^\fs$. There is a unique unramified character 
$\chi_c : \cG \to \R_{>0}$ such that 
\[
\chi_c (z) = |\chi_\Phi (z)|^{-1} \text{ for all }z \in \Cent (G).
\]
Let $c : \bW \to \Cent (G)$ be the associated homomorphism. By Lemma \ref{lem:desideratum2} 
\[
\pi (c \Phi,\rho) = \chi_c \otimes \pi (\Phi,\rho) \in \Irr (\cG)^\fs ,
\]
and by construction this representation has a unitary central character.
Since $\chi_c$ is trivial on $\cG_\der , \pi (c \Phi,\rho)$ is essentially square-integrable
if and only if $\pi (\Phi,\rho)$ is so. As $c (\bW_F) \subset \Cent (G)$, Borel's conditions 
on the Levi subgroups for $\Phi$ and $c \Phi$ are equivalent. Therefore it suffices to prove the 
proposition in case $\pi (\Phi,\rho)$ has unitary central character. We assume this from now on.

Suppose that $\pi (\Phi,\rho)$ is essentially square-integrable. For every matrix 
coefficient $f$ of $\pi (\Phi,\rho), |f|$ is a square-integrable function on 
$\cG / \Cent (\cG)$. In particular $|f|$ is tempered, so $\pi (\Phi,\rho)$ is also tempered. 
By Lemma \ref{lem:essL2} $\pi (\Phi,\rho)$ 
is contained in a component of $\Irr (\cG)_\temp$ of minimal dimension $\dim (\Cent (G))$.

According to \cite[Theorem B]{BHK} the equivalence categories $\Rep (\cG)^\fs \to
\Mod (\cH (H))$ from Theorem \ref{thm:Roche} induces a homeomorphism 
$\Irr (\cG)^\fs_\temp \to \Irr (\cH (H))_\temp$. 
Therefore $\pi (t_q,x,\rho_q)$ lies in a $\dim (\Cent (G))$-dimensional component of 
$\Irr (\cH (H))_\red$, and this is a component of minimal dimension. Now 
Lemma \ref{lem:essL2}.2 says that $\pi (t_q,x,\rho_q) \in \Irr (\cH (G))$ is essentially 
square-integrable and that $\dim (\Cent (G)) = \dim (H / H_\der)$. Since 
$H = \Cent_G (\Phi (\bI_F))$, this implies that $\dim H = \dim G$, which by the connectedness
of $G$ means that $H = G$.

Now Proposition \ref{prop:tempered}.2 tells us that $\{t_q,x\}$ is not contained in any
Levi subgroup of a proper parabolic subgroup of $G$. Obviously the same goes for the
image of $\Phi$.

For the opposite direction, suppose that the image of $\Phi$ is not contained in any proper 
Levi subgroup of $G$. The Levi subgroup $H^\circ \subset G$ contains the image of $\Phi$, 
so $H^\circ = H = G$. If $\{t_q,x\}$ were contained in a proper Levi subgroup $L \subset G$, 
we could extend them to a Langlands parameter $\Phi'$ with image in $L$. By Lemma 
\ref{lem:compareParameters} $\Phi'$ would be conjugate to $\Phi$, which would imply that the
image of $\Phi$ is contained in a conjugate of $L$. That would violate our assumption, so
$\{t_q,x\}$ cannot be contained in any proper Levi subgroup of $G$.

Proposition \ref{prop:tempered}.2 says that $\pi (t_q,x,\rho_q) \in \Irr (\cH (G))$ is essentially 
square-integrable. Recall that we assumed that $\pi (\Phi,\rho) \in \Irr (\cG)$ has unitary
central character. With the description of its infinitesimal central character given in
Theorem \ref{thm:main} we see that 
\[
|\lambda (t_q)| = 1 \text{ for all } \lambda \in X^* (G / G_\der) = X_* (\Cent (G)). 
\]
Hence the restriction of $\pi (t_q,x,\rho_q)$ to $\cH (G / G_\der)$ is
tempered. The restriction of $\pi (t_q,x,\rho_q)$ to $\cH (G / \Cent (G))$ is square-integrable,
so in particular tempered. As noted in the proof of Proposition \ref{prop:tempered}.1, these
two facts imply that $\pi (t_q,x,\rho_q)$ is also tempered as $\cH (G)$-module. Now the above 
arguments using Lemma \ref{lem:essL2} can be applied in the reverse order, and they show that 
$\pi (\Phi,\rho)$ is essentially square-integrable.
\end{proof}

\section{Functoriality}
\label{sec:func}

The fifth of Borel's desiderata in Condition \ref{cond:Borel} says that 
the LLC should be functorial with respect to some specific morphisms of reductive
groups. To show that this holds in our setting, we need substantial
technical preparations. The first results of this section are valid without
any restriction on the residual characteristic.

Let $\eta : \tilde \cG \to \cG$ be a morphism of connected reductive split
$F$-groups, with commutative kernel and cokernel. Let $\check \eta : G \to \tilde G$
be the dual homomorphism, as in \cite[\S 1.2]{BorAut}.

\begin{lem}\label{lem:etaProperties}
Define $\tilde \cT = \eta^{-1}(\cT)$.
\begin{enumerate}
\item $\tilde \cT$ is a split maximal torus of $\tilde \cG$ and 
$\ker (\eta : \tilde \cT \to \cT) = \ker (\eta : \tilde \cG \to \cG)$.
\item $\ker \eta \subset \Cent (\tilde \cG)$ and 
$\eta^{-1}(\Cent (\cG)) = \Cent (\tilde \cG)$.
\item The map $X^* (\cT) \to X^* (\tilde \cT) : \alpha \mapsto \alpha \circ \eta$ 
induces a bijection $R(\cG ,\cT) \to R (\tilde \cG, \tilde \cT)$. Similarly
$X^* (\tilde T) \to X^* (T) : \beta \mapsto \beta \circ \check \eta$ induces
a bijection $R(\tilde G, \tilde T) \to R (G,T)$.
\item $\mathrm{coker}(\eta : \tilde \cT \to \cT) \cong 
\mathrm{coker}(\eta : \tilde \cG \to \cG)$.
\end{enumerate}
\end{lem}
\begin{proof}
(1) Decompose the Lie algebras of $\tilde \cG$ and $\cG$ as
\begin{equation}\label{eq:N.28}
\begin{aligned}
& \Lie (\tilde \cG) = \Lie (\tilde{\cG}_{\der}) \oplus \Lie (\Cent (\tilde \cG)) , \\
& \Lie (\cG) = \Lie (\cG_{\der}) \oplus \Lie (\Cent (\cG)) .
\end{aligned}
\end{equation}
The assumptions on the kernel and cokernel of $\eta$ imply that $\Lie (\eta)$ restricts
to an isomorphism 
\begin{equation}\label{eq:N.29}
\Lie (\eta) : \Lie (\tilde{\cG}_{\der}) \to \Lie (\cG_{\der})
\end{equation}
and to a $F$-linear map
$\Lie (\Cent (\tilde \cG)) \to \Lie (\Cent (\cG))$. Hence $\Lie (\tilde \cT)$ is the
Lie algebra of $\Cent (\tilde \cG)$ times a maximal split torus of $\tilde{\cG}_{\der}$.
It follows that the unit component $\tilde{\cT}^\circ$ of $\tilde \cT$ is a split 
maximal torus of $\tilde \cG$. Now $\ker (\eta) \Cent (\tilde \cG)^\circ /  
\Cent (\tilde \cG)^\circ$ is a finite normal subgroup of $\tilde \cG / \Cent (\tilde \cG)^\circ$,
so it is central and contained in the maximal torus $\tilde{\cT}^\circ / 
\Cent (\tilde \cG)^\circ$ of $\tilde \cG / \Cent (\tilde \cG)^\circ$. Consequently
$\ker \eta$ is contained in $\tilde{\cT}^\circ$. As $\cT$ is connected and
$\tilde \cT = \eta^{-1}(\cT) ,\; \tilde \cT$ is also connected.\\
(2) We just saw that $\ker \eta$ is contained in the maximal split torus
$\tilde \cT$ of $\tilde \cG$. But $\cT$ was arbitrary, so $\ker \eta$ lies in every
split maximal torus of $\tilde \cG$. The intersection of all such tori is contained in
the centre of $\tilde \cG$, so $\ker \eta$ as well.

Since $\cG$ is connected, $\Cent (\cG)$ is the kernel of the adjoint representation
of $\cG$. As $\Lie (\Cent (\cG))$ is a trivial summand of the adjoint representation,
$\Cent (\cG)$ is also the kernel of $\mathrm{Ad} : \cG \to \End_\C (\Lie (\cG_{\der}))$.
In view of \eqref{eq:N.29} there is a commutative diagram
\[
\begin{array}{ccc}
\tilde \cG & \xrightarrow{\:\; \eta \;\:} & \cG \\
\downarrow \tilde{\mathrm{Ad}} & & \downarrow \mathrm{Ad} \\
\End_\C (\Lie (\tilde{\cG}_{\der})) & \longrightarrow & \End_\C (\Lie (\cG_{\der})) .
\end{array}
\]
Now we see that
\[
\eta^{-1}(\Cent (\cG)) = \eta^{-1} (\ker \mathrm{Ad}) = 
\ker \tilde{\mathrm{Ad}} = \Cent (\tilde \cG) .
\]
(3) From part (1) and the isomorphism \eqref{eq:N.29}
we deduce that $\eta$ maps any root subgroup of $(\tilde \cG, \tilde \cT)$ bijectively
to a root subgroup of $(\cG,\cT)$. This yields the bijection
$R(\cG,\cT) \to R (\tilde \cG, \tilde \cT)$, which is given explicitly by
$\alpha \mapsto \alpha \circ \eta$. The second claim follows by dualizing the root data.\\
(4) In view of \eqref{eq:N.29} $\eta$ restricts to an isomorphism of root subgroups
$\tilde{\cG}_{\alpha \circ \eta} \to \cG_\alpha$, for every $\alpha \in R(\cG,\cT)$.
With part (1) we see that the preimage $\tilde \cB = \eta^{-1}(\cB)$ of our Borel
subgroup $\cB \subset \cG$ is a Borel subgroup of $\tilde \cG$. Let $\cU$ (resp. $\tilde \cU$)
be the unipotent radical of $\cB$ (resp. $\tilde \cB$). The Bruhat decomposition says that
\[
\cG = \cU N_{\cG}(\cT) \cU \text{ and }
\tilde \cG = \tilde \cU N_{\tilde \cG}(\tilde \cT) \tilde \cU .
\]
As $\eta : \tilde \cU \to \cU$ is an isomorphism and 
\[
N_{\cG}(\cT) / \cT \cong \cW^\cG \cong \cW^{\tilde \cG} \cong 
N_{\tilde \cG}(\tilde \cT) / \tilde \cT ,
\]
the inclusions $\cT \to \cG$ and $\tilde \cT \to \tilde \cG$ give an isomorphism 
$\mathrm{coker}(\eta : \tilde \cT \to \cT) \to \mathrm{coker}(\eta : \tilde \cG \to \cG)$.
\end{proof}

Recall that $\chi \in \Irr (\cT)$ and $\fs = [\cT,\chi]_\cG$. Let $c_\fs : 
\fo_F^\times \to T$ be the restriction of the Langlands parameter of $\chi$.

\begin{lem}\label{lem:etaBernstein}
Define $\tilde \fs = [\tilde \cT,\chi \circ \eta ]_{\tilde \cG}$ and
\[
c_{\tilde \fs} = \check \eta \circ c_\fs : \fo_F^\times \to \tilde T .
\]
Then $c_{\tilde \fs}$ is the restriction of the Langlands parameter of $\chi \circ \eta$ 
to $\fo_F^\times$ and $\eta^*$ sends $\Rep (\cG)^\fs$ to $\Rep (\tilde \cG)^{\tilde \fs}$.
\end{lem}
\begin{proof}
It follows from the construction of the local Langlands correspondence for split tori
that the diagram
\[
\begin{array}{ccc}
\Irr (\cT) & \xrightarrow{\; \eta^* \;} & \Irr (\tilde \cT) \\
\downarrow & & \downarrow \\
\{ \text{L-parameters for } \cT \} & \xrightarrow{\; \check \eta \;} &
\{ \text{L-parameters for } \tilde \cT \}
\end{array}
\]
commutes. With Lemma \ref{lem:cBernstein} this proves the first claim.

By definition $\Rep (\cG)^\fs$ is the category of all smooth $\cG$-representations $\pi$
with the property that every irreducible subquotient of $\pi$ occurs as a subquotient
of $I_\cB^\cG (\chi \otimes t)$ for some unramified character $t \in X_\unr (\cT)$. So
for the second claim it suffices to show that $\eta^* \big( I_\cB^\cG (\chi \otimes t) \big)
\in \Rep (\tilde \cG)^{\tilde \fs}$ for all $t \in X_\unr (\cT)$. 

Clearly $\eta (\tilde{\cT_0}) \subset \cT_0$, so $\eta^* (t) \in \Irr (\tilde{\cT})$ is
unramified. Hence $\eta^* (\chi \otimes t)$ is an unramified twist of $\eta^* (\chi)$ and
$\eta^* (\chi \otimes t) \in [\tilde{\cT}, \chi \circ \eta ]_{\tilde \cT}$. By Lemma
\ref{lem:etaProperties}.(4) $\eta$ induces a homeomorphism $\tilde \cG / \tilde \cB \to
\cG / \cB$, which implies that 
\[
\eta^* \big( \mathrm{Ind}_\cB^\cG (\chi \otimes t) \big) 
= \mathrm{Ind}_{\tilde \cB}^{\tilde \cG} (\eta^* (\chi \otimes t)). 
\]
It follows from Lemma \ref{lem:etaProperties}.(3) that the difference between parabolic 
induction and normalized parabolic induction consists of twisting by essentially the same 
unramified character on both sides, so we get
\[
\eta^* \big( I_\cB^\cG (\chi \otimes t) \big) = I_{\tilde \cB}^{\tilde \cG} 
(\eta^* (\chi \otimes t)) \in \Rep (\tilde \cG)^{\tilde \fs} . \qedhere
\]
\end{proof}

As before, we define $H = \Cent_G (c_\fs)$ and $\tilde H = \Cent_{\tilde G}(c_{\tilde \fs})$. 
By Lemma \ref{lem:etaProperties}.(3) there is a bijection
\begin{multline}
R(H,T) = \{ \alpha \in R(G,T) \mid \alpha (c_\fs (\fo_F^\times)) = 1 \} \longleftrightarrow \\
\{ \tilde \alpha \in R(\tilde G,\tilde T) \mid \tilde \alpha (\eta \circ c_\fs (\fo_F^\times)) 
= 1 \} = R (\tilde H,\tilde T) . 
\end{multline}
Hence $\cW^{H^\circ} \cong \cW^{\tilde{H}^\circ}$. As $\check \eta (H) \subset \tilde H$,
we have a canonical inclusion
\begin{equation}\label{eq:N.33}
\cW^\eta : \cW^H \to \cW^{\tilde H} .
\end{equation}
However, in general it is not a bijection.
\begin{ex}
Consider the canonical homomorphism $\eta : \tilde \cG = \SL_3 (F) \to \PGL_3 (F) = \cG$. 
Let $\zeta$ be a character of order 3 of $\fo_F^\times$ and put $c_\fs = (\zeta,\zeta^2,1)$. 
Since $c_\fs$ is trivial on $Z(\GL_3 (F))$, it defines a Bernstein component for the standard 
maximal torus $\tilde \cT$ of $\PGL_3 (F)$. In this case $\cW^H = \{1\}$. Similarly 
$\check \eta \circ c_\fs$ defines a Bernstein component for the
standard maximal torus $\cT$ of $\SL_3 (F)$. For $(a,b,c) \in \cT (\fo_F)$:
\[
c_\fs (a,b,c) = \zeta (a) \zeta (b^2) = \zeta (b c^{-1}) = c_\fs (b,c,a) ,
\]
from which it follows easily that $\cW^{\tilde H} \cong \mathbb Z / 3 \mathbb Z$.
\end{ex}

Let $\tilde{H}' \subset \tilde H$ be the subgroup generated by $\tilde{H}^\circ$ and
$W^\fs = \cW^H$, via \eqref{eq:N.33}. Then $\cH (\tilde H)$ contains a subalgebra 
\[
\cH (\tilde{H}') \cong \cH (\tilde{H}^\circ) \rtimes \pi_0 (H) . 
\]
The advantage of this algebra over $\cH (\tilde H)$ is that $\eta$ induces
an algebra homomorphism
\begin{equation}\label{eq:N.39}
\begin{aligned}
& \phi_\eta : \cH (\tilde{H}') \to \cH (H) , \\
& \phi_\eta (T_w \theta_{\mu}) = T_w \theta_{\eta (\mu)}
\text{ for } w \in W^\fs , \mu \in X_* (\tilde \cT) = X^* (\tilde T).
\end{aligned}
\end{equation}

Recall that $\Phi$ is a Langlands parameter for $\cG$ as in \eqref{eqn:Phi}.
As remarked in Section \ref{sec:comppar}, all the geometric representations $\rho$
of $\pi_0 \big( \Cent_G (\Phi)$, which can be used to enhance $\Phi$ to a KLR parameter,
factor through $\pi_0 \big( \Cent_G (\Phi) / \Cent (G) \big)$.

\begin{lem}\label{lem:pi0eta}
$\check \eta$ induces injective group homomorphisms
\begin{align*}
& \pi_0 \big( \Cent_G (\Phi) / \Cent (G) \big) \to 
\pi_0 \big( \Cent_{\tilde G}(\check \eta \circ \Phi) / \Cent (\tilde G) \big) , \\
& \pi_0 \big( \Cent_G (\Phi) / \Cent (H) \big) \to 
\pi_0 \big( \Cent_{\tilde G}(\check \eta \circ \Phi) / \Cent (\tilde{H}') \big) .
\end{align*}
\end{lem}
\begin{proof}
Clearly $\check \eta$ restricts to a group homomorphism
\[
\Cent_H (\Phi) = \Cent_G (\Phi) \to \Cent_{\tilde G}(\check \eta \circ \Phi) = 
\Cent_{\tilde H}(\eta \circ \Phi) .
\]
By Lemma \ref{lem:etaProperties}.(4) it induces an injective homomorphism
\begin{equation}\label{eq:N.30}
\check \eta : \Cent_H (\Phi) / \Cent (G) \to  
\Cent_{\tilde H}(\check \eta \circ \Phi) / \Cent (\tilde G) .
\end{equation}
By \eqref{eq:N.33} Lie$(\eta)$ maps Lie$(Z(H))$ to Lie$(Z(\tilde{H}'))$, so by Lemma
\ref{lem:etaProperties}.(4) also
\begin{equation}\label{eq:N.42}
\check \eta : \Cent_H (\Phi) / \Cent (H) \to  
\Cent_{\tilde H}(\check \eta \circ \Phi) / \Cent (\tilde{H}')
\end{equation}
is injective. We will show that \eqref{eq:N.30} and \eqref{eq:N.42} are isogenies.

The properties of $\eta$ imply that
\begin{multline}\label{eq:N.31}
\Lie (\check \eta) : \Lie (\Cent_H (\Phi) / \Cent (G)) =
\Lie (H) / \Cent (\Lie (G)) \to \\ \Lie (\Cent_{\tilde H}(\check \eta \circ \Phi) / 
Z(\tilde G)) = \Lie (\Cent_{\tilde H}(\check \eta \circ \Phi)) / \Cent \Lie (\tilde G)  
\end{multline}
is an isomorphism of reductive Lie algebras. The group $\Phi (\SL_2 (\C))$ is contained
in $H_{\der}$, so by \eqref{eq:N.31} $\Lie (\check \eta)$ maps $\Lie \big( \Phi (\SL_2 (\C)) 
\big)$ bijectively to $\Lie \big( \check \eta \circ \Phi (\SL_2 (\C)) \big)$. Therefore
\begin{equation}\label{eq:N.32}
\Lie \big( \Cent_H \big( \Phi (\SL_2 (\C)) \big) / \Cent (G) \big) \to
\Lie \big( \Cent_{\tilde H} \big( \check \eta \circ \Phi (\SL_2 (\C)) \big) / 
\Cent (\tilde G) \big) 
\end{equation}
is another isomorphism of reductive Lie algebras. To reach the Lie algebras of
\eqref{eq:N.30}, it remains to restrict to elements that commute with $\Phi (\varpi_F)$,
respectively $\check \eta \circ \Phi (\varpi_F)$. For simplicity we assume that $T$ is
a maximal torus of $\Cent_H (\Phi)$, this can always be achieved by replacing $\Phi$
with a conjugate Langlands parameter. Then we have a bijection
\begin{multline*}
R \big( \Cent_H (\Phi),T \big ) = \big\{ \alpha \in R \big(\Cent_H \big(\Phi (\SL_2 (\C))
\big),T \big) \mid \alpha (\Phi (\varpi_F)) = 1 \big \} \longleftrightarrow \\
\big\{ \tilde \alpha \in R \big(\Cent_{\tilde H} \big(\check \eta \circ \Phi (\SL_2 (\C))
\big),\tilde T \big) \mid \tilde \alpha (\eta \circ \Phi (\varpi_F)) = 1 \big \} =
R \big( \Cent_{\tilde H} (\check \eta \circ \Phi ),\tilde T \big) .
\end{multline*}
With \eqref{eq:N.32} this this implies that 
\[
\Lie (\check \eta) : \Lie ( \Cent_H (\Phi) / \Cent (G) ) \to  
\Lie \big( \Cent_{\tilde H}(\check \eta \circ \Phi) / \Cent (\tilde G) \big) .
\]
is an isomorphism, so \eqref{eq:N.30} is indeed an isogeny. The same argument shows 
that \eqref{eq:N.42} is an isogeny.

As \eqref{eq:N.30} and \eqref{eq:N.42} are also injective, they embed the left hand side 
in the right hand side as a number of connected components. Hence the induced maps on 
the component groups are injective.
\end{proof}

Now we reinstate Condition \ref{con:char}.
Let $(\tilde J,\tilde \tau)$ be Roche's type for $\tilde \fs$. The explicit construction in 
\cite[\S 3]{Roc} shows that $\tilde J = \eta^{-1}(J)$ and $\tilde \tau = \rho \circ \eta$.
By \cite[Theorem 4.15]{Roc} the support of $\cH (\tilde \cG,\tilde \tau)$ is
$\tilde J (W^{\tilde \fs} \rtimes X_* (\tilde \cT) ) \tilde J$, where we embed 
$X_* (\tilde \cT)$ in $\cT$ via $\mu \mapsto \mu (\varpi_F)$. By Theorem \ref{thm:Roche}
\begin{equation}\label{eq:N.41}
\cH (\tilde \cG , \tilde \tau) \cong \cH (\tilde H) = 
\cH (\tilde{H}^\circ) \rtimes \pi_0 (\tilde H) .
\end{equation}
Let $\cH (\tilde \cG, \tilde \tau)'$ be the subalgebra of 
$\cH (\tilde \cG, \tilde \tau)$ isomorphic to $\cH (\tilde{H}')$ and with support 
$\tilde J (W^\fs \rtimes X_* (\tilde \cT) ) \tilde J$. We obtain an algebra 
homomorphism $\cH (\eta,\tau) : \cH (\tilde \cG,\tilde \tau)' \to \cH (\cG,\tau)$ with
\begin{equation}\label{eq:N.40}
\cH (\eta,\rho)(f) (g) = \left\{
\begin{array}{ll}
f (\eta^{-1}(g)) & g \in J (W^\fs \rtimes \eta (X_* (\cT)) ) J \\
0 & g \in G \setminus J (W^\fs \rtimes \eta (X_* (\cT)) ) J .
\end{array} \right.
\end{equation}
This is well-defined because any $f \in  \cH (\tilde \cG,\tilde \tau)'$ takes one
common value on the entire preimage $\eta^{-1}(g)$. Taking \eqref{eq:N.45} and 
Lemma \ref{lem:etaBernstein} into account, we consider the diagram
\begin{equation}\label{eq:N.34}
\xymatrix{
\Rep (\cG)^\fs \ar[d]^{\eta^*} & \Mod (\cH (\cG,\tau)) 
\ar[l]^{\mathrm{ind}_{\cH (\cG,\tau)}^{\cH (\cG)}} \ar[d]^{\cH (\eta,\tau)^*} & 
\Mod (\cH (H)) \ar[l]^{\sim} \ar[d]^{\phi_\eta^*} \\
\Rep (\tilde \cG)^{\tilde \fs} & \Mod (\cH (\tilde \cG,\tilde \tau)')
\ar[l]^{\mathrm{ind}_{\cH (\tilde \cG,\tilde \tau)'}^{\cH (\tilde \cG)}} &
\Mod (\cH (\tilde{H}')) \ar[l]^\sim 
} 
\end{equation}
The left upper horizontal arrow is invertible with as inverse the map that sends any
$\cG$-representation to its $\tau$-isotypical component, as in \eqref{eq:N.7}.

\begin{lem}\label{lem:etaDiagram}
The diagram \eqref{eq:N.34} commutes up to a natural isomorphism. 
\end{lem}
\begin{proof}
The right hand square commutes by definition, so consider only the left hand square.
Take $V \in \Mod (\cH (\cG,\tau))$. We must compare 
\begin{equation}\label{eq:N.35}
\cH (\tilde \cG) \otimes_{\cH (\tilde \cG,\tilde \tau)'} \cH (\eta,\tau)^* (V) 
\quad \text{with} \quad \eta^* \big( \cH (\cG) \otimes_{\cH (\cG,\tau)} V \big) .
\end{equation}
One problem that we encounter is the lack of a reasonable map 
$\cH (\tilde \cG) \to \cH (\cG)$. To overcome this we make use of the algebra
$\cH^\vee (\cG)$ of essentially left-compact distributions on $\cG$, which was
introduced in \cite{BeDe}. It naturally contains both $\cH (\cG)$ and a copy
of $\cG$. From \cite[\S 1.2]{BeDe} it is known that $\Mod (\cH^\vee (\cG))$ is
naturally equivalent with $\Rep (\cG)$. Moreover $\cH (\cG)$ is a two-sided
ideal of $\cH^\vee (\cG)$, so the modules \eqref{eq:N.35} are canonically
isomorphic with
\[
\cH^\vee (\tilde \cG) \otimes_{\cH (\tilde \cG,\tilde \tau)'} \cH (\eta,\tau)^* (V) ,
\text{ respectively } \eta^* \big( \cH^\vee (\cG) \otimes_{\cH (\cG,\tau)} V \big) . 
\]
An advantage of $\cH^\vee (\cG)$ over $\cH (\cG)$ is that it is functorial in
$\cG$, see \cite[Theorem 3.1]{Moy}. The algebra homomorphism 
\[
\cH^\vee (\eta) : \cH^\vee (\tilde \cG) \to \cH^\vee (\cG) \quad \text{extends}
\quad \cH (\eta,\tau) : \cH (\tilde \cG,\tilde \tau)' \to \cH (\cG,\tau) .
\]
This yields a canonical map
\[
\cH^\vee (\eta) \otimes \mathrm{id}_V : \cH^\vee (\tilde \cG) 
\otimes_{\cH (\tilde \cG,\tilde \tau)'} \cH (\eta,\tau)^* (V)
\to \cH^\vee (\cG) \otimes_{\cH (\cG,\tau)} V .
\]
By Lemma \ref{lem:etaProperties}.(4) $\cG = \eta (\tilde \cG) \cT$, and the
action of $\cT$ on $V$ is already given $\cH (\cG,\tau)$ since $\cT \subset
J X_* (\cT) J$. Therefore $\cH^\vee (\eta) \otimes \mathrm{id}_V$ is surjective. 

It is also $\tilde \cG$-equivariant if we regard its target as 
$\eta^* \big( \cH^\vee (\cG) \otimes_{\cH (\cG,\tau)} V \big)$. In particular
its kernel is a $\tilde \cG$-subrepresentation of $\cH^\vee (\tilde \cG) 
\otimes_{\cH (\tilde \cG,\tilde \tau)'} \cH (\eta,\tau)^* (V)$. As 
$(\tilde J,\tilde \tau)$ is a $\tilde \fs$-type, 
$\ker (\cH^\vee (\eta) \otimes \mathrm{id}_V)$ is of the form 
$\cH^\vee (\tilde \cG) \otimes_{\cH (\tilde \cG,\tilde \tau)} N$ for some
$\cH (\tilde \cG,\tilde \tau)$-module 
\[
N \subset \mathrm{ind}_{\cH (\tilde \cG,\tilde \tau)'}^{\cH 
(\tilde \cG,\tilde \tau)} \cH (\eta,\tau)^* (V).
\]
Let $E$ be a set of representatives for $\cW^{\tilde H} / \cW^\eta (\cW^H)$.
Then any element of $N$ can be written as $n = \sum_{w \in E} T_w v_w$.
We have 
\[
0 = \cH^\vee (\eta) \otimes \mathrm{id}_V (n) = \sum\nolimits_{w \in E} 
\cH^\vee (\eta) (T_w) \otimes_{\cH (\cG,\tau)} v_w .
\]
The elements $\cH^\vee (\eta)(T_w)$ with $w \in E$ are linearly independent over 
$\cH (\cG.\tau)$, because the support of $\cH^\vee (\eta)(T_w)$ is
$\eta (\tilde J w \tilde J) = J w J$. It follows that $v_w = 0$ for all $w \in E$.

Hence $N = 0$ and $\cH^\vee (\eta) \otimes \mathrm{id}_V$ is injective. This 
shows that \eqref{eq:N.34} commutes up to the canonical isomorphism between the
$\tilde \cG$-representations \eqref{eq:N.35}.
\end{proof}

It is clear that the formula \eqref{eq:N.39} also defines an algebra homomorphism
$\phi_\eta : \cH_v (\tilde{H}') \to \cH_v (H)$ for any $v \in \C^\times$, and
that these maps lift to a homomorphism of $\C [\mathbf{q^{\pm 1/2}}]$-algebras
\[
\phi_\eta : \cH_{\mathbf{\sqrt q}} (\tilde{H}') \to \cH_{\mathbf{\sqrt q}} (H) .
\]
Denote the category of finite length semisimple modules of an algebra
$A$ by $\Mod_{\fss} (A)$.

\begin{lem}\label{lem:semisimplicity}
\begin{enumerate}
\item $\phi_\eta^* : \Mod (\cH_v (\tilde{H}')) \to \Mod (\cH_v (H))$ preserves
finite length and complete reducibility.
\item There is a commutative diagram
\[
\xymatrix{
\Mod_{\fss}(\cH_q (H)) \ar @{<->}[r] \ar[d]^{\phi_\eta^*} & 
\Mod_{\fss}(\cW^H \ltimes X^* (T)) \ar[d]^{\eta^*} \\
\Mod_{\fss}(\cH_q (\tilde{H}')) \ar @{<->}[r] & 
\Mod_{\fss}(\cW^H \ltimes X^* (\tilde T)) 
} 
\]
in which the horizontal arrows extend the left slanted map in Theorem \ref{thm:S.4}
additively.
\end{enumerate}
\end{lem}
\begin{proof}
(1) For these considerations the kernel of $\phi_\eta$ plays no role, we need only
look at the subalgebra $\phi_\eta (\cH_v (\tilde{H}'))$ of $\cH_v (H)$. It has a 
basis $\{ T_w \theta_\lambda \mid w \in \cW^H, \lambda \in \eta (X^* (\tilde T)) \}$.
Since $\eta$ has commutative cokernel, $\cW^H \ltimes \eta (X^* (\tilde T)) +
X^* (T)^{\cW^H}$ is of finite index in $\cW^H \ltimes X^* (T)$, and it contains
$\cW^H \ltimes \Z R (H^\circ,T) + X^* (T)^{\cW^H}$. The group extension
\begin{equation}\label{eq:N.43}
\cW^H \ltimes \Z R (H^\circ,T) + X^* (T)^{\cW^H} \subset \cW^H \ltimes X^* (T)
\end{equation}
is of the form $X \subset X \rtimes \Gamma$, where $\Gamma \subset \cW^H 
\ltimes X^* (T)$ is the finite group of elements that preserve the fundamental
alcove in the Coxeter complex of $\cW^H \ltimes \Z R (H^\circ,T) + X^* (T)^{\cW^H}$.
Hence the inclusion of affine Hecke algebras corresponding to \eqref{eq:N.43}
is of the form
\[
\cH_v (H'') \subset \cH_v (H'') \rtimes \Gamma = \cH_v (H) . 
\]
It is well-known from Clifford theory \cite[Appendix A]{RamRam} that the restriction
map $\Mod (\cH_v (H'') \rtimes \Gamma) \to \Mod (\cH_v (H''))$ preserves finite
length and complete reducibility. Since $\phi_\eta (\cH_v (\tilde{H}'))$ lies between
these two algebras, the same holds for 
\[
\Mod (\cH_v (H)) \to \Mod \big( \phi_\eta (\cH_v (\tilde{H}')) \big) .
\]
(2) Consider the standard $\cH_{\mathbf{\sqrt q}}(H)$-module 
$\tilde \pi (\Phi,\rho) \otimes_\C \C [\mathbf{q^{\pm 1/2}}]$, as in \eqref{eq:N.15}.
Its specialization at a generic $v \in \C^\times$ is irreducible and it equals 
$\pi (t_v,x,\rho_v)$. Recall from \ref{sec:repAHA} that this is a 
$\cH_v (H)$-submodule of $\H_* (\cB_H^{t_v,x},\C) \otimes \C [\pi_0 (H)]$. By Lemma
\ref{lem:etaProperties} $H$ and $\tilde H$ have isomorphic varieties of Borel 
subgroups, and the description of $\cH_v (H)$-action entails that
\[
\phi_\eta^* \big( \H_* (\cB_H^{t_v,x},\C) \otimes \C [\pi_0 (H)] \big) \cong
\H_* (\cB_{\tilde H}^{\check \eta (t_v),\check \eta (x)},\C) \otimes \C [\pi_0 (H)] .
\]
By part (1) $\phi_\eta^* (\pi (t_v,x,\rho_v))$ is completely reducible. Hence
there is a unique representation $\tilde \rho = \oplus_i \tilde{\rho}_i$ of 
\[
\pi_0 \big( Z_{\tilde{H}'} (\check \eta (t_v),\check \eta (x) ) \big) =
\pi_0 \big( Z_{\tilde{H}'} (\check \eta \circ \Phi) \big)
\]
such that
\[
\pi (\check \eta (t_v),\check \eta (x), \tilde \rho ) :=
\oplus_i \pi (\check \eta (t_v),\check \eta (x), \tilde{\rho}_i ) . 
\]
We need to identify $\tilde \rho$. Like in the proof of Lemma \ref{lem:Smod},
the family of modules $\phi_\eta^* (\pi (t_v,x,\rho_v))$ depends algebraically
on $v$, so $\tilde \rho$ does not depend on $v \in \C^\times$ as long as 
$v$ is generic. As the set of generic parameters is dense in $\C^\times$, we
must have 
\[
\phi_\eta^* \big( \tilde \pi (\Phi,\rho) \otimes_\C \C [\mathbf{q^{\pm 1/2}}] \big)
\cong \tilde \pi (\Phi,\tilde \rho) \otimes_\C \C [\mathbf{q^{\pm 1/2}}] =
\oplus_i \tilde \pi (\Phi,\tilde{\rho}_i)
\otimes_\C \C [\mathbf{q^{\pm 1/2}}] .
\]
In particular this holds for $v=q$ and for $v=1$. Looking at the unique
irreducible quotients (for $v=q$) or at the subrepresentations in top homological
degree (for $v=1$), we find
\begin{align*}
& \phi_\eta^* (\pi (t_q,x,\rho_q)) \cong 
\pi (\check \eta (t_q),\check \eta (x), \tilde \rho) = 
\oplus_i \pi (\check \eta (t_q),\check \eta (x), \tilde{\rho}_i) , \\
& \eta^* (\tau (t_1,x,\rho_1)) \cong 
\tau (\check \eta (t_1),\check \eta (x), \tilde \rho) = 
\oplus_i \tau (\check \eta (t_1),\check \eta (x), \tilde{\rho}_i) .
\end{align*}
Thus the diagram in statement commutes for irreducible representations and, 
being additive, for all semisimple modules of finite length.
\end{proof}

Now we can determine the effect of $\eta^*$ on irreducible representations.
Let $(\Phi,\rho)$ be a KLR parameter for $\cG$, with $(t_q,x,\rho_q)$ as in
Lemma \ref{lem:compareParameters}. Recall that $\rho$ is trivial on the image
of $Z(G)$ in $\pi_0 (Z_G (\Phi))$.

\begin{prop}\label{prop:functoriality}
Let $\eta : \tilde \cG \to \cG, \check \eta : G \to \tilde G$ and $(\Phi,\rho)$
be as above. Identify $\pi_0 \big( \Cent_G (\Phi) / \Cent (G) \big)$ with a subgroup of
$\pi_0 \big( \Cent_{\tilde G} (\check \eta \circ \Phi) / \Cent (\tilde G) \big)$ via 
Lemma \ref{lem:pi0eta}. Then
\[
\eta^* \big( \pi (\Phi,\rho) \big) = \pi \big( \check \eta \circ \Phi, 
\mathrm{ind}_{\pi_0 (\Cent_G (\Phi) / \Cent (G))}^{\pi_0 ( \Cent_{\tilde G} 
(\check \eta \circ \Phi) / \Cent (\tilde G) )} \rho \big) .
\]
Here we use the convention $\pi (\check \eta \circ \Phi, \oplus_i \rho_i) =
\oplus_i \pi (\check \eta \circ \Phi,\rho_i)$ for \\ $\rho_i \in \Irr \big( 
\pi_0 ( \Cent_{\tilde G} (\check \eta \circ \Phi) / \Cent (\tilde G) )$.
\end{prop}
\noindent In particular this proves a precise version of 
Condition \ref{cond:Borel}.(5) for $\Irr (\cG,\cT)$.
\begin{proof}
By Lemma \ref{lem:etaDiagram} and \eqref{eq:N.46}
\begin{equation}\label{eq:N.36}
\begin{aligned}
\eta^* (\pi (\Phi,\rho)) & = \eta^* \big( \cH (\cG) \otimes_{\cH (\cG,\tau)}
\pi (t_q,x,\rho_q) \big) \\
& \cong \cH (\tilde \cG) \otimes_{\cH (\tilde \cG,\tilde \tau)'} \cH (\eta,\tau)^*
\pi (t_q,x,\rho_q) \\
& \cong \cH (\tilde \cG) \otimes_{\cH (\tilde \cG,\tilde \tau)} 
\mathrm{ind}_{\cH (\tilde \cG,\tilde \tau)'}^{\cH (\tilde \cG,\tilde \tau)} 
\cH (\eta,\tau)^* \pi (t_q,x,\rho_q) .
\end{aligned}
\end{equation}
By \eqref{eq:N.41} it suffices to analyse the module 
$\mathrm{ind}_{\cH (\tilde{H}')}^{\cH (\tilde H)} \phi_\eta^* \pi (t_q,x,\rho_q)$.
By Lemma \ref{lem:semisimplicity} the module $\phi_\eta^* \pi (t_q,x,\rho_q)$
has finite length and is semisimple, and its parameters can be read off from 
the $X^* (\tilde T) \rtimes \cW^H$-module $\eta^* (\tau (t_1,x,\rho_1))$.

For simplicity we drop the subscripts 1. Recall from \eqref{eq:S.35} that
$\tau (t,x,\rho)$ is isomorphic to 
\[
\mathrm{Ind}_{X^*(T) \rtimes \cW^H_t}^{X^*(T) \rtimes \cW^H} 
\big( \Hom_{\pi_0 (\Cent_H (t,x) / Z(G))} \big( \rho, H_* (\mathcal B^x_{M^\circ} 
,\C) \otimes \C [\Z_H (t,x) / \Z_{M^\circ}(x)] \big) \big) .
\]
By Lemma \ref{lem:etaProperties} $\check \eta$ induces an isomorphism
\[
M^\circ / \Z (G) = \Z_H (t)^\circ / \Z (G) \to \Z_{\tilde{H}'}(\check \eta 
(t))^\circ / \Z (\tilde{G}) =: \tilde{M}^\circ / \Z (\tilde{G}) .
\]
As $\check \eta$ also provides an isomorphism between the respective unipotent
varieties of $M$ and $\tilde{M} = \Z_{\tilde{H}}(\check \eta (t))$,
\[
\pi_0 (\Z_{M^\circ}(x) / \Z (H)) \cong 
\pi_0 (\Z_{\tilde{M}^\circ}(\check \eta (x)) / \Z (\tilde{H}')) .
\]
Steinberg's description of the centralizer of a semisimple element 
\cite[2.8]{Ste1965} and again Lemma \ref{lem:etaProperties} show that the 
inclusion $N_{\tilde{H}'} (\tilde T) \to \tilde{H}'$ induces a group isomorphism 
\[
\cW^H_{\check \eta (t)} / \cW^H_t \xrightarrow{\; \sim \;} \Z_{\tilde{H}'}
(\check \eta (t), \check \eta (x)) \big/ \check \eta (\Z_H (t,x)) \Z (\tilde{G}) .
\]
It follows that $\tau (t,x,\rho)$ is also isomorphic to
\begin{multline*}
\mathrm{Ind}_{X^*(T) \rtimes \cW^H_{\check \eta (t)}}^{X^*(T) \rtimes \cW^H} 
\big( \Hom_{\pi_0 (\Cent_H (t,x) / \Z (G))} \big( \rho, \\
H_{d(x)} (\mathcal B^x_{M^\circ} ,\C) \otimes \C [\Z_{\tilde{H}'} (\check \eta (t)
,\check \eta (x)) / \Z_{\tilde{M}^\circ}(\check \eta (x))] \big) \big) . 
\end{multline*}
By Lemma \ref{lem:pi0eta} the composition of this representation with $\eta$ is
\[
\begin{aligned}
\mathrm{Ind}_{X^*(\tilde T) \rtimes \cW^H_{\check \eta (t)}}^{X^*(\tilde T) \rtimes \cW^H} 
& \Big( \Hom_{\pi_0 (\Cent_H (t,x) / \Z (G))} \big( \rho, \\
& H_{d(x)} (\mathcal B^{\check \eta (x)}_{\tilde{M}^\circ} ,\C) \otimes \C [\Z_{\tilde{H}'} 
(\check \eta (t) ,\check \eta (x)) / \Z_{\tilde{M}^\circ}(\check \eta (x))] \big) \Big) \cong \\
\mathrm{Ind}_{X^*(\tilde T) \rtimes \cW^H_{\check \eta (t)}}^{X^*(\tilde T) \rtimes \cW^H}
& \Big( \Hom_{\pi_0 (\Z_{\tilde{H}'}(\check \eta (t),\check \eta (x)) / \Z (\tilde{G}))} \big( 
\mathrm{Ind}_{\pi_0 (\Cent_H (t,x) / \Z (G))}^{\pi_0 (\Z_{\tilde{H}'}(\check \eta (t),
\check \eta (x)) / \Z (\tilde{G}))} \rho , \\
& H_{d(x)} (\mathcal B^{\check \eta (x)}_{\tilde{M}^\circ} ,\C) \otimes \C [\Z_{\tilde{H}'} 
(\check \eta (t) ,\check \eta (x)) / \Z_{\tilde{M}^\circ}(\check \eta (x))] \big) \Big) \cong \\
\tau \big( \check \eta (t) ,\check \eta (x), & \mathrm{Ind}_{\pi_0 (\Cent_H (t,x) / \Z (G))}^{
\pi_0 (\Z_{\tilde{H}'}(\check \eta (t),\check \eta (x)) / \Z (\tilde{G}))} \rho  \big) .
\end{aligned}
\]
Now it follows from Lemma \ref{lem:semisimplicity} and Lemma \ref{lem:compareParameters} that 
\[
\phi_\eta^* \pi (t_q,x,\rho_q) \cong \pi \big( \check \eta (t_q) ,\check \eta (x), 
\mathrm{Ind}_{\pi_0 (\Cent_H (\Phi) / \Z (G))}^{\pi_0 (\Z_{\tilde{H}'}
(\check \eta \circ \Phi) / \Z (\tilde{G}))} \rho \big) .
\]
Next we induce this $\cH (\tilde{H}')$-module to $\cH (\tilde{H})$:
\[
\begin{aligned}
& \mathrm{ind}_{\cH (\tilde{H}')}^{\cH (\tilde H)} 
\pi \big( \check \eta (t_q) ,\check \eta (x), 
\mathrm{Ind}_{\pi_0 (\Cent_H (\Phi) / \Z (G))}^{\pi_0 (\Z_{\tilde{H}'}
(\check \eta \circ \Phi) / \Z (\tilde{G}))} \rho \big) = \\
& \mathrm{ind}_{\cH (\tilde{H}')}^{\cH (\tilde H)}  
\Hom_{\pi_0 (\Z_{\tilde{H}'}(\check \eta \circ \Phi))}
\big( \mathrm{Ind}_{\pi_0 (\Cent_H (\Phi) / \Z (G))}^{\pi_0 (\Z_{\tilde{H}'} 
(\check \eta \circ \Phi) / \Z (\tilde{G}))} \rho,
H_* (\cB_{\tilde H}^{\check \eta (x)},\C) \otimes \C [ \pi_0 (\tilde{H}') ] \big) \cong \\
& \Hom_{\pi_0 (\Z_{\tilde{H}'}(\check \eta \circ \Phi) / \Z (\tilde{G}))}
\big( \mathrm{Ind}_{\pi_0 (\Cent_H (\Phi) / \Z (G))}^{\pi_0 (\Z_{\tilde{H}'} 
(\check \eta \circ \Phi) / \Z (\tilde{G}))} \rho,
H_* (\cB_{\tilde H}^{\check \eta (x)},\C) \otimes \C [ \pi_0 (\tilde{H}) ] \big) \cong \\
& \Hom_{\pi_0 (\Z_{\tilde{H}}(\check \eta \circ \Phi) / \Z (\tilde{G}))}
\big( \mathrm{Ind}_{\pi_0 (\Cent_H (\Phi) / \Z (G))}^{\pi_0 (\Z_{\tilde{H}} 
(\check \eta \circ \Phi) / \Z (\tilde{G}))} \rho,
H_* (\cB_{\tilde H}^{\check \eta (x)},\C) \otimes \C [ \pi_0 (\tilde{H}) ] \big) \cong \\
& \pi \big( \check \eta (t_q) ,\check \eta (x), 
\mathrm{Ind}_{\pi_0 (\Cent_H (\Phi) / \Z (G))}^{\pi_0 (\Z_{\tilde{H}}
(\check \eta \circ \Phi) / \Z (\tilde{G}))} \rho \big) .
\end{aligned}
\]
From this we get to  the $\tilde \cG$-representation 
\[
\pi \big( \check \eta \circ \Phi, 
\mathrm{ind}_{\pi_0 (\Cent_G (\Phi) / \Cent (G))}^{\pi_0 ( \Cent_{\tilde G} 
(\check \eta \circ \Phi) / \Cent (\tilde G) )} \rho \big)
\]
via \eqref{eq:N.41} and \eqref{eq:N.36}.
\end{proof}

\section{The labelling by unipotent classes} 
\label{sec:unip}

In this and the next section we will show that the conjectures developed by the 
authors in \cite{ABP1,ABP2,ABPS1,ABPS2} hold for principal series representation
of split groups. These conjectures are expected to hold for every Bernstein
component of a quasi-split reductive group. For convenience we recall
the version that we will prove.

Let $T_{\cpt}^\fs$ denote the set of tempered representations in $T^\fs$.
Then $T^\fs_\cpt$ is a compact real torus, and it corresponds to the unique maximal 
compact subgroup of $T$ under \eqref{eq:N.22}.
The action of  $W^\fs$ on $T^\fs$ preserves $T_{\cpt}^\fs$, so we can form 
the compact orbifold $T_{\cpt}^\fs /\!/ W^\fs$.\\

\begin{conjecture} \label{conj:ABPS}
There exists a bijection 
\begin{equation}\label{eq:N.44}
\mu^\fs : T^{\mathfrak{s}}/\!/W^{\mathfrak{s}} \longrightarrow \Irr (\mathcal G)^\fs
\end{equation}
with the following properties:
\vspace{1mm}

(1) The bijection $\mu^\fs$ restricts to a bijection
\[
\mu^\fs \colon T_{\cpt}^\fs /\!/ W^\fs \longrightarrow 
\Irr (\mathcal G)^\fs_{\temp} ,
\]
where the subscript temp denotes the subset of tempered representations.
\vspace{1mm}

(2) The bijection $\mu^\fs$ is continuous, where $T^\fs /\!/W^\fs$ has the Zariski 
topology and $\Irr (\mathcal G)^\fs$ has the Jacobson topology. The composition 
\[
\pi^\fs \circ \mu^\fs : T^\fs /\!/W^\fs \to \Irr (\cG)^\fs \to T^\fs/ W^\fs
\]
of $\mu^\fs$ with the cuspidal support map $\pi^\fs$
is a finite morphism of affine algebraic varieties.
\vspace{1mm}

(3) There is an algebraic family
\[
\theta_z \colon T^\fs/\!/W^\fs \longrightarrow T^\fs/W^\fs
\]
of finite morphisms of algebraic varieties, with $z \in \Cset^{\times}$, such that
$\theta_1$ is the canonical projection and $\theta_{\sqrt{q}} = \pi^\fs \circ \mu^\fs$.
\vspace{1mm}

(4) Correcting cocharacters.  For each irreducible component  
$\bf{c}$ of the affine variety $T^\fs/\!/W^\fs$ there is a cocharacter 
(i.e. a homomorphism of algebraic groups)
\[
h_{\bc} \colon \Cset^{\times} \longrightarrow  T^\fs
\]
such that 
\begin{equation*}
\theta_{z} [w,t] = b(h_{\bc}(z) \cdot t) 
\end{equation*}
for all $[w, t]\in\bf{c}$, where 
$b\colon T^\fs \longrightarrow T^\fs/W^\fs$ is the quotient map.
\vspace{1mm}

(5) $L$-packets. This property is conditional on the existence of Langlands parameters 
for the block $\Irr(\cG)^\fs$.   In that case, the intersection of an $L$-packet with 
that block is well-defined. This property refers to the intersection of such an $L$-packet 
with the  given Bernstein block.   

Let $\{\bc_1, \ldots, \bc_r\}$ be the irreducible components of the affine 
variety $T^\fs /\!/W^\fs$. There exists a complex reductive group $H$ and, for every 
irreducible component $\bc$ of $T^\fs \q W^\fs$, a unipotent conjugacy class $\lambda (\bc)$ 
in $H$, such that for every two points $[w,t]$ and $[w',t]$ of $T^\fs/\!/W^\fs$:
\[
\mu^\fs [w,t] \;\mathrm{and}\; \mu^\fs [w',t'] \text{ are in the same L-packet}
\]
if and only
\begin{itemize}
\item $\theta_z [w,t] = \theta_z [w',t']$ for all $z \in \C^\times$; 
\item $\lambda(\bc) = \lambda(\bc')$,  where $[w,t] \in \bc$ and $[w',t'] \in \bc'$.
\end{itemize}
\end{conjecture}

Let $\fs \in \mathfrak B (\cG,\cT)$ and construct $c^\fs$ as in Section \ref{sec:Lp}. 
By Theorem \ref{thm:main} we can parametrize $\Irr (\cG)^\fs$ with
$H$-conjugacy classes of KLR parameters $(\Phi,\rho)$ such that 
$\Phi \big|_{\fo_F^\times} = c^\fs$. We note that \{KLR parameters$\}^\fs$ 
is naturally labelled by the unipotent classes in $H$: 
\begin{equation}
\{\text{KLR parameters} \}^{\fs,[x]} := \big\{ (\Phi,\rho) \mid 
\Phi \big( 1, \matje{1}{1}{0}{1} \big) \text{ is conjugate to } x \big\} .
\end{equation}
In this way we can associate to any of the parameters in Theorem \ref{thm:main}
a unique unipotent class in $H$:
\begin{equation} \label{eq:labelling}
\Irr (\cG )^\fs = \bigcup\nolimits_{[x]} \Irr (\cG )^{\fs,[x]} ,\quad
(T^\fs /\!/ W^\fs )_2 = \bigcup\nolimits_{[x]} (T^\fs /\!/ W^\fs )_2^{[x]} .
\end{equation}
Via the affine Springer correspondence from Section \ref{sec:affSpringer} the set of
equivalence classes in \{KLR parameters$\}^\fs$ is naturally in bijection with
$(T^\fs /\!/ W^\fs )_2$. Recall from Section \ref{sec:extquot} that 
\[
\widetilde{T^\fs} = \{ (w,t) \in W^\fs \times T^\fs \mid w t = t \}
\]
and $T^\fs /\!/ W^\fs = \widetilde{T^\fs} / W^\fs$. In view of Section \ref{sec:extquot}
$(T^\fs /\!/ W^\fs )_2$ is also in bijection with $T^\fs /\!/ W^\fs$, albeit not naturally. 

Only in special cases a canonical bijection $T^\fs /\!/ W^\fs \to (T^\fs /\!/ W^\fs)_2$ 
is available. For example when $G = \GL_n (\C)$, the finite group $\cW^H_t$  
is a product of symmetric groups: in this case there is a canonical  $c$-$\Irr$ system,    
according to the classical theory of Young tableaux. 

We have to construct a map \eqref{eq:N.44} with the desired properties, but
in general it can already be hard to define any suitable map from \{KLR parameters$\}^\fs$
to $T^\fs /\!/ W^\fs$, because it is difficult to compare the parameters $\rho$ 
for different $\Phi$'s. It goes better the other way round and with $\Irr (\cG )^\fs$
as target. In this way will transfer the labellings \eqref{eq:labelling} to
$T^\fs /\!/ W^\fs$.

From \cite[Section 8]{Roc} we know that $\Irr (\cG )^\fs$ is in 
bijection with the equivalence classes of irreducible representations of the
extended affine Hecke algebra $\mathcal H (H)$. To relate it to $T^\fs /\!/ W^\fs$
the parametrization of Kazhdan, Lusztig and Reeder is unsuitable, it is more
convenient to use the methods developed in \cite{Opd,Sol}. 

Let us fix some notations. Choose a Borel subgroup $B \subset H^\circ$ 
containing $T$. Let $P$ be a set of roots of $(H^\circ,T)$ which are simple with
respect to $B$ and let $R_P$ be the root system that they span. They determine
a parabolic subgroup $W_P \subset W^\fs$ and a subtorus
\[
T^P := \{ t \in T \mid \alpha (t) = 1 \; \forall \alpha \in R_P \}^\circ .
\]
In \cite[Theorem 3.3.2]{Sol} $\Irr (\mathcal H (H))$ is mapped, in a natural
finite-to-one way, to equivalence classes of triples $(P,\delta,t)$. Here
$P$ is as above, $t \in T^P$ and $\delta$ is a discrete series
representation of a parabolic subalgebra $\mathcal H_P$ of $\mathcal H (H)$.
This $t$ is the same as in the affine Springer parameters.

The pair $(P,\delta)$ gives rise to a \emph{residual coset} $L$ in the sense of
\cite[Appendix A]{Opd}. Explicitly, it is the translation of $T^P$ by an
element $cc(\delta) \in T$ that represents the central character of $\delta$ 
(a $W_P$-orbit in a subtorus $T_P \subset T$). 
The element $cc(\delta) t \in L$ corresponds to $t_q$. 
The collection of residual cosets is stable under the action of $W^\fs$.

\begin{prop}\label{prop:residual}
\begin{enumerate}
\item There is a natural bijection between
\begin{itemize}
\item $H$-conjugacy classes of Langlands parameters $\Phi$ with 
$\Phi \big|_{\mathbf I_F} = c^\fs$;
\item $W^\fs$-conjugacy classes of pairs $(t_q,L)$ with $L$ a residual coset for
$\mathcal H (H)$ and $t_q \in L$.
\end{itemize}
\item Let $Y^P$ be the union of the residual cosets of $T^P$. The stabilizer of $Y^P$
in $W^\fs$ is the stabilizer of $R_P$.
\item Suppose that $w \in W^\fs$ fixes $cc(\delta)$. Then $w$ stabilizes $R_P$.
\end{enumerate}
\end{prop}
\begin{proof}
(1) Opdam constructed the maps in both directions for $\mathcal H (H^\circ)$. 
To go from $\mathcal H (H^\circ)$ to $\mathcal H (H)$ is easy,
one just has to divide out the action of $\pi_0 (H)$ on both sides. 

Let us describe the maps for $H^\circ$ more explicitly. 
To a residual coset $L$ Opdam \cite[Proposition B.3]{Opd} 
associates a unipotent element $x \in B$ such that $l x l^{-1} = x^q$ for all $l \in L$. 
Then $\Phi$ is a Langlands parameter with data $t_q,x$.

For the opposite direction we may assume that 
\[
\Phi \big( \mathbf W_F ,\matje{z}{0}{0}{z^{-1}} \big) \subset T \quad \forall z \in \C^\times
\]
and that $x = \Phi \big( 1,\matje{1}{1}{0}{1} \big) \in B$. Then 
\[
T^P := \Z_T \big(\Phi (\mathbf I_F \times \SL_2 (\C)) \big)^\circ = 
\Z_T \big( \Phi (\SL_2 (\C)) \big)^\circ
\]
is a maximal torus of $\Z_{H^\circ}\big(\Phi (\mathbf I_F \times \SL_2 (\C)) \big)$. We take
\[
t_q = \Phi \Big(\varpi_F , \matje{q^{1/2}}{0}{0}{q^{-1/2}} \Big) 
\quad \text{and} \quad L = T^P t_q .
\]
This is essentially \cite[Proposition B.4]{Opd}, but our way to write it down avoids
Opdam's assumption that $H^\circ$ is simply connected.\\
(2) Clear, because any element that stabilizes $Y^P$ also stabilizes $T^P$.\\
(3) Since $cc(\delta)$ represents the central character of a discrete series representation
of $\cH_P$, at least one element (say $r$) in its $W_P$-orbit lies in the obtuse 
negative cone in the subtorus $T_P \subset T$ (see Lemma 2.21 and Section 4.1 of \cite{Opd}).
That is, $\log |r|$ can be written as $\sum_{\alpha \in P} c_\alpha \alpha^\vee$ with
$c_\alpha < 0$. Some $W_P$-conjugate $w'$ of $w \in W^\fs$ fixes $r$ and hence $\log |r|$.
But an element of $W^\fs$ can only fix $\log |r|$ if it stabilizes the
collection of coroots $\{ \alpha^\vee \mid \alpha \in P\}$. It follows that $w'$ and $w$ 
stabilize $R_P$.
\end{proof}

In particular the above natural bijection associates to any $W^\fs$-conjugacy class of
residual cosets $L$ a unique unipotent class $[x]$ in $H$. 
Conversely a unipotent class $[x]$ can correspond to more than one $W^\fs$-conjugacy 
class of residual cosets, at most the number of connected components of $\Z_T (x)$
if $x \in B$.

Let $\mathfrak U^\fs \subset B$ be a set of representatives for the unipotent classes in $H$. 
For every $x \in \mathfrak U^\fs$ we choose an algebraic homomorphism
\[
\gamma_x \colon \SL_2 (\C) \to H \quad \text{with} \quad \gamma_x \matje{1}{1}{0}{1} = x
\quad \text{and} \quad \gamma_x \matje{z}{0}{0}{z^{-1}} \in T .
\]
As noted in \eqref{eqn:gamt} all choices for $\gamma_x$ are conjugate under 
$\Z_{M^\circ}(x)$. For each $x \in \mathfrak U^\fs$ we define
\[
\{ \text{KLR parameters} \}^{\fs,x} = \{ (\Phi,\rho) \mid \Phi 
\big|_{\mathbf I_F \times \SL_2 (\C)} = c^\fs \times \gamma_x , \Phi (\varpi_F) \in T \} .
\]
We endow this set with the topology such that a subset is open if and only if its
image in $T$ under $(\Phi,\rho) \mapsto \Phi (\varpi)$ is open. In this way
\begin{equation} \label{eq:UsParam}
\bigsqcup\nolimits_{x \in \mathfrak U^\fs} \{ \text{KLR parameters} \}^{\fs,x}
\end{equation}
becomes a nonseparated scheme with maximal separated quotient \\
$\sqcup_{x \in \mathfrak U^\fs} \Z_T (\im \gamma_x)$. Notice that \eqref{eq:UsParam} 
contains representatives for all equivalence classes in \{KLR parameters$\}^\fs$.

\begin{prop}\label{prop:mus}
Assume that Condition \ref{con:char} holds.
There exists a continuous bijection
\[
\mu^\fs : T^\fs /\!/ W^\fs \to \Irr (\cG )^\fs 
\]
such that:
\begin{enumerate}
\item The diagram
\[
\xymatrix{
T^\fs /\!/ W^\fs \ar[r]^{\mu^\fs} \ar[d]^{\rho^\fs} & \Irr (\cG )^\fs \ar[r] &
\{\text{$\KLR$ parameters} \}^\fs / H \ar[d] \\
T^\fs / W^\fs & & c (H)_{\ss} \ar[ll] } 
\]
commutes. Here the unnamed horizontal maps are those from Theorem \ref{thm:main} 
and the right vertical arrow sends $(\Phi,\rho)$ to the $H$-conjugacy class of 
$\Phi (\varpi_F)$.
\item For every unipotent element $x \in H$ the preimage
\[
(T^\fs /\!/ W^\fs )^{[x]} := (\mu^\fs )^{-1} \big( \Irr (\cG )^{\fs,[x]} \big)
\]
is a union of connected components of $T^\fs /\!/ W^\fs$.
\item Let $\epsilon$ be the map that makes the diagram
\[
\xymatrix{ 
T^\fs /\!/ W^\fs \ar[rr]^\epsilon \ar[dr]^{\mu^\fs} & & 
(T^\fs /\!/ W^\fs )_2 \ar[dl] \\
& \Irr (\cG )^\fs & }
\]
commute. Then $\epsilon$ comes from a $c$-$\Irr$ system.
\item $T^\fs /\!/ W^\fs \xrightarrow{\; \mu^\fs \;} \Irr (\cG )^\fs \to 
\{\text{$\KLR$ parameters} \}^\fs / H$ lifts to a map 
\[
\widetilde{\mu^\fs} : \widetilde{T^\fs} \to 
\bigsqcup\nolimits_{x \in \mathfrak U^\fs} \{ \text{KLR parameters} \}^{\fs,x}
\]
such that the restriction of $\widetilde{\mu^\fs}$ to any connected component of 
$\widetilde{T^\fs}$ is algebraic and an isomorphism onto its image.
\end{enumerate}
\end{prop}
\begin{proof} 
Proposition \ref{prop:residual}.1 yields a natural finite-to-one map from $\Irr 
(\mathcal H (H))$ to $W^\fs$-conjugacy classes $(t_q,L)$, namely 
\begin{equation}\label{eq:pitqL}
\pi (\Phi,\rho) \mapsto \Phi \mapsto (t_q,L) .
\end{equation} 
In \cite[Theorem 3.3.2]{Sol} this map was obtained in 
a different way, which shows better how the representations depend on the parameters 
$t,t_q,L$. That was used in \cite[Section 5.4]{Sol} to find a continuous bijection
\begin{equation}\label{eq:mus}
\mu^\fs : T^\fs /\!/ W^\fs \to \Irr (\mathcal H (H)) \cong \Irr (\cG )^\fs .
\end{equation}
The strategy is essentially a step-by-step creation of a $c$-$\Irr$ system for
$T^\fs /\!/ W^\fs$ and $\mathcal H (H)$. One starts with the components of 
$\tilde T_\fs$ of dimension 0, and proceeds by induction on the dimension.
In step $d$ one considers $d$-dimensional families of representations, and uses that  
in \eqref{eq:mus} the fibers over a fixed $t \in T^\fs / W^\fs$ have the same
cardinality on both sides.

Only the condition on the unit element
and the trivial representation is not considered in \cite{Sol}. Fortunately
there is some freedom left, which we can exploit to ensure that $\mu^\fs (1,T^\fs)$ 
is the family of spherical $\mathcal H (H)$-representations, see Section 
\ref{sec:spherical}. 
This is possible because every principal series representation of $\mathcal H (H)$ has 
a unique irreducible spherical subquotient, so choosing that for $\mu^\fs (1,t)$ does
not interfere with the rest of the construction. Via Theorem \ref{thm:main} we can 
transfer this $c$-$\Irr$ system to a $c$-$\Irr$ system for the two extended 
quotients of $T^\fs$ by $W^\fs$, so property (3) holds.

By construction the triple $(P,\delta,t)$ associated to the representation 
$\mu^\fs (w,t)$ has the same $t \in T$, modulo $W^\fs$. 
That is, property (1) is fulfilled.

Furthermore $\mu^\fs$ sends every connected component of $T^\fs /\!/ W^\fs$ to a family
of representations with common parameters $(P,\delta)$. Hence these representation are
associated to a common residual coset $L$ and to a common unipotent class $[x]$,
which verifies property (2).

Let $\mathbf c$ be a connected component of $(T^\fs /\!/ W^\fs )^{[x]}$, with $x \in
\mathfrak U^\fs$. The proof of Proposition \ref{prop:residual}.1 shows that 
$\mathbf c$ can be realized in $\Z_T (\im \gamma_x)$. In other words, we can find a 
suitable $w = w (\mathbf c) \in W^\fs$ with $T^w \subset \Z_T (\im \gamma_x)$. Then there 
is a connected component $T^w_{\mathbf c}$ of $T^w$ such that
\[
\begin{split}
& \mathbf c := \big( w,T^w_{\mathbf c} / \Z (w,\mathbf c) \big) ,\\ 
& \text{where } \Z (w,\mathbf c) = \{ g \in \Z_{W^\fs}(w) \mid g \cdot T^w_{\mathbf c} = 
T^w_{\mathbf c} \} . 
\end{split}
\]
In this notation $\widetilde{\mathbf c} := (w,T^w_{\mathbf c})$ is a connected component of
$\widetilde{T^\fs}$. We want to define $\widetilde{\mu^\fs} : \widetilde{\mathbf c} \to
\{\text{KLR parameters} \}^{\fs,x}$. For every $[w,t] \in \mathbf c ,\;
\mu^\fs [w,t]$ determines an equivalence class in \{KLR parameters$\}^{\fs,x}$. Any
$(\Phi,\rho)$ in this equivalence class satisfies $\Phi (\varpi_F) \in \Z_T (\im \gamma_x)
\cap W^\fs t$. Hence there are only finitely many possibilities for $\Phi (\varpi_F)$,
say $t_1,\ldots,t_k$. For every such $t_i$ there is a unique 
$(\Phi_i,\rho_i) \in \{\text{KLR parameters} \}^{\fs,x}$ with $\Phi_i (\varpi_F) = t_i$ and 
$\pi (\Phi_i,\rho_i) \cong \mu^\fs [w,t]$. Every element of $\widetilde{\mathbf c}$ lying 
over $[w,t] \in \mathbf c$ is of the form $(w,t_i)$. (Not every $t_i$ is eligible though, 
for some we would have to modify $w$.) We put
\[
\widetilde{\mu^\fs} (w,t_i) := (\Phi_i ,\rho_i) .
\]
Since the $\rho$'s are irrelevant for the topology on \{KLR parameters$\}^{\fs,x} ,\;
\widetilde{\mu^\fs} (\widetilde{\mathbf c})$ is homeomorphic to $T^w_{\mathbf c}$ and 
$\widetilde{\mu^\fs} : \widetilde{\mathbf c} \to \widetilde{\mu^\fs} (\widetilde{\mathbf c})$
is an isomorphism of affine varieties. This settles the final property (4).
\end{proof}

\section{Correcting cocharacters and L-packets}
\label{sec:cochar}

In this section we construct the correcting cocharacters on the extended quotient
$T^\fs /\!/ W^\fs$. With part (5) of Conjecture \ref{conj:ABPS}, these show how
to determine when two elements of $T^\fs /\!/ W^\fs$ give rise to $\cG$-representations
in the same L-packets.

Every KLR parameter $(\Phi,\rho)$ naturally determines a cocharacter $h_\Phi$ 
and elements $\theta (\Phi,\rho,z) \in T^\fs$ by
\begin{equation}\label{eq:defhPhi}
\begin{aligned}
& h_\Phi (z) = \Phi \big( 1,\matje{z}{0}{0}{z^{-1}} \big) ,\\
& \theta (\Phi,\rho,z) = \Phi \big( \varpi_F, \matje{z}{0}{0}{z^{-1}} \big) =
\Phi (\varpi_F) h_\Phi (z) .
\end{aligned}
\end{equation}
Although these formulas obviously do not depend on $\rho$, it turns out to be convenient to 
include it in the notation anyway.
However, in this way we would end up with infinitely many correcting cocharacters, most
of them with range outside $T$. To reduce to finitely many cocharacters with values in $T$, 
we will restrict to KLR parameters associated to $x \in \mathfrak U^\fs$.

Recall that Proposition \ref{prop:mus}.2 determines a labelling of the connected 
components of $T^\fs /\!/ W^\fs$ by unipotent classes in $H$. This enables us to define the 
correcting cocharacters: for a connected component $\mathbf c$ of $T^\fs /\!/ W^\fs$ with 
label (represented by) $x \in \mathfrak U^\fs$ we take the cocharacter 
\begin{equation}\label{eq:defhx}
h_{\mathbf c} = h_x : \C^\times \to T ,\quad h_x (z) = \gamma_x \matje{z}{0}{0}{z^{-1}} .
\end{equation}
Let $\widetilde{\mathbf c}$ be a connected component of $\widetilde{T^\fs}$ that projects
onto $\mathbf c$. We define
\begin{equation} \label{eq:defThetaz}
\begin{aligned}
& \widetilde{\theta_z} : \widetilde{\mathbf c} \to T^\fs ,& & (w,t) \mapsto 
\theta \big( \widetilde{\mu^\fs}(w,t),z \big) , \\
& \theta_z : \mathbf c \to T^\fs / W^\fs ,& & [w,t] \mapsto W^\fs \widetilde{\theta_z}(w,t) . 
\end{aligned}
\end{equation}
For $\widetilde{\mathbf c}$ as in the proof of Proposition \ref{prop:mus}, which we can
always achieve by adjusting by element of $W^\fs$, our construction results in
\[
\widetilde{\theta_z} (w,t) = t \, h_{\mathbf c}(z) .
\]

\begin{lem}\label{lem:Lpackets}
Let $[w,t],[w',t'] \in T^\fs /\!/W^\fs$. Then $\mu^\fs [w,t]$ and $\mu^\fs [w',t']$ are
in the same L-packet if and only if
\begin{itemize}
\item $[w,t]$ and $[w',t']$ are labelled by the same unipotent class in $H$;
\item $\theta_z [w,t] = \theta_z [w',t']$ for all $z \in \C^\times$.
\end{itemize}
\end{lem}
\begin{proof}
Suppose that the two $\cG$-representations $\mu^\fs [w,t] = \pi (\Phi,\rho)$ and \\
$\mu^\fs [w',t'] = \pi (\Phi',\rho')$ belong to the 
same L-packet. By definition this means that $\Phi$ and $\Phi'$ are $G$-conjugate. 
Hence they are labelled by the same unipotent class, say $[x]$ with $x \in \mathfrak U^\fs$. 
By choosing suitable representatives we may assume that $\Phi = \Phi'$ and that 
$\{(\Phi,\rho),(\Phi,\rho')\} \subset \{\text{KLR parameters} \}^{\fs,x}$. Then
$\theta (\Phi,\rho,z) = \theta (\Phi,\rho',z)$ for all $z \in \C^\times$. Although
in general $\theta (\Phi,\rho,z) \neq \widetilde{\theta_z} (w,t)$, they differ only 
by an element of $W^\fs$. Hence $\theta_z [w,t] = \theta_z [w',t']$ for all $z \in \C^\times$.

Conversely, suppose that $[w,t],[w',t']$ fulfill the two conditions of the lemma. Let 
$x \in \mathfrak U^\fs$ be the representative for the unipotent class which labels them.
By Proposition \ref{prop:residual}.1 we may assume that $T^w \cup T^{w'} \subset 
\Z_T (\im \gamma_x)$. Then 
\[
\widetilde{\theta_z} [w,t] = t \, h_x (z) \quad \text{and} \quad
\widetilde{\theta_z} [w',t'] = t' \, h_x (z)
\]
are $W^\fs$ conjugate for all $z \in \C^\times$. As these points depend continuously on $z$
and $W^\fs$ is finite, this implies that there exists a $v \in W^\fs$ such that
\[
v (t \, h_x (z)) = t' \, h_x (z) \quad \text{for all } z \in \C^\times .
\]
For $z = 1$ we obtain $v(t) = t'$, so $v$ fixes $h_x (z)$ for all $z$. Via the Proposition
\ref{prop:residual}.1, $h_x (q^{1/2})$ becomes an element $cc(\delta)$ for a residual coset
$L_x$. By parts (2) and (3) of Proposition \ref{prop:residual} $v$ stabilizes the collection
of residual cosets determined by $x$, namely the connected components of $\Z_T (\im \gamma_x)
h_x (q^{1/2})$. 

Let $(t_q,L), (t'_q,L')$ be associated to $\mu^\fs [w,t], \mu^\fs [w',t']$
by \eqref{eq:pitqL}. Then $t_q = t h_x (q^{1/2})$ and $t'_q = t' h_x (q^{1/2})$, so the above
applies. Hence $v$ sends $L$ to another residual coset determined by $x$. As $v(L)$ 
contains $t'_q$, it must be $L'$. Thus $(t_q,L)$ and $(t'_q,L')$ are $W^\fs$-conjugate, which
by Proposition \ref{prop:residual}.1 implies that they correspond to conjugate Langlands
parameters $\Phi$ and $\Phi'$. So $\mu^\fs [w,t]$ and $\mu^\fs [w',t']$ are
in the same L-packet.
\end{proof}

\begin{cor}\label{cor:properties}
Properties 1--5 of Conjecture \ref{conj:ABPS} hold for $\mu^\fs$ as in 
Proposition \ref{prop:mus}, with the morphism $\theta_z$ from \eqref{eq:defThetaz}
and the labelling by unipotent classes in $H$. 

Together with Theorem \ref{thm:main} this proves Conjecture \ref{conj:ABPS} 
for all Bernstein components in the principal series of a split reductive 
$p$-adic group (under Condition \ref{con:char} on the residual characteristic).
\end{cor}
\begin{proof}
Property 1 follows from Theorem \ref{thm:main}.2 and Proposition \ref{prop:mus}.1.
The definitions of \eqref{eq:defhx} and \eqref{eq:defThetaz} establish
property 4. The construction of $\theta_z$, in combination with Theorem \ref{thm:main}.1 
and Proposition \ref{prop:mus}.1, shows that properties 2 and 3 are fulfilled. 
Property 5 is none other than Lemma \ref{lem:Lpackets}.
\end{proof}


\begin{thebibliography}{99}

\bibitem[AdRo]{AdRo} J.D. Adler, A. Roche, An intertwining result for $p$-adic groups,
Canad. J. Math. {\bf 52.3} (2000), 449--467.
\bibitem[ABP1]{ABP1} A.-M. Aubert, P.~Baum, R.~Plymen, The Hecke algebra of a
reductive $p$-adic group: a geometric conjecture, pp. 1--34 in:
Noncommutative geometry and number theory, Eds: C. Consani and M.
Marcolli, Aspects of Mathematics {\bf E37}, Vieweg Verlag (2006).
\bibitem[ABP2]{ABP2} A.-M. Aubert, P. Baum, R.J. Plymen, Geometric structure in
the representation theory of
$p$-adic groups,  C.R. Acad. Sci. Paris, Ser. I {\bf 345} (2007), 573--578.
\bibitem[ABPS1]{ABPS1} A.-M. Aubert, P. Baum, R.J. Plymen, M. Solleveld, 
Geometric structure in smooth dual and local Langlands conjecture,
Japan J. Math 9 (2014), 99--136.
\bibitem[ABPS2]{ABPS2} A.-M. Aubert, P. Baum, R.J. Plymen, M. Solleveld, 
Geometric structure for Bernstein blocks, arXiv:1408.0673, to appear in 
J. Noncommutative Geometry.
\bibitem[ABPS3]{ABPS7} A.-M. Aubert, P.F. Baum, R.J. Plymen, M. Solleveld,
``Conjectures about $p$-adic groups and their noncommutative geometry'',
arXiv:1508.02837 (2015).
\bibitem[BaMo]{BM} D. Barbasch, A. Moy, A unitarity criterion for $p$-adic
groups, Invent. Math. {\bf 98} (1989), 19--37.
\bibitem[BeDe]{BeDe} J.N. Bernstein, P. Deligne, Le "centre" de Bernstein, 
pp. 1--32 in: 
Repr\'esentations des groupes r\'eductifs sur un corps local,
Travaux en cours, Hermann, 1984.
\bibitem[Bor1]{Bor} A. Borel, Admissible representations of a semi-simple group over a 
local field with vectors fixed under an Iwahori subgroup, 
Inv. Math. {\bf 35} (1976), 233--259.
\bibitem[Bor2]{BorAut} A. Borel, Automorphic L-functions, pp. 27--61 in: 
Automorphic forms, representations and L-functions. Part 1,
Proc. Symp. Pure Math {\bf 33.2} (1979).
\bibitem[BHK]{BHK} C.J. Bushnell, G. Henniart, P.C. Kutzko, Types and explicit Plancherel 
formulae for reductive $p$-adic groups, pp. 55--80 in: On certain L-functions,
Clay Math. Proc. {\bf 13}, American Mathematical Society, 2011.
\bibitem[BuKu]{BKtyp} C.J. Bushnell, P.C. Kutzko, Smooth representations of
reductive $p$-adic groups: Structure theory via types, Proc.
London Math. Soc. {\bf 77} (1998), 582--634.
\bibitem[Car]{Carter} R.W. Carter, Finite Groups of Lie Type, Conjugacy classes
and complex characters, Wiley Classics Library, 1993.
\bibitem[Cas]{Cas} W. Casselman, The unramified principal series representations of 
$p$-adic groups I. The spherical function,
Compos. Math. {\bf 40} (1980), 387--406.
\bibitem[DeOp]{DeOp} P. Delorme, E.M. Opdam,      
The Schwartz algebra of an affine Hecke algebra,
J. reine angew. Math. {\bf 625} (2008), 59--114.
\bibitem[Dix]{Dix} J. Dixmier, Les C*-alg\`ebres et leurs representations,
Cahiers Scientifiques {\bf 29}, Gauthier-Villars \'Editeur, Paris, 1969.
\bibitem[ChGi]{CG} N. Chriss, V. Ginzburg, Representation theory and complex
geometry, Birkh\"auser, 2000.
\bibitem[FoRo]{FR} R.~Fowler, G.~R\"ohrle, On cocharacters associated to
nilpotent elements of reductive groups, Nagoya Math.~J. {\bf 190} (2008), 105--128.
\bibitem[GoRo]{GoRo} D. Goldberg, A. Roche, Hecke algebras and $SL_n$-types,
Proc. London Math. Soc. {\bf 90.1} (2005), 87--131.
\bibitem[HeOp]{HeOp} G.J. Heckman, E.M. Opdam, Harmonic analysis for affine Hecke algebras, 
pp. 37--60 in: Current developments in mathematics, Int. Press, Boston MA, 1996
\bibitem[Hot]{Hot} R.~Hotta, On Springer's representations, J. Fac. Sci. Uni.
Tokyo, IA {\bf 28} (1982), 863--876.
\bibitem[Iwa]{I} K. Iwasawa, Local class field theory, Oxford Math. Monograph, 1986.
\bibitem[Jan]{J} J.C.~Jantzen, Nilpotent Orbits in Representation Theory, in: Lie
Theory, Lie Algebras and Representations (J.-P.~Anker and B.~Orsted,
eds.), Progress in Math. {\bf 228}, Birkh\"auser, Boston, 2004.
\bibitem[Kat]{Kat} S.-I. Kato,
A realization of irreducible representations of affine Weyl groups,
Indag. Math. \textbf{45.2} (1983), 193--201.
\bibitem[KaLu]{KL} D.~Kazhdan, G.~Lusztig, Proof of the Deligne-Langlands
conjecture for Hecke algebras, Invent. math. {\bf 87} (1987), 153--215.
\bibitem[Lus1]{LuCellsII} G. Lusztig, Cells in affine Weyl groups, II, J. Algebra
{\bf 109} (1987), 536--548.
\bibitem[Lus2]{LuCellsIII} G. Lusztig, Cells in affine Weyl groups, III, J. Fac.
Sci. Univ. Tokyo, Sect. IA, Math, {\bf 34} (1987), 223--243.
\bibitem[Lus3]{LuCellsIV} G. Lusztig, Cells in affine Weyl groups, IV, J. Fac.
Sci. Univ. Tokyo, Sect. IA, Math, {\bf 36} (1989), 297--328.
\bibitem[Lus4]{LuGrad} G. Lusztig, Affine Hecke algebras and their graded version
J. Amer. Math. Soc {\bf 2.3} (1989), 599--635.
\bibitem[Mor]{MENS} L.~Morris, Tamely ramified supercuspidal representations,
Ann. scient. \'Ec. Norm. Sup. 4e s\'erie, {\bf 29} (1996), pp.~639--667.
\bibitem[Moy]{Moy} A. Moy, Distribution algebras on p-adic groups and Lie algebras,
Canad. J. Math. {\bf 63.5} (2011), 1137--1160.
\bibitem[Opd]{Opd} E.M. Opdam, On the spectral decompostion of affine Hecke algebras,
J. Inst. Math. Jussieu {\bf 3} (2004), 531--648.
\bibitem[RaRa]{RamRam} A. Ram J.~Ramagge, Affine Hecke algebras, cyclotomic
Hecke algebras and Clifford theory, Birkh\"auser, Trends in Math. (2003), 428--466.
\bibitem[Ree1]{ReeWhittaker} M. Reeder, Whittaker functions, prehomogeneous vector 
spaces and standard representations of $p$-adic groups,
J. reine angew. Math. {\bf 450} (1994), 83--121.
\bibitem[Ree2]{R} M.~Reeder, Isogenies of Hecke algebras and a Langlands correspondence
for ramified principal series representations,  Represent. Theory {\bf
6} (2002), 101--126.
\bibitem[Roc]{Roc} A. Roche, Types and Hecke algebras for principal
series representations of split reductive $p$-adic groups, Ann.
scient. \'Ec. Norm. Sup. {\bf 31} (1998), 361--413.
\bibitem[Sho]{Shoji} T.~Shoji, Green functions of reductive groups over a
finite field, PSPM {\bf 47} (1987), Amer. Math. Soc., 289--301.
\bibitem[Sol]{Sol} M.~Solleveld, On the classification of irreducible
representations of affine Hecke algebras with unequal parameters,
Represent. Theory {\bf 16} (2012), 1--87. 
\bibitem[SpSt]{SpringerSteinberg} T.~Springer, R.~Steinberg, Conjugacy
classes, pp. 167--266 in: Lecture Notes in Math. {\bf 131}, Springer-Verlag, 
Berlin, 1970.
\bibitem[Ste1]{Ste1965} R. Steinberg,
Regular elements of semisimple algebraic groups,
Publ. Math. Inst. Hautes \'Etudes Sci. \textbf{25} (1965), 49--80.
\bibitem[Ste2]{Steinbergtorsion} R.~Steinberg, Torsion in reductive groups, 
Adv. Math. {\bf 15} (1975), 63--92. 
\bibitem[Wal]{Wal} J.-L. Waldspurger, La formule de Plancherel pour les groupes 
$p$-adiques (d'apr\`es Harish-Chandra),
J. Inst. Math. Jussieu {\bf 2.2} (2003), 235--333.

\end{thebibliography}
\end{document}